\documentclass[12pt]{amsart}
\usepackage{latexsym}
\usepackage{amsmath}
\usepackage{amssymb}
\usepackage{amscd}
\usepackage{amsthm}
\usepackage{xypic}
\xyoption{curve}
\usepackage{ifthen} 
\usepackage{hyperref} 
\usepackage{graphicx}
\theoremstyle{plain}
 \newtheorem{theorem}{Theorem}[section]
 \newtheorem{proposition}[theorem]{Proposition}
 \newtheorem{lemma}[theorem]{Lemma}

\theoremstyle{definition}
 \newtheorem{definition}[theorem]{Definition}

\theoremstyle{remark}
 \newtheorem{remark}[theorem]{Remark}

\makeindex
\begin{document}
\title[Levels of stable rationality]{Linear bounds for levels of stable rationality}
\author[Bogomolov]{Fedor Bogomolov}
\address{F. Bogomolov, Courant Institute of Mathematical Sciences\\
251 Mercer St.\\
New York, NY 10012, U.S.A., \emph{and}}
\address{Laboratory of Algebraic Geometry, GU-HSE\\
7 Vavilova Str.\\
Moscow, Russia, 117312}
\email{bogomolo@courant.nyu.edu}
\author[B\"ohning]{Christian B\"ohning}
\address{Christian B\"ohning, Fachbereich Mathematik der Universit\"at Hamburg\\
Bundesstra\ss e 55\\
20146 Hamburg, Germany, \emph{and}}
\address{Mathematisches Institut der Georg-August-Universit\"at G\"ottingen\\
Bunsenstr. 3-5\\
37073 G\"ottingen, Germany}
\email{christian.boehning@math.uni-hamburg.de}
\email{boehning@uni-math.gwdg.de}
\author[Graf von Bothmer]{Hans-Christian Graf von Bothmer}
\address{Hans-Christian Graf von Bothmer, Mathematisches Institut der Georg-August-Universit\"at G\"ottingen\\
Bunsenstr. 3-5\\
37073 G\"ottingen, Germany}
\email{bothmer@uni-math.gwdg.de}

\thanks{Second and third author supported by the German Research Foundation
(Deutsche Forschungsgemeinschaft (DFG)) through
the Institutional Strategy of the University of G\"ottingen}

\maketitle
\newcommand{\PP}{\mathbb{P}} 
\newcommand{\QQ}{\mathbb{Q}} 
\newcommand{\ZZ}{\mathbb{Z}} 
\newcommand{\CC}{\mathbb{C}} 
\newcommand{\rmprec}{\wp}
\newcommand{\rmconst}{\mathrm{const}}
\newcommand{\xycenter}[1]{\begin{center}\mbox{\xymatrix{#1}}\end{center}} 
\newboolean{xlabels} 
\newcommand{\xlabel}[1]{ 
                        \label{#1} 
                        \ifthenelse{\boolean{xlabels}} 
                                   {\marginpar[\hfill{\tiny #1}]{{\tiny #1}}} 
                                   {} 
                       } 
\setboolean{xlabels}{false} 

\

\begin{abstract}
Let $G$ be one of the groups $\mathrm{SL}_n (\CC )$, $\mathrm{Sp}_{2n} (\CC )$, $\mathrm{SO}_m (\CC )$, $\mathrm{O}_m (\CC )$, or $G_2$. For a generically free $G$-representation $V$, we say that $N$ is a level of stable rationality for $V/G$ if $V/G \times \PP^N$ is rational. In this paper we improve known bounds for the levels of stable rationality for the quotients $V/G$. In particular, their growth as functions of the rank of the group is linear for $G$ one of the classical groups.\end{abstract}

\section{Introduction} \xlabel{sIntroduction}
In the birational geometry of algebraic varieties, an important problem consists in determining the birational types of quotient spaces $V/G$ where $V$ is a generically free linear representation of the linear algebraic group $G$, both defined over $\CC$, which will be our base field. We will suppose in the sequel that $G$ is connected. The quotient $V/G$ is said to be stably rational \emph{of level} $N$ if $V/G \times \PP^N$ is rational.  Whether or not $V/G$ is stably rational (of some level), is a property of the group $G$ and not of the particular generically free representation $V$ by the no-name lemma \cite{Bo-Ka}. \\
It will be desirable to obtain good bounds on levels of stable rationality $N$ for $V/G$ as above, for a given $G$ and as large as possible a class $\mathcal{C}$ of generically free $G$-representations $V$. By this we generally understand that  one wants to determine an explicit function $N = N( r, \: V)$ of $r$ and $V$ (in the given class $\mathcal{C}$) such that $V/G \times \PP^{N}$ is rational, and  $N(r, \: V)$ is small, where $r$ is the rank of $G$. More precisely, if $G$ is a group running through one of the infinite series of simple groups of type $A_r$, $B_r$, $C_r$, $D_r$ (of some fixed isogeny type), one would like to determine a function $N = N(r)$ which gives a level of stable rationality for generically free $G$-quotients $V/G$, uniformly for all $V$ in some fixed large class $\mathcal{C}$, and such that the asymptotic behaviour of $N(r)$ is $N(r) = O(r)$ (Landau symbol).  This is what is meant by ``linear bounds" in the title. For exceptional groups of types $G_2$, $F_4$, $E_6$, $E_7$ and $E_8$ one would like to find a small constant $N$ that gives a level of stable rationality. One should mention that these types of questions become rather simpler for groups with a nontrivial radical, cf. \cite{BBB10}, so that the semisimple case is fundamental. Let us also mention that results of this sort of course have applications to the rationality question for algebraic varieties because many varieties $X$ can be fibred over generically free linear group quotients $V/G$, with rational general fibre, so that the total space is birational to the product of base and fibre.

\

We will obtain $N(r) = O(r)$ for the groups $\mathrm{SL}_n (\CC )$, $\mathrm{Sp}_{2n} (\CC )$, $\mathrm{SO}_{m} (\CC )$, and also the nonconnected groups $\mathrm{O}_m (\CC )$, for large classes of representations $\mathcal{C}$. Moreover, we will improve the bound for $G_2$ somewhat. Before describing the results in more detail, we mention that previously one had only $N(r) = O(r^2)$ for these classical groups. To be precise, what was known previously, at least to us, can be summarized in the following table (the class $\mathcal{C}$ this applies to is the class  of all generically free $G$-representations):

\

\begin{center}
\begin{tabular}{| c |  c |}\hline
Group $G$  & Level of stable rationality $N$ \\  \hline\hline
$\mathrm{SL}_n (\CC )$ &  $n^2 -1$ \\ \hline
$\mathrm{SO}_{2n+1} (\CC )$ &  $ 2n^2 + 3n +1$ \\ \hline
$\mathrm{Sp}_{2n} (\CC )$ & $2n^2 + n$\\ \hline
$\mathrm{SO}_{2n} (\CC )$ & $ 2n^2 +n$  \\ \hline
$G_2$  &   $17$ \\ \hline
\end{tabular}  
\end{center}

\

Stable rationality was also known for the orthogonal groups $\mathrm{O}_{2n} (\CC )$ and $\mathrm{O}_{2n+1} (\CC )$, with the same levels as for the special orthogonal groups, and for the simply connected exceptional groups $F_4$, $E_6$, $E_7$, cf. \cite{Bogo86}. We point out that stable rationality remains open for the Spin-groups, and for $E_8$. Our methods here do not seem to yield substantial improvements for $F_4$, $E_6$, $E_7$. Let us comment briefly on how the results in the table are obtained: $\mathrm{SL}_n (\CC )$ and $\mathrm{Sp}_{2n} (\CC )$ are special groups (every \'{e}tale locally trivial principal bundle for them is Zariski locally trivial), so for a generically free representation $V$ the quotient is stably rational of level their dimensions. For $\mathrm{SO}_m (\CC )$ one considers the action on a variety $X$ which is birational to a tower of equivariant vector bundles over an $\mathrm{SO}_m (\CC )$-representation as base: $X$ consists of orthogonal $m$-frames $(v_1, \dots , v_m)$ in $\CC^m$ with $\langle v_i , v_i \rangle \neq 0$. Note that $\dim X = \dim \mathrm{SO}_m (\CC ) + m$, which is the value given in the table. In fact, $V/\mathrm{SO}_m (\CC ) \times X = (V \times X) /\mathrm{SO}_m (\CC )$ by the no-name lemma, and in $X$ there is a $(\mathrm{SO}_m (\CC ), \; H)$-section $\Pi$ where $H$ is an elementary abelian $2$-group and $\Pi$ is a product of generic lines. So $(V \times X) /\mathrm{SO}_m (\CC ) = (V/H) \times \Pi $ is rational. For $G_2$ the argument is similar, but for $X$ one takes instead 
\[
X = \{ (A, B, C) \, \mid \, A \perp B, \; A, B, AB \perp C \; \mathrm{and}\; A,B,C \; \mathrm{of}\; \mathrm{nonzero} \; \mathrm{norm} \} 
\]
as a subset of $\mathbb{O}^3$, where $\mathbb{O}$ are traceless octonions. The action of $G_2$ is free on $X$. See \cite{Bogo86} for details.

\

Let us now describe our results in more detail. Let $G$ be one of the groups $\mathrm{SL}_n (\CC )$, $\mathrm{Sp}_{2n} (\CC )$, $\mathrm{SO}_{m} (\CC )$, $\mathrm{O}_m (\CC )$ or $G_2$. We will first narrow down the class of representations $\mathcal{C}$ of these groups which we will consider and make a statement about. Namely, $\mathcal{C}$ contains precisely the $G$-representations $V$ of the form $V = W \oplus S^{\epsilon}$, where: $W$ is an irreducible representation of $G$ whose ineffectivity kernel (a finite central subgroup) coincides with the stabilizer in general position. $S$ is a standard representation for each of the groups involved, namely $\CC^n$ for $\mathrm{SL}_n (\CC )$, $\CC^{2n}$ for $\mathrm{Sp}_{2n} (\CC )$, $\CC^m$ for $\mathrm{SO}_m (\CC )$ and $\mathrm{O}_m (\CC )$, $\CC^7$ for $G_2$. Here $\epsilon \in \{ 0, 1\}$, and $\epsilon =0$ if and only if $W$ is already $G$-generically free. Thus $V$ will always be $G$-generically free. The following table summarizes our main results Theorems \ref{tSLIrreducible}, \ref{tSp}, \ref{tO}, \ref{tSO}, \ref{tG2}.

\

\begin{center}
\begin{tabular}{| c |  c |}\hline
Group $G$  & Level of stable rationality $N$ \\ 
                      & for $V/G$ for $V\in\mathcal{C}$ \\  \hline\hline
$\mathrm{SL}_n (\CC )$, $n\ge 1$ &  $n$ \\ \hline
$\mathrm{SO}_{2n+1} (\CC )$, $n\ge 2$ &  $ 2n+1$ \\ \hline
$\mathrm{O}_{2n+1} (\CC )$, $n\ge 2$ &  $2n+1$ \\ \hline
$\mathrm{Sp}_{2n} (\CC ), \; n\ge 4$ & $2n$\\ \hline
$\mathrm{SO}_{2n} (\CC )$, $n\ge 2$ & $ 2n$ except for $W$ with highest weight\\
                                                  & $c\omega_{n-1} $ or $c\omega_n$, $c\in \mathbb{N}$,\\
                                                  & where we know only $4n$.\\ \hline
$\mathrm{O}_{2n} (\CC )$, $n\ge 2$ &    $2n$    \\ \hline
$G_2$  &   $7$ \\ \hline
\end{tabular}  
\end{center}

\

Here $\omega_{n-1}$ and $\omega_n$ denote the fundamental weights corresponding to the last two nodes of the Dynkin diagram of $\mathfrak{so}_{2n} (\CC )$. 
Note that confining ourselves to the class $\mathcal{C}$ is not very restrictive: any generically free $G$-representation containing a representation in $\mathcal{C}$ as a summand will have a quotient which is stably rational of the same level. By a method involving Severi Brauer varieties (exemplified in Proposition \ref{pInductionStepSL}) one could also get the result for any generically free $G$- representation containing an irreducible summand $W$ whose stabilizer in general position coincides with the ineffectivity kernel. Then it would remain to analyze those representations $V$ which are sums of irreducible representations with a nontrivial stabilizer in general position (there is a finite list for each fixed rank), but such that $V$ is generically free. As we felt that this would have made the paper much longer without adding too much content, we have not done this.\\
The basic strategy for the proof of our result is not difficult: for $V$ in the class $\mathcal{C}$, we consider in each of the above cases $V \oplus S$. The $G$-action on $V\oplus S$ has a section given by the normalizer $G'$ of the stabilizer in general position in $S$. This $G'$ is a group of smaller rank, and in general, $V$ will decompose with respect to the semisimple part of $G'$ into several summands (however, $G'$ is often nonreductive). Thus we may hope to prove rationality of $V/G'$, either by induction on the rank, or because $V$ becomes highly reducible as module for the semisimple part of $G'$. However, to carry this program to the end, involves a long (but joyful) ride through the complete representation theory of the groups involved. 

\

Here is a short road map of the paper. The proof for $\mathrm{SL}_n (\CC )$ is an induction on the rank. As both the induction step and induction base are somewhat lengthy, we have chosen to separate them; the former is in section \ref{sInductionProcessSL}, the latter in section \ref{sInductionBaseSL}. Section \ref{sInductionProcessSp} contains the proof for the symplectic group. It is not an induction, but what is used is that under the restriction $\mathrm{Sp}_{2n} (\CC ) \downarrow \mathrm{Sp}_{2n-2} (\CC )$, representations of $\mathrm{Sp}_{2n} (\CC )$ become highly reducible. One may then produce a generically free $G'$-quotient of $V$.  Section \ref{sOrthogonalGroups} treats the case of the orthogonal and special orthogonal groups. We need a technical result in this section on stabilizers in general position for the group $\mathrm{O}_{2n} (\CC )$. We put this into an appendix, Appendix  \ref{AppendixOrthogonalGroup}. Finally, the case of $G_2$ is covered in  section \ref{sG2}. This case is the most straight-forward of all of them. 

\

Let us mention that our results for $\mathrm{Sp}_{2n} (\CC )$ and $G_2$ hold also over $\mathbb{Q}$, but for the other groups we do not want to assert this as we are not sure if the Severi-Brauer method used in their proofs goes through.

\

Moreover, the results in the present article provide tools for proving rationality for generically free quotients of linear representations by groups with nontrivial unipotent radicals, a program which was begun in the article \cite{BBB10}. 

\

\textbf{Acknowledgments.} During the work the first author was partially supported
by NSF grant DMS-1001662 and by AG Laboratory GU-HSE grant RF government ag. 11 11.G34.31.0023.  
The second and third authors were supported by the German Research Foundation 
(Deutsche Forschungsgemeinschaft (DFG)) through 
the Institutional Strategy of the University of G\"ottingen.  
The second author wants to thank the University of Vienna for hospitality during a stay in the summer of 2010 where this work began to take shape.

\section{The induction process for $\mathrm{SL}_n (\CC )$}\xlabel{sInductionProcessSL}
First we note the following simple but useful result.

\begin{lemma}\xlabel{lOppositeParabolic}
Let $V$ be a generically free representation for a linear algebraic group $G$, $R$ a $G$-representation such that $R$ and the
dual 
$R^{\vee }$ have a dense
$G$-orbit. If $P$ resp. $P'$ are the stabilizers of a generic point in $R$ resp. $R^{\vee }$, then $V/P$ and $V/P'$ are
birationally equivalent.
\end{lemma}

\begin{proof}
By the no-name lemma \cite{Bo-Ka} $(V \oplus R)/G$ is birational to $(V/G) \oplus R$. The latter is birational to $(V/G) \oplus 
R^{\vee }$, hence to $(V\oplus R^{\vee })/G$. On the other hand, $V$ is a $(G, P)$-section in $V\oplus R$, and a $(G,
P')$-section in $V\oplus R^{\vee }$.
\end{proof}

Let $V$ be an irreducible $\mathrm{SL}_n (\CC )$-representation. We have
\[
V\simeq \Sigma^{\lambda } \CC^n
\]
for some $n$-tuple of integers (the highest weight of $V$) $\lambda = (\lambda_1 , \dots , \lambda_n)$, $\lambda_1\ge \dots \ge
\lambda_n\ge 0 $ (where according to our convention e.g. $\Sigma^{1,1, 0, \dots , 0}\CC^n = \Lambda^2 \CC^n$). Let $e_1, \dots
, e_n$ be the standard basis of $\CC^n$, $e_1^{\vee } , \dots , e_n^{\vee }$ the dual basis in $(\CC^n)^{\vee }$. The
stabilizer $P$ of $e_1$ in $\mathrm{SL}_n (\CC )$ is a group isomorphic to the group $
\mathrm{SL}_{n-1} (\CC )\ltimes (\CC^{n-1})^{\vee }$. According to a classical branching law (\cite{G-W}, Chapter 8), the
representation
$V$ decomposes as a $\mathrm{SL}_{n-1} (\CC )$-representation as
\[
V = \bigoplus_{\mu } \Sigma^{\mu} \CC^{n-1}
\]
where $\mu = (\mu_1, \dots , \mu_{n-1})$ is a tuple of integers which \emph{interlaces} $\lambda$ which means
\[
\lambda_1 \ge \mu_1 \ge \lambda_2 \ge \dots \ge \mu_{n-1} \ge \lambda_n \, ,
\]
and the branching is multiplicity-free. As $P$-representation $V$ has a filtration
\begin{gather}\label{fFiltration}
0 \subset V_{1} \subset \dots \subset V_{l} = V
\end{gather}
with $Q_i:=V_i/V_{i-1}$ the maximal completely reducible subrepresentation in $V/V_{i-1}$, $i=1, \dots , l$,
 as in \cite{BBB10}, section 3, where we put $V_0 =0$. Here 
\[
l = (\lambda_1 + \dots + \lambda_{n-1}) - (\lambda_2 + \dots + \lambda_n) +1 = \lambda_1 - \lambda_n +1
\]
and $Q_i$ is the sum of those $V_{\mu}$, $\mu$ interlacing $\lambda$, with
\[
\sum_{j=1}^{n-1} \mu_j = (\lambda_2 + \dots + \lambda_n) + i - 1\, .
\]
In particular, $Q_1$ and $Q_l$ are irreducible. We may also consider the stabilizer $P'$ inside $\mathrm{SL}_n (\CC )$ of the
point $e_n^{\vee }$ in $(\CC^n)^{\vee }$ which is isomorphic to $\mathrm{ASL}_{n-1}(\CC ) = \mathrm{SL}_{n-1} (\CC )\ltimes
\CC^{n-1}$. The semisimplification of $V$ as $P'$-module is of course the same as that as $P$-module, but the
associated filtration is reversed. A remark that will be sometimes useful below is
\begin{remark}\xlabel{rSpan}
The $P$-span of $Q_l$ is $V$. A similar fact holds for $P'$. This is because otherwise the maximal completely reducible
submodule of $V^{\vee }$ would be reducible which is not the case as the $P$-module $V^{\vee }$ is the restriction to $P$ of
the $\mathrm{SL}_n (\CC )$-module $V^{\vee }$ which has a unique line of highest weight vectors if $V$ is irreducible.
\end{remark}

We propose to prove first on our way to the main result

\begin{theorem}\xlabel{tSLIrreducible}
For any $n$ and any generically free irreducible $\mathrm{SL}_n (\CC )$-representation $V$, $V/\mathrm{SL}_n (\CC )$ is stably
rational of level
$n$. If a central quotient $\mathrm{SL}_n (\CC )/(\ZZ / m \ZZ )$, $m \mid n$, $m \neq 1$ acts generically freely in $V$, then
$(V\oplus (\CC^n)^{\vee })/\mathrm{SL}_n (\CC )$ is stably rational of level $n$. 
\end{theorem}

To prove this it suffices of course to establish rationality of $V/P$ or $V/P'$, resp. $(V\oplus (\CC^n)^{\vee })/P$ or
$(V\oplus (\CC^n)^{\vee })/P'$.  This will be proven by induction, and the main work will go into establishing the \emph{base
of the induction}, which is done by a detailed analysis of the irreducible $\mathrm{SL}_n (\CC )$-representations with a
nontrivial stabilizer in general position for small values of $n$. The induction step is easier, and we do this in form of 
Proposition \ref{pInductionStepSL} below, the proof of which occupies the present section. First a convenient 

\begin{definition}\xlabel{dreRepresentations}
We call an $\mathrm{SL}_n (\CC )$-representation $V$ an \emph{R-representation} (R for \emph{regular}) if a central quotient
$\mathrm{SL}_n (\CC )/(\ZZ / m \ZZ )$ acts generically freely, and an \emph{E-representation} otherwise (E for
\emph{exceptional}). Thus an E-representation is one with a notrivial stabilizer in general position or $\CC$.
\end{definition}

\begin{proposition}\xlabel{pInductionStepSL}
Suppose Theorem \ref{tSLIrreducible} is true for any $n$ and any irreducible R-representation $V$ of
$\mathrm{SL}_n (\CC )$ in the set of $V$'s satisfying (A) or (B) below: 
\begin{itemize}
\item[(A)]
If $V=\Sigma^{\lambda } \CC^n$ is as above, $\lambda =
(\lambda_1 , \dots , \lambda_n )$, and 
\begin{gather*}
\mu_0 := (\lambda_1, \dots , \lambda_{n-1})\, ,\\
\mu_1 := (\lambda_2, \dots , \lambda_n)
\end{gather*}  
then $\Sigma^{\mu_0} \CC^{n-1}$ and $\Sigma^{\mu_1} \CC^{n-1}$  are both E-representations for $\mathrm{SL}_{n-1} (\CC )$, i.e.
are irreducible representations of $\mathrm{SL}_{n-1} (\CC )$ with a nontrivial stabilizer in general position.
\item[(B)]
$V$ is an exterior power $\Lambda^l \CC^n$. 
\end{itemize}
Then Theorem
\ref{tSLIrreducible} is true in general. 
\end{proposition}

\begin{proof}
It suffices to show the implication: 
\begin{quote}
($\sharp$) If $V=\Sigma^{\lambda } \CC^n$ is an irreducible R-representation which is not an exterior power 
such that $\mu_0$ \emph{or} $\mu_1$ also give an R-representation $W=\Sigma^{\mu_0} \CC^{n-1}$ \emph{or} $W=\Sigma^{\mu_1}
\CC^{n-1}$ for $\mathrm{SL}_{n-1} (\CC )$, then the assertion of Theorem \ref{tSLIrreducible} for $W$ implies the assertion of
Theorem \ref{tSLIrreducible} for
$V$.
\end{quote}
Informally, we
can reduce the dimension and use induction (note: Theorem \ref{tSLIrreducible} is trivially true for $\mathrm{SL}_1 (\CC ) =
(1)$), unless we reach a point where we must go from an R-representation to an E-representation.

\

To prove ($\sharp$) we will prove: under the hypotheses in ($\sharp$) the assertion of Theorem \ref{tSLIrreducible} for $W$
implies rationality for $V/P$ or $V/P'$. This implies clearly the assertion of Theorem \ref{tSLIrreducible} for $V$.\\
We will consider the case where $\mu_0$ gives an R-representation first.\\
Let $V$, $W$ be as previously; restrict $V$ to $P$ and write the associated
filtration as in formula
\ref{fFiltration}. Then
\[
Q_l = W = \Sigma^{\mu_0} \CC^{n-1}
\]
and this is an R-representation by assumption. Moreover $l$ is at least $3$ since we excluded exterior powers, so it makes
sense to talk about $Q_{l-1}$ and $Q_{l-2}$. We state:
\begin{quote}
\textbf{Hypothesis (H):} $\dim Q_{l-1} \ge 2(n-1)$ and $\dim Q_{l-2} \ge n-1$.
\end{quote}
If this hypothesis is true, which we assume for the moment, then $V/P$ is rational: we apply the Severi-Brauer method explained
in the proof of
\cite{BBB10}, Proposition 4.4. By this method, the function field of $V/P$ arises from the function field $K$ of the
Severi-Brauer scheme $S$  associated to the two-step extension
\[
0 \to Q_{l-1} \to V/V_{l-2} \to Q_l \to 0
\] 
obtained by dividing out homotheties in the fibres of the vector bundle $(V/V_{l-2})/U \to Q_l$, i.e.
forming the associated projective bundle, and then dividing by the $\mathrm{SL}_{n-1} (\CC )$-action. In fact, $V/P$ is a vector
bundle over
$(V/V_{l-2})/P$ which in turn is a Zariski locally trivial $\CC^{\ast }$-bundle over $S$. Thus the function field of $V/P$
arises by adjoining $\ge 1 + (n-1) =n$ indeterminates to the function field $K$ of $S$ by the inequality on the dimension of
$Q_{l-2}$. But if we consider
$S'$, the Severi-Brauer scheme 
$\PP ( Q_{l-1}
\oplus
\CC^{n-1}) /\mathrm{SL}_{n-1} (\CC )$ over the base
$Q_{l-1}/\mathrm{SL}_{n-1} (\CC )$, then $S$ and $S'$ are stably equivalent, so by the assumption on the dimension of
$Q_{l-1}$, $S$ is generically a vector bundle over some $\PP^{n-2}$-bundle which is in the same subgroup as $S'$. But the
assertion of Theorem \ref{tSLIrreducible} which we assumed for $W$, implies stable rationality of level $n=(n-1)+1$ for $S'$ and
hence also for the afore-mentioned $\PP^{n-2}$-bundle in the same subgroup of the Brauer group of the base, since the two are
stably isomorphic of level $n-2$. Hence $V/P$ is rational. This implies the assertion of Theorem \ref{tSLIrreducible} for $V$.

\

If $\mu_1$ gives an R-representation instead of $\mu_0$ we may argue in the same way as above, replacing $P$ by $P'$ and
$\mu_0$ by $\mu_1$ throughout, assuming Hypothesis (H) now for the filtration with respect to $P'$.

\

What remains to be done is to analyze in which cases Hypothesis (H) may fail (for $P$ or $P'$). Let us put $m=n-1$. Suppose
an irreducible representation of
$\mathrm{SL}_{m} (\CC )$ has dimension $< 2m$. Since $\dim \mathrm{SL}_{m} (\CC ) = m^2 -1 \ge 2m$, as soon as $n\ge 3$, such a
representation is an E-representation. 
Looking at the table of \cite{Po-Vi}, we find that the only possibilities are:
\begin{itemize}
\item[(1)]
One of $\CC$, $\CC^2$, $\mathrm{Sym}^2 \CC^2$, $\mathrm{Sym}^3 \CC^2$, $\mathrm{Sym}^4 \CC^2$ for $\mathrm{SL}_2 (\CC )$.
\item[(2)]
$\CC$, $\CC^m$ or the dual $(\CC^m)^{\vee }$ for all $m\ge 3$. 
\item[(3)]
$\Lambda^2 \CC^4$ for $\mathrm{SL}_4 (\CC )$. 
\end{itemize}
Moreover, the irreducible $\mathrm{SL}_m (\CC )$-representations of dimension $< m$ are then just $\CC$ for all $m$ and in
addition $\CC^2$ for $\mathrm{SL}_2 (\CC )$.\\
Now note that $Q_l$, being irreducible, must be contained in one of the representations listed in (1), (2), (3) above if we
tensor these with $\CC^m$ (in the case where we argue with $P$) or with $(\CC^m)^{\vee }$ (in the case where we argue with
$P'$) and must be an R-representation. This excludes the possibilities of
$\CC$,
$\CC^m$ or
$(\CC^m)^{\vee}$ at once, as well as $\mathrm{Sym}^2 \CC^2$, $\mathrm{Sym}^3 \CC^2$. It is true that 
\[
\mathrm{Sym}^5 \CC^2 \subset \mathrm{Sym}^4 \CC^2 \otimes \CC^2
\]
and
\[
\Sigma^{2,1} \CC^4 \subset \Lambda^2 \CC^4 \otimes \CC^4, \quad \Sigma^{2,1} (\CC^4)^{\vee } \subset \Lambda^2 (\CC^4)
\otimes
(\CC^4)^{\vee }
\]
are the only R-representations occurring. Thus Hypothesis (H) may fail \emph{a priori} when 
\begin{itemize}
\item[(a)]
restricting $V=\Sigma^{(5+j , j, 0)} (\CC^3)$, $j\ge 0$, to $P$. 
\item[(b)]
restricting $V=\Sigma^{(2+j, 1+j, j ,j, 0)} (\CC^5)$, $j\ge 0$, to $P$.
\item[(a')]
restricting $V=\Sigma^{(k, 5, 0)} (\CC^3)$, $k\ge 5$, to $P'$. 
\item[(b')]
or restricting $V=\Sigma^{(k, 2,2,1,0)} (\CC^5)$, $k\ge 2$, to $P'$.
\end{itemize}
An analysis of the above four cases shows that Hypothesis (H) actually only fails in the following situation:
If we restrict $\Sigma^{(5, 0, 0)} \CC^3$ to $P$ or $\Sigma^{(5,5,0)} \CC^3$ to $P'$. By duality, it suffices to treat the first
case, we will prove directly that 
\[
\Sigma^{(5,0,0 )} \CC^3 /P \; \mathrm{is} \; \mathrm{rational.}
\]
By Lemma \ref{lOppositeParabolic} it suffices to prove that $\Sigma^{(5,0,0)}\CC^3 /P'$ is rational. But 
\[
\mathrm{Sym}^5 \CC^2\oplus \mathrm{Sym}^4 \CC^2 \oplus \mathrm{Sym}^3 \CC^2 \oplus \mathrm{Sym}^2 \CC^2 \oplus \CC
\]
is a $(P', \: \mathrm{SL}_2 (\CC ))$-section in this space whence the assertion.

\

To conclude the proof it suffices to remark that, generally, $Q_{l-2}$ cannot be composed of
summands $\CC$ because this would mean again that
$Q_l$ is an E-representation. The representation $\CC^2$ is to be excluded for a similar reason for $\CC^2\otimes
\mathrm{Sym}^2 \CC^2$ contains only E-representations. Thus Hypothesis (H) cannot fail because the inequality on the dimension
of $Q_{l-2}$ fails. Thus Proposition \ref{pInductionStepSL} is proven.
\end{proof}

The first question left unanswered by Proposition \ref{pInductionStepSL} is what to do with exterior powers. This is addressed
in:

\begin{proposition}\xlabel{pExteriorPowers}
Theorem \ref{tSLIrreducible} is true for all exterior powers $V=\Lambda^i \CC^n$, all $n$, if it is true for the special
exterior powers
\[
\Lambda^3 \CC^{10}, \Lambda^3 (\CC^{10})^{\vee }, \Lambda^4 \CC^9, \: \Lambda^4 (\CC^9)^{\vee }\, .
\]
(We emphasize that the statement refers to exterior powers $\Lambda^i \CC^n$ which are R-representations).
\end{proposition}

\begin{proof}
If an exterior power $V= \Sigma^{(1,1,\dots , 1, 0, \dots , 0)} \CC^n = \Lambda^k \CC^n$ is an R-representation then certainly
$3\le k \le n-3$. It implies $n\ge 6$. In addition to the ones already excluded, the following exterior powers are
E-representations, according to the table in \cite{Po-Vi}:
\begin{gather*}
\Lambda^3 \CC^n, \: n=6,7,8,9, \quad  \Lambda^4 \CC^8 \quad \mathrm{or}\: \mathrm{one}\: \mathrm{of} \: \mathrm{the}\:
\mathrm{duals.}
\end{gather*}
It suffices again to prove that $V/P$ or $V/P'$ is rational. For $P$ we have for the quotients in the filtration
\[
Q_1 = \Lambda^{k-1} \CC^{n-1}, \; Q_2 = \Lambda^{k} \CC^{n-1}
\]
and for $P'$ one has $Q_1$ and $Q_2$ interchanged. Thus restricting to either $P$ or $P'$, one can always find an
R-representation as completely reducible quotient unless 
\[
V= \Lambda^3 \CC^{10}, \Lambda^4 \CC^9 \quad \mathrm{or}\: \mathrm{one}\: \mathrm{of} \: \mathrm{the}\:
\mathrm{duals.}
\]
(recall that $V$ itself is an R-representation). We exclude these cases for the moment. Then we may again use the Severi-Brauer
method already exemplified in the proof of Proposition \ref{pInductionStepSL} provided the fibre has dimension $\ge 3(n-1)$.
Putting $m=n-1$ we see that 
\[
\dim \Lambda^2 \CC^m = \frac{m(m-1)}{2} \ge 3m
\]
if $m\ge 7$, so that for $n\ge 8$ there is no problem with the induction. Since $3\le k \le n -3$, we find that for $n=6, 7$
there are no exterior powers which are R-representations. In fact, $n$ must be at least $9$ for that. Hence Proposition
\ref{pExteriorPowers} is proven.
\end{proof}

We may combine Propositions \ref{pInductionStepSL} and \ref{pExteriorPowers} to obtain:

\begin{proposition}\xlabel{pMainReduction}
Theorem \ref{tSLIrreducible} holds in general if it holds for representations $V$ in the following list:
\begin{gather*}
\Lambda^3 \CC^{10}, \; \Lambda^4 \CC^9 \quad \mathrm{or}\: \mathrm{one}\: \mathrm{of} \: \mathrm{the}\:
\mathrm{duals} 
\end{gather*}
\emph{or} $V=\Sigma^{\lambda} \CC^3$ and $\lambda$ is one of
\begin{gather*}
(a, 4, 0), \: a=4,5,6,7,8, \\
(b, 3, 0), \: b=4,5,6,7, \\
(c, 2, 0), \: c=3,4,5,6, \\
(d, 1, 0), \: d=3,4,5,\\
(4,0,0).
\end{gather*}
\end{proposition}

\begin{proof}
It remains to analyze which R-representations $V= \Sigma^{\lambda } \CC^n$ are such that both
\[
\mu_0 = (\lambda_1, \dots , \lambda_{n-1}), \quad \mu_1 = (\lambda_2, \dots , \lambda_n)
\]
give E-representations. We start by considering the E-representations which are there for all $n$:
\[
\CC , \: \CC^n, \: (\CC^n)^{\vee}, \: \mathrm{Sym}^2 (\CC^n), \: \mathrm{Sym}^2 (\CC^n)^{\vee }, \: \Lambda^2 (\CC^n), \:
(\Lambda^2 \CC^n)^{\vee }, \: (\mathrm{Ad} \CC^n)_0\, .
\]
They correspond to weights which may be written:
\begin{gather*}
(0,\dots , 0), \: (1, 0, \dots , 0) , \: (1,\dots , 1, 0), \: (2, 0,\dots , 0), \: (2, \dots, 2, 0), \\
(1,1,0,\dots , 0), \: (1, \dots , 1, 0, 0), \: (2,1,\dots , 1, 0) .
\end{gather*}
Suppose $\lambda$ is such that $\mu_0$ gives one of these representations. Then $\lambda$ can be written
\begin{gather*}
(j,\dots ,j, 0), \: (1+j, j, \dots ,j, 0) , \: (1+j,\dots , 1+j, j, 0), \: (2+j, j,\dots ,j, 0),\\ \: (2+j, \dots, 2+j, j, 0),\:
 (1+j,1+j,j,\dots ,j, 0), \: (1+j, \dots , 1+j, j, j,0),\\ \: (2+j,1+j,\dots , 1+j,j, 0) .
\end{gather*}
Suppose that $\mu_1$ also gives one of these representations. Then $j$ must be either $1$, $2$ or $0$. Under the hypothesis
that $V$ is an R-representations we conclude that the only exceptions for $n\ge 11$ giving both an exceptional $\mu_0$ and
$\mu_1$ are the representations associated to:
\begin{eqnarray}
(3,2,\dots , 2, 0),\\
(3,1,\dots , 1, 0), \\
(4,2,\dots , 2, 0),\\
(2,1,\dots , 1, 0, 0), \\
(2,2,1,\dots , 1, 0).
\end{eqnarray}
Here (2) and (3) are dual, (5) and (6) are dual, and (4) is self-dual. We prove rationality of (3), (4), (5) modulo $P'$, the
proof of the other cases being analogous using Lemma \ref{lOppositeParabolic}. The proof for (4) is easy because here $V/V_1$
is generically free for $P'$, and $V_1$ has dimension larger than the dimension of the affine group $P'$ which is special.\\ The
decomposition of (3) has (putting $m=n-1$)
\begin{gather*}
Q_1= \mathrm{Sym}^2 \CC^m, \: Q_2 = \CC^m + \Sigma^{3,1,\dots , 1, 0} \CC^m , \: Q_3 = \CC + (\mathrm{Ad}\CC^m)_0 \\
Q_4 = (\CC^m)^{\vee }\, .
\end{gather*}
Note that this $V$ has a $P'$-subrepresentation (of dimension larger than $P'$) with semisimplification $\mathrm{Sym}^2 \CC^m +
\Sigma^{3,1,\dots , 1, 0} \CC^m$, and the quotient representation is the restriction of $(\mathrm{Ad} \CC^n)_0$ to $P'$, which
is generically $P'$-free, since $(\mathrm{Ad} \CC^n)_0 +\CC^n$ is generically $\mathrm{SL}_n (\CC )$-free. We conclude as
in the previous case.\\
The decomposition for (5) is 
\begin{gather*}
Q_1= (\mathrm{Ad}\CC^m)_0, \: Q_2 = (\CC^m)^{\vee } + \Sigma^{2,1,\dots , 1, 0, 0}\CC^m, \: Q_3= \Lambda^2 (\CC^m)^{\vee }\, .
\end{gather*}
This has a subrepresentation with semisimplification $(\mathrm{Ad}\CC^m)_0 + (\CC^m)^{\vee }$ which has dimension exactly the
dimension of $P'$. The quotient representation is generically free for $P'$, and we conclude as in the previous cases.

\

Thus we conclude that Theorem \ref{tSLIrreducible} is true once it is true for the exceptions listed in Proposition
\ref{pExteriorPowers} and those R-representations $\Sigma^{\lambda }\CC^n$ with $\mu_0$ and $\mu_1$ both giving
E-representations \emph{and} 
\[
n \le 10\, .
\]
Except for the eight exceptional representations already listed above, there are just the following additional ones for $m
=n-1\le 9$ according to the table in \cite{Po-Vi}:
\begin{gather*}
\Lambda^3 \CC^m, \: m=6,7,8,9, \quad  \Lambda^4 \CC^8 ,\quad  \Sigma^{2,2} \CC^4 ( \simeq S^2_0 \mathrm{SO}_6)\\
\mathrm{Sym}^3 \CC^2 , \; \mathrm{Sym}^4 \CC^2, \; \mathrm{Sym}^3 \CC^3\\
\quad \mathrm{or}\: \mathrm{one}\: \mathrm{of} \: \mathrm{the}\:
\mathrm{duals.}
\end{gather*}
We first analyze when for $n=4$ a representation has $\mu_0$ or $\mu_1$ one of $\mathrm{Sym}^3 \CC^3$ or
$\mathrm{Sym}^3( \CC^3)^{\vee }$. This can happen only for
\begin{gather*}
(3,3,3,0), \quad (4,3,3,0),\quad (6, 3, 3, 0),  \quad (5,3,3,0), \quad (3, 3, 0, 0)
\end{gather*}
or one of the duals. In the first case (the one of $(3,3,3,0)$) we can reduce to a proof of rationality for
\[
\mathrm{Sym}^3 (\CC^3)^{\vee} + \mathrm{Sym}^2 (\CC^3)^{\vee} + \CC
\]
hence to a question of rationality for $\mathrm{O}_3 (\CC ) =\ZZ /2\ZZ \times \mathrm{PSL}_2 (\CC )$ which always has a positive
answer.\\
In the cases of $(4,3,3,0)$, $(6, 3, 3, 0)$ and $(5,3,3,0)$ we get -restricting to $P'$- a $P'$-quotient representation of the given representation
which is isomorphic to the restriction to $P'$ of the $\mathrm{SL}_4 (\CC )$-representation $\mathrm{Sym}^3 (\CC^4)^{\vee }$.
Hence this $P'$-quotient is generically free for $P'$, and the fibre dimension is larger than that of $P'$ ($=11$) in all the cases
(the fibre contains $V(1,3)$, $V(3, 3)$, resp. $V(2,3)$).

\

In the case $(3, 3, 0, 0)$ we get -restricting to $P$- a filtration with
\begin{gather*}
Q_1= \mathrm{Sym}^{3}\CC^3, \quad Q_2= \Sigma^{3,1,0}\CC^3, \quad Q_3= \Sigma^{3,2,0}\CC^3,\quad  Q_4= \mathrm{Sym}^{3}(\CC^3 )^{\vee }
\end{gather*}
It will be sufficient to show that the quotient $P$-module with filtration steps $Q_3$ and $Q_4$ is generically free. The following $2$-dimensional subspace gives a section for the action of $\mathrm{SL}_3 (\CC )$ on $\mathrm{Sym}^3 (\CC^3)^{\vee }$ (the \emph{Hesse pencil})
\[
a(x_0^3 + x_1^3 + x_2^3) + b x_0x_1x_2 =0 \, ,
\]
and the stabilizer in general position for $\mathrm{Sym}^{3}(\CC^3 )^{\vee }$ in $\mathrm{SL}_3 (\CC )$ is conjugate to the subgroup $H$ which consists of monomial matrices in $\mathrm{SL}_3 (\CC )$ whose nonzero entries are cube roots of unity. $H$ is an extension of the group $\ZZ/3\ZZ$ of cyclic permutations of the coordinates $x_0, \: x_1, \: x_2$ by the group $( \ZZ /3\ZZ )^2$ of diagonal matrices with entries cube roots of unity. We have to see how the stabilizer subgroup $H \ltimes \CC^3$ of a general point $q\in V(0, 3)$ acts on $Q_3 \oplus (\CC\cdot q)$: it will be sufficient to know the weight space decomposition of $Q_3$.  In terms of the standard coordinate functions $\epsilon_1$, $\epsilon_2$, $\epsilon_3$ on the toral Lie algebra $\mathfrak{t}$, $\epsilon_i (\mathrm{diag}(t_1, t_2 , t_3)) = t_i$, one has that the weights of $V(1, 2)$ are (cf. \cite{Fu-Ha}, p. 180 ff.) 
\begin{gather*}
-\epsilon_i -\epsilon_j +\epsilon_k ,\; \{ i,j,k\} = \{1,2,3\}  \quad    -2 \epsilon_i + \epsilon_j , \; i\neq j  \quad ( \mathrm{multiplicity} \; 1 )\\
 \epsilon_i  \quad ( \mathrm{multiplicity} \; 2 )\, .
\end{gather*}
To conclude the proof that the $P$-module with filtration $Q_3$, $Q_4$ is generically free, it remains to note that $H$ acts generically freely on 
\[
\bigoplus V_{-\epsilon_i -\epsilon_j +\epsilon_k}  + \bigoplus V_{ -2 \epsilon_i + \epsilon_j}\, .
\]

\

A final peculiarity of the case $n=4$ is that we can have a representation such that both $\mu_0$ and $\mu_1$ give the adjoint representation of $\mathrm{SL}_3 (\CC )$; this happens for
\[
(3, 2, 1, 0)\, .
\]
Here the filtration with respect to $P$ has (we just write the exponent $(a, b, c)$ in place of $\Sigma^{a,b,c} \CC^3$ to save space)
\begin{gather*}
Q_1= (2,1,0), \quad Q_2= (1,0,0) + (2,2,0) + (3,1,0), \\ Q_3= (1,1,0)+(2,0,0)+ (3,2,0), \quad Q_4 = (2,1,0)\, .
\end{gather*}
Here again we show that $Q_3$ and $Q_4$ give a quotient $P$-module which is generically free. A general traceless matrix in $Q_4 = V(1, 1)$ is diagonalizable with pairwise distinct diagonal entries, and the stabilizer in general position in $\mathrm{SL}_3 (\CC )$ of $V(1, 1)$ is hence conjugate to the maximal torus $T$ of diagonal matrices. To see how $T\ltimes \CC^3$ acts on $Q_3 \oplus \CC\cdot m$ for a general $m\in V(1, 1)$, we remark that the weight space decomposition of $Q_3$ again contains that of $\Sigma^{3, 2, 0} \CC^3$ already computed above. Thus we can conclude the argument by noting that $T$ acts generically freely on the same sum of weight spaces as in the preceding case.

\

We will later treat the case $n=3$ separately. Note that now for $4\le n \le 10$ we can write those $\lambda$ giving exceptional
$\mu_0$ and $\mu_1$ as 
\[
\lambda = (\lambda_1, \dots , \lambda_{n-1}, 0)
\]
with $0 \le \lambda_{n-1} \le 2$ and $0 \le \lambda_1 \le 4$. This allows us to quickly classify the remaining exceptions. Note that for
$4\le n$ the exceptions (2), (3), (4), (5), (6) above of course reoccur. The proofs of their rationality remains the same.
Therefore we only list the additional exceptions: 
\begin{gather}
\Lambda^3 \CC^{10}, \Lambda^4 \CC^9 \quad \mathrm{or}\: \mathrm{one}\: \mathrm{of} \: \mathrm{the}\:
\mathrm{duals}
\end{gather}
(those exceptions we already know) \emph{and}
\[
\Sigma^{2,2,0,0,0} \CC^5 \mathrm{or}\: \mathrm{the}\:
\mathrm{dual}\, .
\]
Now the decomposition in the last case has (we restrict to $P'$):
\[
Q_1 = \Sigma^{2,2} \CC^4,\: Q_2 = \Sigma^{2,1} \CC^4, \: Q_3 = \mathrm{Sym}^2 \CC^4\, .
\] 
Thus this representation has a $P'$-generically free quotient ($Q_2 \oplus Q_3$) and $\dim Q_1 = 20$ is bigger than the
dimension of $P'$.\\
Let us consider now the case $n=3$; here there are naturally quite a few: 
\begin{gather}
(a, 4, 0), \: a=4,5,6,7,8, \\
(b, 3, 0), \: b=4,5,6,7, \\
(c, 2, 0), \: c=3,4,5,6, \\
(d, 1, 0), \: d=3,4,5,\\
(4,0,0).
\end{gather}
This proves Proposition \ref{pMainReduction}.
\end{proof}

\section{The base for the induction for $\mathrm{SL}_n (\CC )$}\xlabel{sInductionBaseSL}

The purpose of this section is to prove rationality of $V/P$ (or $V/P'$ which is equivalent by Lemma \ref{lOppositeParabolic})
for the representations $V$ listed in Proposition \ref{pMainReduction}, thus completing the proof of Theorem
\ref{tSLIrreducible}. The easier part is

\begin{proposition}\xlabel{pInductionBaseSL2}
If $V=\Sigma^{\lambda} \CC^3$ and $\lambda$ is one of
\begin{gather*}
(a, 4, 0), \: a=4,5,6,7,8, \\
(b, 3, 0), \: b=4,5,6,7, \\
(c, 2, 0), \: c=3,4,5,6, \\
(d, 1, 0), \: d=3,4,5,\\
(4,0,0),
\end{gather*}
then $V/P$ (resp. $V/P'$) is rational.
\end{proposition}

\begin{proof}
Apart from $\Sigma^{(8, 4, 0)}\CC^3
/P$, the remaining cases have filtrations as
$P$-representations or $P'$-representations with last quotient $Q_l$ one of
\[
\CC, \: \CC^2 , \: \mathrm{Sym}^2 \CC^2, \: \mathrm{Sym}^3 \CC^2 \, .
\]
If $Q_l$ is $\CC$ we have a $(P,\: \mathrm{SL}_2 (\CC ))$-section in our space whence rationality of the quotient. If $Q_l$
is $\CC^2$, then we get a $(P, \: \mathbb{G}_a)$-section, and $\mathbb{G}_a$-quotients are rational.\\
For the end
$\mathrm{Sym}^2 \CC^2$  note that a generic line in this space gives a $(\mathrm{SL}_2 (\CC ), \mathbb{Z}/2 \ltimes \CC^{\ast
})$-section in this space; so every $V$ terminating in $\mathrm{Sym}^2 \CC^2$ has a $(P, \mathbb{Z}/2 \ltimes \CC^{\ast
})$-section $S$ as well. $\CC^{\ast }$ acts in the weight spaces $S_{\chi}$ into which $S$ decomposes via the corresponding
character $\chi$, and $\mathbb{Z}/2$ interchanges $S_{\chi }$ and $S_{-\chi }$. Thus the rationality follows (there is at least
one non-zero one-dimensional weight space $S_0$ in $S$; $S/\CC^{\ast }$ is birational to $S/S_0$ with $\ZZ/2$ acting as
mentioned before).\\
For $Q_l = \mathrm{Sym}^3 \CC^2$ note that the stabilizer $H$ in general position in $\mathrm{SL}_2 (\CC )$
is $\ZZ / 3$, $N(H)= \mathfrak{S}_3$, and hence a generic line is a $(\mathrm{SL}_2 (\CC ), \: \mathfrak{S}_3 )$-section. Thus
a $V$ terminating in $\mathrm{Sym}^3 \CC^2$ has a $(P, \: \mathfrak{S}_3 )$-section, and we reduce to a rationality question
for a quotient by $\mathfrak{S}_3$ which always has a positive answer.\\
Thus it remains to prove rationality for $\Sigma^{(8, 4, 0)}\CC^3/P$. But this has a $P$-generically free two-step quotient for
which the dimension of the corresponding $P$-subrepresentation is bigger than $\dim P =5$.
\end{proof}

It remains to analyze the two exceptional exterior power representations.

\begin{proposition}\xlabel{pLambda3}
The quotient $(\Lambda^3 \CC^{10})/P'$ is rational.
\end{proposition}

\begin{proof}
The $P'$-representation is an extension
\[
0 \to \Lambda^3 \CC^9 \to V \to \Lambda^2 \CC^9 \to 0 \, .
\] 
The group $\mathrm{SL}_9 (\CC )$ acts generically transitively in $\Lambda^2 \CC^9$ with stabilizer in general position
$\mathrm{Sp}_8 (\CC )\ltimes \CC^8$. The representation $\Lambda^3 \CC^9$ decomposes as $\mathrm{Sp}_8 (\CC
)$-representation
\begin{gather*}
\Lambda^3 (\CC^9) = \Lambda^3 (\CC^8 + \CC ) \\
= \Lambda^3 \CC^8 + \Lambda^2 \CC^8 = (\Lambda^3 \CC^8)_0 + \CC^8 + (\Lambda^2 \CC^8)_0 + \CC\, ,
\end{gather*}
subscript $0$ referring to spaces of traceless (harmonic) tensors for $\mathrm{Sp}_8 (\CC )$. Thus $V/P'$ is birational to the
quotient of the extension
\[
0 \to (\Lambda^3 \CC^8)_0 \to W \to (\Lambda^2 \CC^8)_0 \to 0
\]
by the group $\mathrm{Sp}_8 (\CC )\ltimes \CC^8$. The stabilizer in general position in $(\Lambda^2 \CC^8)_0$ inside
$\mathrm{Sp}_8 (\CC )$ is isomorphic to $H=\prod_{i=1}^4\mathrm{SL}_2 (\CC )$, the normalizer being a group $N(H)=\mathfrak{S}_4
\ltimes \mathrm{SL}_2 (\CC )^4$. As $N(H)$-representation \[\CC^8 = R_1 + R_2 + R_3 + R_4\, ,\] $R_i = \CC^2$, a standard
representation of $\mathrm{SL}_2 (\CC )$, $\mathfrak{S}_4$ permuting the copies. The $H$-invariants in $(\Lambda^2 \CC^8)_0$
form a three-dimensional space $S$, the standard representation for $N(H)/H = \mathfrak{S}_4$. Now $(\Lambda^3 \CC^8)_0$
decomposes as $N(H)$-representation
\begin{gather*}
(\Lambda^3 \CC^8)_0 =\\ \left( \sum_{\stackrel{1\le i< j < k \le 4}{\{ i, j, k \} \cup \{h \} = \{ 1,2,3,4 \}}}   (\Lambda^2 R_i + \Lambda^2 R_j  + \Lambda^2 R_k) \otimes
R_{h} \right)^0  +
\sum_{1\le i< j < k
\le 4} R_i
\otimes R_j
\otimes R_k\, , 
\end{gather*}
where the first summand denotes the $N(H)$-representation obtained by removing one $N(H)$-invariant summand $R_1 + R_2 + R_3 + R_4$ from the
$N(H)$-representation
\[
 \sum_{\stackrel{1\le i< j < k \le 4}{\{ i, j, k \} \cup \{h \} = \{ 1,2,3,4 \}}}  (\Lambda^2 R_i + \Lambda^2 R_j  + \Lambda^2 R_k) \otimes
R_{h} 
\]
(so the former is $2 (R_1 + R_2 + R_3 + R_4)$ as $H$-representation), 
and the $N(H)$-action is the one suggested by the notation ($\mathfrak{S}_4$ acts by permuting the
indices). So we have to prove rationality of the quotient of the extension 
\[
0 \to (\Lambda^3 \CC^8)_0 \to W' \to S \to 0
\]
by the group $N(H) \ltimes \CC^8 = N(H) \ltimes (R_1 + R_2 + R_3 + R_4)$. Note that the translations $\CC^8$ map $S$ only into
$(\sum_{1\le i< j < k \le 4} (\Lambda^2 R_i + \Lambda^2 R_j  + \Lambda^2 R_k) \otimes R_{h})^0$, and that the $N(H)$-action on
$\sum_{1\le i< j < k \le 4} R_i \otimes R_j \otimes R_k$ is generically free, so that it suffices to prove stable rationality
of the quotient
\[
\left( \sum_{1\le
i< j < k \le 4} R_i \otimes R_j \otimes R_k \right) /N(H)
\]
of level at most $8 + 3 = \dim (R_1 + R_2 + R_3 + R_4 ) + \dim S$. We will prove rationality of the quotient by $N(H)$ of the
$N(H)$-representation
\[
(\sum_{1\le i< j < k \le 4} R_i \otimes R_j \otimes R_k ) + (R_1 + R_2 + R_3 + R_4)\, .
\]
A point in $R_1 + R_2 + R_3 + R_4$ is an $(N(H), \Gamma )$-section, where 
\[
\Gamma = \mathfrak{S}_4 \ltimes (\mathbb{G}_a \times \mathbb{G}_a \times \mathbb{G}_a \times \mathbb{G}_a) \, .
\]
Thus it remains to prove rationality of $(\sum_{1\le i< j < k \le 4} R_i \otimes R_j \otimes R_k ) /\Gamma$. We embed $\Gamma$
in a larger group acting on $(\sum_{1\le i< j < k \le 4} R_i \otimes R_j \otimes R_k )$, namely
\[
\Gamma' = \mathfrak{S}_4 \ltimes \prod_{1\le i< j < k \le 4} \left( (\mathbb{G}_a )_i \times  (\mathbb{G}_a )_j  \times
(\mathbb{G}_a )_k \right)
\]
where we added subscripts to $\mathbb{G}_a$ to keep track of the copies. Now  $(\sum_{1\le i< j < k \le 4} R_i \otimes R_j
\otimes R_k ) /\Gamma$ is birational to a vector bundle over the base $(\sum_{1\le i< j < k \le 4} R_i \otimes R_j
\otimes R_k )/\Gamma'$. The latter is clearly rational.
\end{proof}

\begin{proposition}\xlabel{pLambda4}
The quotient $(\Lambda^4 \CC^9)/P'$ is rational.
\end{proposition}

\begin{proof}
Here the $P'$-representation is an extension
\[
0 \to \Lambda^4 \CC^8 \to V \to \Lambda^3 \CC^8 \to 0 \, .
\]
The stabilizer in general position inside $\mathrm{SL}_8 (\CC )$ of a generic point in $\Lambda^3 \CC^8$ is $H=\mathrm{PSL}_3
(\CC )$. It is embedded in $\mathrm{SL}_8$ via the representation of $\mathrm{SL}_3 (\CC )$ on $V(1,1) \subset \CC^3 \otimes
(\CC^3)^{\vee }$ (adjoint representation) which has dimension $8$. Its normalizer inside $\mathrm{SL}_8 (\CC )$ is 
$N(H) = \langle \epsilon \rangle \ltimes \mathrm{PSL}_3 (\CC ))$, where $\epsilon^2$ is a generator of the center $C = \ZZ /8\ZZ$ of $\mathrm{SL}_8 (\CC )$ and $\epsilon$ acts on $\mathrm{PSL}_3 (\CC )$ via an
outer automorphism $g\mapsto (g^t)^{-1}$. The action of $\epsilon$ on $V(1,1)$ is given by exchanging $\CC^3$ and $(\CC^3)^{\vee
}$ and multiplication by $\zeta = \mathrm{exp}(2\pi i /16 )$ (this is inserted to make the determinant of the total transformation equal to $1$). Thus, in terms of traceless $3\times 3$ matrices $M$, the  generator $\epsilon$ acts via
\[
\epsilon \cdot M : = \zeta M^t \, .
\]
The $H$-invariants $S =(\Lambda^3 \CC^8 )^H$ are one-dimensional. We have to compute how $\Lambda^4 \CC^8$ decomposes as $N(H)$-representation. Note that we have to compute
$\Lambda^4 V(1,1)$ as $\mathrm{PSL}_3 (\CC )$-representation. Noting that generally as $\mathrm{GL}(E) \times
\mathrm{GL}(F)$-module 
\[
\Lambda^c (E\otimes F ) = \bigoplus_{\lambda } \Sigma^{\lambda } E \otimes \Sigma^{\bar{\lambda }} F
\]
($\bar{\lambda }$ denoting the conjugate partition to $\lambda$, and the sum running over Young diagrams $\lambda$ with $c$
boxes), we find
\begin{gather*}
\Lambda^2 (\CC^3 \otimes (\CC^3)^{\vee }) = (\CC^3)^{\vee }\otimes \mathrm{Sym}^2 (\CC^3)^{\vee } + \CC^3 \otimes
\mathrm{Sym}^2 (\CC^3), \\
\Lambda^3 (\CC^3 \otimes (\CC^3)^{\vee }) = \mathrm{Sym}^3 \CC^3 + \mathrm{Sym}^3 (\CC^3)^{\vee } + \Sigma^{2,1} \CC^3 \otimes
\Sigma^{2,1} (\CC^3 )^{\vee }, \\
\Lambda^4 (\CC^3 \otimes (\CC^3)^{\vee }) \\= \mathrm{Sym}^2 \CC^3 \otimes \mathrm{Sym}^2 (\CC^3)^{\vee } + \CC^3 \otimes
\Sigma^{3,1} (\CC^3)^{\vee } + \Sigma^{3, 1} \CC^3 \otimes (\CC^3)^{\vee }\, .
\end{gather*} 
Hence, in view of $\Lambda^i (V(1,1) + \CC ) = \Lambda^i V(1,1) + \Lambda^{i-1} V(1,1)$ and the correspondence $V(a,b) =
\Sigma^{a+b, b} \CC^3$, we find inductively
\begin{gather*}
\Lambda^2 V(1,1) = V(3,0) + V(1,1) + V(0, 3) , \\
\Lambda^3 V(1,1) = V(3,0) + V(0, 3) + V(2,2) + V(1,1) + V(0, 0), \\
\Lambda^4 V(1,1) = 2 (V(1,1) + V(2,2)) \, ,
\end{gather*}
where the action of $\epsilon$ is still induced by the process of exchanging $\CC^3$ with $(\CC^3)^{\vee }$ and multiplication by
$\zeta$. Thus we have to prove rationality of the extension
\[
0 \to 2(V(1,1) + V(2,2)) \to W \to S =  (\Lambda^3 \CC^8 )^H \to 0
\] 
modulo the group $N(H) \ltimes \CC^8 = N(H) \ltimes V(1,1)$, hence we have to prove rationality of 
\[
\left( V(2,2) + V(2,2) + V(1,1) \right) / N(H)\, .
\]
The action of $N(H)$ on $V(2,2) + V(1,1)$ is already generically free, so it suffices to prove stable rationality of this of
level at most $\dim V(2,2) = 27$, and for this it suffices to exhibit a generically free $N(H)$-representation of dimension
$\le 27$ with rational quotient. We take
\[
R=\CC^3 \otimes (\CC^3)^{\vee } + \CC^3 \otimes (\CC^3)^{\vee }
\]
(pairs of $3\times 3$-matrices). Again we let $\epsilon$ act via exchanging $\CC^3$ and $(\CC^3 )^{\vee }$ and multiplication by $\zeta$. It will be convenient to prove that $R/N(H)$ is rational by showing that $R/(\CC^{\ast } \times (\ZZ/2\ZZ \ltimes \mathrm{PGL}_3 (\CC )))$ is rational over which it is generically a torus bundle. Here $\CC^{\ast } \times (\ZZ/2\ZZ \ltimes \mathrm{PGL}_3 (\CC ))$ is the normalizer of $\mathrm{PGL}_3 (\CC )$ in $\mathrm{GL}_8 (\CC )$. 

\

Now $R/(\ZZ/2\ZZ \ltimes \mathrm{PGL}_3 (\CC ))$ is birational to the relative Jacobian of degree $0$ line bundles over the
parameter space $\mathcal{P}$ of plane cubics (rational, birational to $(\CC^3 \otimes (\CC^3)^{\vee } + \CC^3 \otimes (\CC^3)^{\vee })/H$),
modulo the involution identifying $\mathcal{L}$ and $\mathcal{L}^{\vee }$ in a fibre. The quotient is a $\PP^1$-bundle $\mathcal{J} \to \mathcal{P}$ over a
rational base $\mathcal{P}$, which has a section corresponding to the trivial line bundle. Thus $R/(\ZZ/2\ZZ \ltimes \mathrm{PGL}_3 (\CC ))$ is rational, and if one in addition divides by $\CC^{\ast }$, one still gets something rational, because $\CC^{\ast}$ acts on $\mathcal{P}$ linearly (induced by the scaling $(x,y,z)\mapsto (\lambda x, \lambda y, z)$), so $\mathcal{P}/\CC^{\astÊ}$ is rational, and $\mathcal{J}/\CC^{\ast }$ is a Zariski locally trivial $\PP^1$-bundle over $\mathcal{P}/\CC^{\astÊ}$ as torus actions have rational sections.
\end{proof}

\section{The case of the symplectic group}\xlabel{sInductionProcessSp}

We proceed to discuss the analogues of the results of section \ref{sInductionProcessSL} for the symplectic group. 

\

We put $N:= 2n$ and consider irreducible $\mathrm{Sp}_N (\CC )$-representations 
\[
V = \Sigma^{(\lambda_1, \dots , \lambda_n )}_0 \CC^N \, .
\]
They are indexed by nonincreasing $n$-tuples of nonnegative integers $\lambda = (\lambda_1, \dots , \lambda_n )$, and
$\Sigma^{\lambda }_0 \CC^N$ and $\Sigma^{\mu }_0\CC^N$ are isomorphic if and only if $\lambda =\mu$ (in contrast to the
$\mathrm{SL}_n (\CC )$-case). Concretely (see \cite{G-W} section 10.2.3), if $A$ is some tableau of shape $\lambda$, $k
=|\lambda | $, (thus $A$ is a numbering of the boxes of $\lambda$ with the ciphers from $1$ to $k$), then
\[
\Sigma^{\lambda }_0 \CC^N = s(A) \left(  \mathcal{H} ( \otimes^k \CC^N , \omega ) \right)
\]
where $\mathcal{H} ( \otimes^k \CC^N , \omega )$ is the space of $\omega$-\emph{harmonic tensors} or \emph{traceless tensors},
i.e. those tensors
$u\inÊ\otimes^k \CC^N$ which are annihilated by all contractions $C_{ij}\, : \otimes^k \CC^N \to \otimes^{k-2}\CC^{N}$ with the
symplectic form $\omega$:
\[
C_{ij} (v_1 \otimes \dots \otimes v_k ) = \omega (v_i , v_j ) v_1 \otimes \dots \otimes \hat{v_i} \otimes \dots \otimes
\hat{v_j} \otimes \dots \otimes v_k \, ,
\] 
and $s(A)$ is the usual \emph{Young symmetrizer} associated to the tableau $A$. Recall that $s(A)$ is an element in the group
algebra $\CC [\mathfrak{S}_k ]$, product of the column skew symmetrizer and row symmetrizer of $A$. This justifies the notation
$\Sigma_0$ in analogy to the notation $\Sigma$ for $\mathrm{GL}_N (\CC )$ or $\mathrm{SL}_N (\CC )$.

\

Now $\CC^N$ is a $\mathrm{Sp}_N (\CC )$-representation, $\mathrm{Sp}_N (\CC )$ acts generically
transitively in it. We choose the standard symplectic basis \[\mathcal{B}Ê= (e_1, \dots , e_n, \; f_n, \dots , f_1 )\] (in this order) such that $\omega$
is given by 
\[
\omega (e_i , f_j ) = \delta_{ij}\, ,
\]
so that $\mathrm{Sp}_{2n}(\CC )$ gets identified with the group of invertible $2n\times 2n$-matrices $M$ such that
\[
M^t J_{2n} M = J_{2n}
\]
where $J_{2n}$ satisfies
\[
J_{2n} = \left( \begin{array}{ccc} 0 & 0 & 1 \\
                                                  0 & J_{2n-2} & 0\\
                                                   -1 & 0 & 0 \end{array}\right)
\]
and 
\[
J_2 = \left(\begin{array}{cc} 0 & 1\\ -1 & 0 \end{array}\right)\, .
\]
We denote the stabilizer of $e_1$ in $\CC^N$ again by $P$. 

\

We want to describe $P$ also in concrete matrix terms. The ordered basis $\mathcal{B}$ is convenient for seeing the symmetry of the situation. 
Namely, the subgroup $P \subset \mathrm{Sp}_N (\CC )$ is then given by the matrices
\[
\left( \begin{array}{ccc}
1     &   v^t J_{2n-2} g     &   c\\
        0_{2n-2}   &     g          &                      v  \\    
        0   &     0_{2n-2}^t          &                    1
\end{array}\right)
\]
where $0_{2n-2}$ is a column vector with $(2n-2)$ zero entries, $v$ is a column vector in $\CC^{2n-2}$, and
$g\in\mathrm{Sp}_{2n-2} (\CC )$ with the same matrix realization as before. Thus
\[
P = \mathrm{Sp}_{2n-2} \ltimes R_u (P)
\]
where $R_u (P)$ is the unipotent radical of $P$, a solvable group consisting of the matrices
\[
\left( \begin{array}{ccc}
1     &   v^t J_{2n-2}      &   c\\
        0_{2n-2}   &     \mathrm{id}_{2n-2}          &                      v  \\    
        0   &     0_{2n-2}^t          &                    1
\end{array}\right)
\]
which is an extension
\[
0 \to \mathbb{G}_a \to R_u (P) \to \mathbb{G}_a^{2n-2} \to 0 \, .
\]
In fact the group $\mathbb{G}_a$ given by matrices
\[
\left( \begin{array}{ccc}
1     &   0_{2n-2}^t    &   c\\
        0_{2n-2}   &     \mathrm{id}_{2n-2}         &                      0_{2n-2}  \\    
        0   &     0_{2n-2}^t          &                    1
\end{array}\right)
\]
is central in $P$. 

\begin{remark}\xlabel{rSectionSp}
As $\mathrm{Sp}_{2n} (\CC )$ acts generically transitively in $\CC^{2n}$ we have for a generically free representation $W$ of
$\mathrm{Sp}_{2n}(\CC )$ that
\[
(W \oplus \CC^{2n})/\mathrm{Sp}_{2n } (\CC ) = W / P \, .
\]
In view of this we have to understand how irreducible $\mathrm{Sp}_{2n } (\CC )$-representations $V$ behave when restricted to
$P$.
\end{remark}

According to \cite{G-W}, Chapter 8, the representation $V$ decomposes as a $\mathrm{Sp}_{N-2}( \CC )$-representation as
\[
V = \bigoplus_{\mu } m (\lambda , \mu ) \Sigma_0^{\mu } \CC^{N-2}
\]  
where the multiplicity $m(\lambda, \mu )$ is nonzero if and only if $\mu$ satisfies the \emph{double interlacing condition}
\[
\lambda_j \ge \mu_j \ge \lambda_{j+2}, \: j =1, \dots , n-1
\]
where we put $\lambda_{n+1} = 0$. If under these conditions
\[
x_1 \ge y_1 \ge x_2 \ge y_2 \ge \dots \ge x_n \ge y_n
\]
is a nonincreasing rearrangement of $\{ \lambda_1, \dots , \lambda_n, \mu_1, \dots, \mu_{n-1}, 0 \}$, then the multiplicity is
given by a product formula
\[
m(\lambda , \mu ) = \prod_{j=1}^n (x_j - y_j +1)\, .
\]

The structure of $P$-module here is a little more difficult to describe than in the $\mathrm{SL}_n (\CC )$-case. As in the
$\mathrm{SL}_n (\CC )$ case, the $P$-module $V$ has a filtration
\[
0 \subset V_1 \subset \dots \subset V_l = V
\] 
such that $Q_i = V_i /V_{i-1} \subset V/V_{i-1}$, $i=1, \dots , l$, $V_{0}:= (0)$, is completely reducible and maximal with that
property. So $R_u (P)$ acts trivially in $Q_i$. 

\begin{remark}\xlabel{rPmodules}
If $V$ is irreducible, $V= \Sigma^{(\lambda_1, \dots , \lambda_n )}_0 \CC^{2n}$, then the maximal completely reducible
subrepresentation $Q_1$ of $V$ as $P$-module is the irreducible $\mathrm{Sp}_{2n-2} (\CC )$-module 
\[
\Sigma^{(\lambda_2, \dots , \lambda_{n-1} , \lambda_n ) } \CC^{2n -2}\, .
\]
This follows since there is a unique line of Borel eigenvectors in $V$ (highest weight vectors), hence there is a unique such line
in $Q_1$. The highest weight of $Q_1$ is the restriction of the highest weight of $V$ to $\mathrm{Sp}_{2n-2} (\CC )$.
\end{remark}

From the matrix form of $P$ given above we deduce that $\CC^N$ itself has a filtration as a $P$-module with
\[
Q_1 = \CC_{(1)} , \; Q_2 = \CC^{N-2} , \; Q_3 = \CC_{(2)} \, ,
\]
where we put subscripts to distinguish the copies of $\CC$.

In general the filtration in the symplectic case is symmetric which is not surprising in view of the fact that symplectic
representations are self-dual.

\

To completely understand the $P$-module filtrations, we have to use a result of \cite{W-Y}: note that the above embedding $\mathrm{Sp}_{2n-2} (\CC )
\subset \mathrm{Sp}_{2n } (\CC )$ factors over an embedding $\mathrm{SL}_2 (\CC ) \times \mathrm{Sp}_{2n-2} (\CC )$ in
$\mathrm{Sp}_{2n } (\CC )$; then in the decomposition
of
$V$ as $\mathrm{SL}_2 (\CC ) \times \mathrm{Sp}_{2n-2} (\CC )$-module, the multiplicity space of $\Sigma^{\mu }_0 \CC^N$ as above
can be interpreted as the $\mathrm{SL}_2 (\CC )$-module
\[
\bigotimes_{i=1}^n V_{x_i - y_i} \, ,
\]
where $V_{x_i -y_i}$ are binary forms of degree $x_i -y_i$. Note that the central subgroup $\mathbb{G}_a$ above is contained
in $\mathrm{SL}_2 (\CC )$ as the unipotent radical in this picture.

\begin{proposition}\xlabel{pStructurePmodules}
If $V$ is a $P$-module coming from an $\mathrm{Sp}_{2n} (\CC )$-module by restriction with filtration
\[
0 \subset V_1 \subset \dots \subset V_l = V
\] 
and as $\mathrm{Sp}_{2n-2} (\CC )$-modules
\[
V = Q_1 \oplus \dots \oplus Q_l\, ,
\]
where $Q_i = V_i /V_{i-1}$, 
then the $Q$'s are determined from the branching rule for $\mathrm{Sp}_{2n} (\CC ) \to \mathrm{SL}_2 (\CC )
\times \mathrm{Sp}_{2n-2} (\CC )$ given above (which implies in particular how the copies of one irreducible representation $R$
of $\mathrm{Sp}_{2n-2}(\CC )$ are distributed in the filtration) together with the following facts:
\begin{itemize}
\item[(1)]
For every irreducible $Q'$ in $Q_i$ there is an irreducible $S'$ in $Q_{i-1}$ with $S' \subset Q'\otimes \CC^{2n-2}$ as
$\mathrm{Sp}_{2n-2} (\CC )$-representation.
\item[(2)]
The central subgroup $\mathbb{G}_a$ shifts the filtration degree by $2$:
\[
\mathbb{G}_a \cdot (V_j) \subset V_{j-2}
\]
and for every irreducible $Q'$ in $Q_i$ there is an irreducible summand isomorphic to $Q'$ inside $Q_{i-2}$ or $Q'$ is mapped to
zero under the subgroup $\mathbb{G}_a$. 
\item[(3)]
The symmetry $Q_i = Q_{l+1 -i}$ holds.
\end{itemize}
\end{proposition}

\begin{proof}
We have to consider the slightly bigger parabolic subgroup $\tilde{P} \supset P$ given by matrices
\[
\left( \begin{array}{ccc}
\lambda     &   v^t J_{2n-2} g     &   \lambda^{-1}c\\
        0_{2n-2}   &     g          &                     \lambda^{-1} v  \\    
        0   &     0_{2n-2}^t          &                    \lambda^{-1}
\end{array}\right)
\]
so that $\tilde{P} = \CC^{\ast } \cdot P$. The filtration of $V$ as $\tilde{P}$-module is the same as the one as $P$-module, and
the quotients $Q_i$ are completely reducible modules for the reductive part $L= \CC^{\ast } \times \mathrm{Sp}_{2n-2} (\CC )$, but
the 
$\CC^{\ast }$-action will encode valuable information about the grading of the filtration.

\

The Lie algebra $\mathfrak{r}$ of $R_u (P)$ is 
\[
\CC^{2n-1} = \left\{  (v , c) \, | \, v\in \CC^{2n-2}, \; c\in \CC \right\} , \quad \left[ (v, c) , \: (v', c' ) \right] =
(0, \omega (v , v' ) ) \, .
\]
As an $L = \CC^{\ast }\times \mathrm{Sp}_{2n-2} (\CC )$-module, $\mathfrak{r}$ decomposes into a standard representation
$\CC^{2n-2}$ and
$\CC$ where $\CC^{\ast }$ acts via scalings by $\lambda$ in $\CC^{2n-2}$. Any
$\tilde{P}$-module $V$ (not necessarily coming from an $\mathrm{Sp}_{2n } (\CC )$-module) may be viewed as an
$\mathfrak{r}$-module where the structure map
\[
\mathfrak{r} \to \mathrm{End} (V, \: V)
\]
is $L$-\emph{equivariant}. If we have a two-step extension $\tilde{P}$-module
\[
0 \to S \to V \to Q \to 0
\]
with $S$ the maximal completely reducible submodule, we get an $L$-homomorphism
\[
\CC \oplus \CC^{2n-2} \to \mathrm{Hom} (Q, \: S)\, .
\]
Then for every irreducible $Q'$ in $Q$ there is an irreducible $S'$ in $S$ with a nonzero homomorphism
\[
\CC \oplus \CC^{2n-2} \to \mathrm{Hom} (Q', \: S')\, ,
\]
since $S$ is maximal completely reducible inside $V$, hence
\[
S' \subset Q' \otimes (\CC \oplus \CC^{2n-2})\, .
\]
We must have $S' \subset Q'\otimes \CC^{2n-2}$ for otherwise, by the structure of the Lie algebra $\mathfrak{r}$, $Q'$ is
annihilated by all of $\mathfrak{r}$.

\

Now suppose our $P$-module is the restriction of some irreducible $\mathrm{Sp}_{2n} (\CC )$-module $V$, so that the $P$-span of
the last (irreducible) $Q_l$ is the whole of $V$. Then the action of $\CC^{\ast }$ will be via some $\lambda^{-k}$ in $Q_l$, via
$\lambda^{-k+1}$ in $Q_{l-1}$ up to $\lambda^k$ in $Q_1$. This shows that the action of
$\mathfrak{r}$ preserves the grading in the sense that
\[
\CC^{2n-2}\cdot Q_i \subset Q_{i-1}\; \mathrm{and} \; \CC \cdot Q_i \subset Q_{i-2}\, ,
\]
which would not have been visible without the $\CC^{\ast }$-action. 
\end{proof}

As an example of how $R_u (P)$ acts let us
compute the filtrations for 
\[
\Sigma^{2, 0, \dots , 0}_0 \CC^N =\mathrm{Sym}^2 \CC^N \; \mathrm{and} \; \Sigma^{1,1,0, \dots , 0}_0 \CC^N = \Lambda^2_0 \CC^N\,
.
\]
Now looking at the expansion of 
\begin{gather*}
\Lambda^2 ( \CC_{(1)} + \CC^{N-2}  + \CC_{(2)})
\end{gather*}
we find for $\Lambda^2_0 \CC^N$ 
\[
Q_1 = \CC^{N-2}, \; Q_2 = \Lambda^2_0\CC^{N-2} + \CC , \; Q_3 = \CC^{N-2}\, .
\]
Comparing this with the expansion of $(\CC^{N-2} + \CC_{(1)} + \CC_{(2)})^{\otimes 2}$ we get for $\mathrm{Sym}^2 \CC^N$:
\[
Q_1 = \CC , \; Q_2 = \CC^{N-2}, \; Q_3 = \mathrm{Sym}^2 \CC^{N-2} + \CC , \; Q_4 = \CC^{N-2}, \; Q_5 =\CC \, .
\]

\

We need some criterion when a $P$-representation is generically free. One should compare the following with the related Lemma 5.6 in \cite{BBB10}.

\begin{lemma}\xlabel{lSpFree}
Let $V$ be a $P$-representation with a filtration as above with three steps:
\[
V_1 \subset V_2 \subset V_3 =V\, .
\]
So $Q_1 =V_1$ is the maximal completely reducible subrepresentation of $V$, $Q_2 = V_2 /V_1$ the maximal completely reducible
subrepresentation of $V/V_1$, and $Q_3 = V/V_2$ is also completely reducible. We assume for simplicity that $Q_3$ is irreducible
and $V = \langle P\cdot Q_3 \rangle $.\\
If $Q_3$ is an R-representation of
$\mathrm{Sp}_{2n-2} (\CC )$, and the action of $P$ on $V$ does not factor through $P/\mathbb{G}_a$, then $V$ is generically free
for
$P$.
\end{lemma}

\begin{proof}
By assumption $\mathbb{G}_a$ acts via nontrivially translating $Q_3$ into $Q_1$ whence $V/\mathbb{G}_a$ is an affine (or vector)
bundle over the two-step extension $V/V_1$ in which $P/\mathbb{G}_a$ acts. Hence we just have to show that $G:= P/\mathbb{G}_a$
acts generically freely in $\tilde{V}:= V/V_1$. We will first investigate when the unipotent radical $U:=\CC^{2n-2}$ of $G$ might
not act generically freely in $\tilde{V}$: in that case we get for a generic $q \in Q_3$ a nonzero element $t$ in the Lie algebra
$\mathfrak{u}$ of $U$ with $t\cdot q= 0$. The variety
\[
\mathcal{K}:=\left\{ (t, \: q)\in \PP (\mathfrak{u} ) \times Q_3 \, |\, t\cdot q =0 \right\}
\]
fibres over $Q_3$ and $\PP (\mathfrak{u}) \simeq \PP^{2n-3}$ and is a homogeneous vector bundle over $\PP^{2n-3}$ by
$\mathrm{Sp}_{2n-2} (\CC )$-equivariance. Its fibre dimension is at least $\dim Q_3 - (2n-3)$ since $\mathcal{K}$ dominates
$Q_3$. Hence the cokernel $\mathcal{Q}$ of the map $\mathcal{K} \to \mathcal{O} \otimes Q_3$, i.e. the image of the map
\[
\tau \, :\, \mathcal{O} \otimes Q_3 \to \mathcal{O} (1) \otimes Q_2\, ,
\]
is a
homogeneous vector bundle over
$\PP^{2n-3}$ of rank at most $2n-3$. These are easy to classify as they arise via representations of the stabilizer parabolic
subgroup inside $\mathrm{Sp}_{2n-2} (\CC )$ on $\PP^{2n-3}$: they are
\[
\mathcal{T} , \; \Omega^1 , \; \mathcal{S}:= \mathcal{O} (-1)^{\perp } / \mathcal{O} (-1), 
\]
or twists of these by $\mathcal{O} (k)$, or direct sums of bundles $\mathcal{O}(j)$; here $\mathcal{O}(-1)^{\perp}$ is the
perpendicular with respect to the symplectic form to $\mathcal{O} (-1) \subset V$. Since $\tau$ is nonzero on $H^0$-level, both
$\mathcal{Q}$ and
$\mathcal{Q}^{\vee }(1)$ have sections. So we can just have that $\mathcal{Q}$ is one of
\[
\mathcal{T} (-1), \: \Omega^1 (2), \: X\otimes \mathcal{O} \oplus Y \otimes \mathcal{O}(1)
\]
where $X$ and $Y$ are $\mathrm{Sp}_{2n-2}(\CC )$-representations the sum of whose dimensions does not exceed $2n -3$, so sums of
trivial representations. Here $H^0 (\mathcal{T} (-1) ) = \CC^{2n-2}$, $H^0 (\Omega^1 (2) ) = \Lambda^2 \CC^{2n-2}$. Note also that
\[
\mathcal{T}(-1) \simeq V/\mathcal{O}(-1), \: \mathcal{O}(1) = V/\mathcal{O}(-1)^{\perp }\, ,
\]
one has the exact sequence
\[
0 \to \mathcal{S} \to \mathcal{T}(-1) \to \mathcal{O} (1) \to 0
\]
and by the Borel-Bott-Weil Theorem $\mathcal{S}$ has no sections (and it is self-dual; however, $H^0 (\mathcal{S} (1)) \simeq
\Lambda^2_0 \CC^{2n-2}$, in agreement with the previous exact sequence).\\
Hence, by equivariance, if $U$ acts not generically freely on $V/V_1$, we obtain that $Q_3$ contains an E-representation for
$\mathrm{Sp}_{2n-2} (\CC )$, contradiction.\\
Hence it remains to see that $\mathrm{Sp}_{2n-2} (\CC )$ acts generically freely on the quotient $(V/V_1)/U$ which is
birationally a vector bundle over $Q_3$. Since $Q_3$ is an R-representation of $\mathrm{Sp}_{2n-2}(\CC )$, it suffices to note
that the center $\{ \pm 1 \}$ of $\mathrm{Sp}_{2n-2}(\CC )$ acts nontrivially on $Q_2 + Q_3$, since $Q_2 \subset
\CC^{2n-2}\otimes Q_3$, and this will continue to hold for the quotient $(V/V_1)/U$, provided $\dim Q_2 > 2n -2$; the latter is
satisfied because otherwise $Q_3$ would be an E-representation.
\end{proof}

We are now in a position to prove

\begin{theorem}\xlabel{tSp}
Suppose that $n\ge 4$. Let $V=\Sigma_0^{\lambda_1, \dots , \lambda_n }\CC^{2n}$ be an irreducible R-representation of
$\mathrm{Sp}_{2n} (\CC )$. Then
$V/P$ is rational. Hence, if
$V$ is already generically free for $\mathrm{Sp}_{2n} (\CC )$, then $V/\mathrm{Sp}_{2n}(\CC )$ is stably rational of level $2n$,
and if
$(-1)$ acts trivially in $V$, then a fortiori $(V\oplus \CC^{2n})/\mathrm{Sp}_{2n} (\CC )$ is stably rational of level $2n$.
\end{theorem}

The main strategy will be to find -if possible- a $P$-generically free quotient $V/V_0$ of the $P$-module $V$ such that $\dim V_0
\ge
\dim 
\mathrm{Sp}_{2n-2} (\CC ) + \dim R_u (P)$. Note that
\[
\dim \mathrm{Sp}_{2m} (\CC ) = 2m^2 +m\, .
\]
This will imply the assertion since the symplectic group and with it $P$ is special. To prove generic
freeness of
$V/V_0$ we have to distiguish two cases: if $Q_l = \Sigma^{\lambda_2, \dots , \lambda_{n}}_0 \CC^{2n-2}$ is an R-representation,
we try to use Lemma \ref{lSpFree} using the structure of $P$-modules summarized in Proposition \ref{pStructurePmodules}; if $Q_l$ is
an E-representation, we try to find a generically free quotient by an ad hoc method in each case. 

\begin{remark}\xlabel{rERepresentationsSp}
For $n\ge 3$, the E-representations for $\mathrm{Sp}_{2n} (\CC )$ are 
\[
\CC, \; \CC^{2n}, \; \mathrm{Sym}^2 \CC^{2n} = \Sigma^{2, 0, \dots , 0}_0 \CC^{2n}, \; \Lambda^2_0 \CC^{2n} =
\Sigma^{1,1,0,\dots, 0}_0 \CC^{2n}
\]
and in addition
\[
\Lambda^3_0 \CC^6 = \Sigma^{1,1,1}_0 \CC^6 , \; \Lambda^4_0 \CC^8 = \Sigma^{1,1,1,1}_0 \CC^8 \, .
\]
An irreducible representation $V$ of dimension $< \dim\mathrm{Sp}_{2n} (\CC )$ must be an E-representation. Hence for $n\ge 3$
these $V$'s are
\[
\CC, \; \CC^{2n}, \; \Lambda^2_0 \CC^{2n} \; (\dim \Lambda^2_0 \CC^{2n} = (n-1)(2n +1) )
\] 
and
\[
\Lambda^3_0 \CC^6 \; (\dim = 14 ) \, .
\]
\end{remark}

It is now easy to obtain

\begin{proposition}\xlabel{pSpR}
Theorem \ref{tSp} is true if $V=\Sigma^{\lambda_1, \dots , \lambda_n }_0 \CC^{2n}$ is such that $Q_l= \Sigma^{\lambda_2, \dots
,\lambda_{n}}_0 \CC^{2n-2}$ is an R-representation for $\mathrm{Sp}_{2n-2} (\CC )$, and $V$ is not an exterior power.
\end{proposition}

\begin{proof}
If $V$ is not an exterior power, then $V$ has a filtration which we will index more conveniently
\[
V_{-k} \subset \dots \subset V_0 \subset \dots \subset V_{k} = V, \;  Q_j = V_{j} /V_{j-1}
\]
whence $Q_j \simeq Q_{-j}$, \emph{and} $k$ \emph{is an integer} $\ge 2$. This follows because if $V$ is not an exterior power,
there will be a $\mu$ satisfying the double interlacing condition with respect to $\lambda$ such that if as above
\[
x_1 \ge y_1 \ge x_2 \ge \dots \ge x_n \ge y_n
\]
is a nonincreasing rearrangement of $\{ \lambda , \mu , 0\}$, then $x_i-y_i \ge 1$ for at least two values of $i$ or $x_i-y_i\ge
2$ for at least one value of $i$. The decomposition of the multiplicity space corresponding to $\mu$ as $\mathrm{SL}_2 (\CC
)$-module will then contain a summand $V(d)$ with $d\ge 2$, whence $k\ge 2$. Then the hypotheses of Lemma \ref{lSpFree} are
satisfied for the $P$-quotient representation $V/W$ of $V$ whose semisimplification consists of $Q_0, \dots , Q_k$ (the
nonnegative portion). Moreover $Q_{-k} = Q_k$ is an R-representation, hence, by Remark \ref{rERepresentationsSp}, has dimension
bigger or equal to $\dim \mathrm{Sp}_{2n-2} (\CC )$; $Q_{-k}$ is contained in $Q_{-k+1}\otimes \CC^{2n-2}$, hence the dimension
of $Q_{-k+1}$ cannot be smaller than $\dim R_u (P) = 2n-2 +1$: otherwise it would be a sum of $\CC$'s and $\CC^{2n-2}$, and
$Q_{-k}$ would be an E-representation. Hence $\dim W \ge \dim P$.
\end{proof}

\begin{proposition}\xlabel{pSpE}
Theorem \ref{tSp} is true if $V=\Sigma^{\lambda }_0 \CC^{2n}$ where one of the following holds
\begin{gather*}
(1)\quad\lambda = (a, 0, \dots , 0), \; a\ge 3 \, ,\\
(2)\quad\lambda = (b, 1, 0,  \dots , 0), \; b\ge 2 \, , \\
(3)\quad\lambda = (c, 2, 0, \dots , 0), \; c\ge 2\, , \\
(4)\quad\lambda = (d, 1,1,0, \dots , 0),  \; d\ge 2\, .
\end{gather*}
It is also true for 
\begin{gather*}
\lambda = (e, 1,1,1), \; e\ge 2, \quad \lambda = (f, 1,1,1,1 ), \; f\ge 2\, .
\end{gather*}
\end{proposition}

\begin{proof}
We will filter $V$ as a $P$-module and use the notation from the proof of Proposition \ref{pSpR}:
\[
V_{-k} \subset \dots \subset V_0 \subset \dots \subset V_{k} = V, \;  Q_j = V_{j} /V_{j-1}, \; Q_j \simeq Q_{-j}\, , \; k\ge 1\, .
\]
Then we can compute the filtration, using Proposition \ref{pStructurePmodules}, in each of the cases (1)-(4) above.\\
In case (1) the integer $k$ is $\ge 3$ and we will have
\[
Q_{k-3} = (3)+(1) , \; Q_{k-2} = (2)+(0) , \; Q_{k-1} = (1) , \; Q_k = (0) \, .
\]
(here we write just the $\mu$'s and omit zeroes at the end to simplify the notation).\\
In case (2) we also have $k\ge 2$ and 
\[
Q_{k-2} = (3) + (2,1) + 2\times (1,0)  , \; Q_{k-1} = (2) + (1,1) + (0)  , \; Q_k = (1)
\]
(here the $(3)$ in $Q_{k-2}$ does not show up if $\lambda = (2,1,0, \dots , 0)$).\\
In case (3) we also get $k \ge 2$ and
\begin{eqnarray*}
Q_{k-2} =& (3,1) + (2,2) + 2\times (2) + (1,1) + (0)   ,\\
Q_{k-1} =& (3) + (2,1) + (1)  ,\\ 
Q_k =& (2)
\end{eqnarray*}
(where in the case $\lambda = (2,2,0,\dots , 0)$ all entries containing a $3$ must be omitted, and one copy of $(2)$ in $Q_{k-2}$).\\
In case (4) we get $k \ge 2$ as well and
\begin{eqnarray*}
Q_{k-2}=& (3,1) + (2,1,1) + (2) + 2\times (1,1),\\ Q_{k-1}=& (2,1) + (1,1,1) + (1),\\ Q_k =& (1,1)\, 
\end{eqnarray*}
(where the case $\lambda =(2,1,1,0 \dots , 0)$ again presents a deviation from the general pattern inasmuch as the entry $(3,1)$
is not present).

\begin{quote}
\textbf{Claim}: The $P$-quotient modules $V/W$ whose semisimplifications are the ones corresponding to the $Q$'s above are
$P$-generically free in all cases (1)-(4).
\end{quote}
Assuming the Claim for the moment, the assertion of the Proposition follows since the fibre $W$ always has dimension
$\ge \mathrm{Sp}_{2n-2} (\CC ) + \dim R_u (P)$: it contains $\mathrm{Sym}^2 \CC^{2n-2}$, $\CC^{2n-2}$ and $\CC$ in cases (1),
(2), and $(3)+(2)+(1)$ in case (3), and $(2,1)+(1,1)+(1)$ in case (4).\\
To prove the Claim in case (1) note that here the $P$-module $V/W$ has an $\mathrm{Sp}_{2n-2}(\CC )$-section which is
$(3) + (2) + (1) + (0)$, hence is generically free.\\
By \cite{Fu-Ha}, formula (25.39), we have for the restrictions of representations to $\mathrm{Sp}_m (\CC )$ from $\mathrm{SL}_m
(\CC )$:
\begin{gather*}
\Lambda^2 \CC^m = \Lambda^2_0 \CC^m + \CC, \; \Sigma^{2,1}\CC^m = \Sigma^{2,1}_0 \CC^m + \CC^m, \; \mathrm{Sym}^2 \CC^m =
\Sigma^{2,0,\dots , 0}_0 \CC^m ,\\
\Sigma^{2,2} \CC^m = \Sigma^{2,2}_0 \CC^m + \Lambda^2_0 \CC^m + \CC \, .
\end{gather*}
In cases (2) and (3) we thus see that, if we first divide by $\mathbb{G}_a$, the corresponding $P/\mathbb{G}_a$-space has a
quotient which is the restriction to $P/\mathbb{G}_a$, embedded naturally into $\mathrm{ASL}_{2n-2}(\CC ) = \mathrm{SL}_{2n-2}
(\CC )\ltimes \CC^{2n-2}$, of the following $\mathrm{ASL}_{2n-2}(\CC )$-representations: in case (2) a three-step extension with
quotients
\[
\Sigma^{2,1} \CC^{2n-2}, \quad \Sigma^{2,0}\CC^{2n-2} + \Lambda^2 \CC^{2n-2} , \quad \CC^{2n-2}
\]
in case (3) a three-step extension with quotients 
\[
\Sigma^{2,2} \CC^{2n-2}, \quad \Sigma^{2,1}\CC^{2n-2} , \quad \Sigma^2 \CC^{2n-2} \, .
\]
The first of these is a section for $\mathrm{ASL}_{2n-2}(\CC )$ in the $\mathrm{SL}_{2n-1}(\CC )$-representation
\[
\Sigma^{2,1}\CC^{2n-1} \times \CC^{2n-1}\, .
\]
For the second the same holds with $\Sigma^{2,2}\CC^{2n-1} \times \CC^{2n-1}$ instead. The latter two are generically free for
$\mathrm{SL}_{2n-1} (\CC )$.\\
In case (4), after dividing by the $\mathbb{G}_a$-action, we get a $P/\mathbb{G}_a$-space which has a quotient which is the
representation $R$ for $\mathrm{Sp}_{2n-2} (\CC )\ltimes \CC^{2n-2}$ 
\[
0\to \Lambda^3_0 \CC^{2n-2} \to R \to \Lambda^2_0 \CC^{2n-2} \to 0 \, .
\]
As in the proof of Proposition \ref{pLambda3} we see that $R$ is an $\mathrm{Sp}_{2n-2} (\CC )\ltimes \CC^{2n-2}$-section in the
restriction of the $\mathrm{SL}_{2n} (\CC )$-representation
\[
\Lambda^3 \CC^{2n} 
\]
to $\mathrm{ASL}_{2n-2}(\CC )$. This restriction is generically free, completing the proof for cases (1)-(4).

\

In the two exceptional cases we get $P$-module filtrations with
\begin{eqnarray*}
Q_{k-2} =& (2,1,0) + 2\times (1,1,1) + \dots ,\\ 
Q_{k-1}=& (2,1,1) + (1,1,0) ,\\
Q_{k} =& (1,1,1)\, ,
\end{eqnarray*}
in the case $\lambda = (e,1,1,1)$ and
\begin{eqnarray*}
Q_{k-2} =& (2,1,1,0) + 2\times (1,1,1,1) + \dots ,\\ 
Q_{k-1}=& (2,1,1,1) + (1,1,1,0) ,\\
Q_{k} =& (1,1,1,1)\, ,
\end{eqnarray*}
in the case $\lambda = (f, 1,1,1,1)$. We will establish the Claim above also in these cases. It is immediate to check that the
dimension of the corresponding $W$ is big enough in these cases. Hence it suffices to establish the generic freeness for
$\mathrm{Sp}_6 (\CC )\ltimes \CC^6$ of the extension
\[
0 \to \Lambda^2_0 \CC^6 + \Sigma^{2,1,1}_0 \CC^6  \to E_1 \to \Lambda^3_0 \CC^6 \to 0
\] 
and the generic freeness for $\mathrm{Sp}_8 (\CC )\ltimes \CC^8$ of the extension
\[
0 \to \Lambda^3_0 \CC^8 + \Sigma^{2,1,1,1}_0 \CC^8 \to E_2 \to \Lambda^4_0 \CC^8 \to 0 \, .
\]
The stabilizer in general position in the first case is $\mathrm{SL}_3 (\CC )$, in the second case it is $( \ZZ/2\ZZ )^7$. In the first case (cf. also the table in theorem 1, p.197ff. in \cite{Po80}), a general element in $\Lambda^3_0 \CC^6$ is \[ a e_1\wedge e_2 \wedge e_3 +  b f_1 \wedge f_2\wedge f_3\] 
(remember that $e_1, e_2, e_3$ resp. $f_1, f_2, f_3$ span maximal isotropic subspaces) with stabilizer 
\[
\left( \begin{array}{cc}  g & 0 \\ 0 & (g^t)^{-1} \end{array}\right)
\]
$g\in\mathrm{SL}_3 (\CC )$ (where we take the ordered basis $(e_1, e_2, e_3, f_1, f_2, f_3)$ here), so that with respect to $\mathrm{SL}_3 (\CC )$ the standard representation $\CC^6$ decomposes as $\CC^6 = \CC^3 + (\CC^3)^{\vee }$. As
\[
\Lambda^2_0 \CC^6 + \Sigma^{2,1,1}_0 \CC^6 = \Lambda^3_0 \CC^6 \otimes \CC^6
\]
(the left-hand side is contained in the right hand side, so it suffices to check dimensions), we get first
\[
\Lambda^3_0 \CC^6 = \CC + \CC + V(0, 2)  + V(2, 0)
\]
and hence we see that the decomposition with respect to $\mathrm{SL}_3 (\CC )$ of $\Lambda^2_0 \CC^6 + \Sigma^{2,1,1}_0 \CC^6 $ contains, complementary to $\CC^3 + (\CC^3)^{\vee }$, also 
\[
\mathrm{Sym}^3 \CC^3  + \CC^3\, .
\]
This is generically free for $\mathrm{SL}_3 (\CC )$ whence we conclude by remarking that the translations $\CC^6$ also act generically freely on the extension $E_1$ above (this follows also from the proof of Lemma \ref{lSpFree}).

\

It remains to consider the extension $E_2$ above. The basic reference is \cite{An}, but see also \cite{Katan}, \cite{Vin76}. The situation here can be described as follows. By the general theory of \cite{Vin76}, the Lie algebras $\mathfrak{g}$ of $E_7$ and $E_6$ are $\ZZ /2\ZZ$-graded, $\mathfrak{g} = \mathfrak{g}_0 \oplus \mathfrak{g}_1$, with
\begin{gather*}
E_7 = \mathfrak{sl}_8 \oplus \Lambda^4 \CC^8 , \\
E_6 = \mathfrak{sp}_8 \oplus \Lambda^4_0 \CC^8
\end{gather*}
and the action of $\mathfrak{g}_0$ on $\mathfrak{g}_1$ corresponds to two representations with a nontrivial stabilizer in general position which we have to consider together here: $\mathrm{SL}_8 (\CC )$ acting on $\Lambda^4 \CC^8$ and $\mathrm{Sp}_8 (\CC )$ acting on $\Lambda^4_0 \CC^8$. Above, under the natural embedding of the single summands $\mathfrak{sp}_8 \subset  \mathfrak{sl}_8$ and $\Lambda^4_0 \CC^8 \subset \Lambda^4 \CC^8$, $E_6$ is a subalgebra of $E_7$. By \cite{Vin76}, a Cartan algebra $\mathfrak{h} \subset \Lambda^4 \CC^8$ resp. $\mathfrak{h}' \subset \Lambda^4_0 \CC^8$ gives a section for the action of $\mathrm{SL}_8 (\CC )$ resp. $\mathrm{Sp}_8 (\CC )$ in these spaces. Note $\dim \mathfrak{h} =7$, $\dim \mathfrak{h}' =6$. We have to understand the stabilizers in general position explicitly, and in \cite{An} an explicit form of $\mathfrak{h}$ is given (\cite{An}, p. 149): \\
To achieve notational consistency with \cite{An}, let us take a basis $e_1, \dots , e_8$ of $\CC^8$ here such that for the symplectic structure given by $\omega$ we have: $\omega (e_i, e_{4+j})= \delta_{ij}$, $i, \: j \in \{ 1, \dots , 4\}$. For a permutation $(i_1\:  j_1\:  k_1\:  l_1\:  i_2\:  j_2\:  k_2\:  l_2)$ of $1, \dots , 8$ we use the abbreviation
\begin{gather*}
(i_1\:  j_1\:  k_1\:  l_1\:\mid \:  i_2\:  j_2\:  k_2\:  l_2): = e_{i_1}\wedge e_{j_1} \wedge e_{k_1} \wedge e_{l_1} + e_{i_2}\wedge e_{j_2} \wedge e_{k_2} \wedge e_{l_2}\, .
\end{gather*}
In \cite{An} it is noted that if for $(i_1\:  j_1\:  k_1\:  l_1\:\mid \:  i_2\:  j_2\:  k_2\:  l_2)$ and $(i_1' \:  j_1' \:  k_1' \:  l_1' \:\mid \:  i_2' \:  j_2' \:  k_2' \:  l_2' )$ the sets $i_1\:  j_1\:  k_1\:  l_1$ and $i_1' \:  j_1' \:  k_1' \:  l_1'$ have precisely two elements in common, then the corresponding elements in $\Lambda^4\CC^8 \subset E_7$ commute, and that hence the following seven pairwise commuting elements form a basis of a Cartan subalgebra $\mathfrak{h}$:
\begin{align*}
s_1 &= (1234\:\mid \:  5678), \quad 
s_2 &= (1357\:\mid \:  6824), \quad 
s_3 &= (1562\:\mid \:  8437), \\
s_4 &= (1683\:\mid \:  4752), \quad
s_5 &= (1845\:\mid \:  7263), \quad
s_6 &= (1476\:\mid \:  2385), \\
s_7 &= (1728\:\mid \:  3546) .
\end{align*}
According to the table given in \cite{Po-Vi}, the stabilizer in general position of $\mathrm{SL}_8 (\CC )$ in $\Lambda^4 \CC^8$ is -modulo the ineffectivity kernel of the action- given by $(\ZZ/2 \ZZ )^6$. We can now explicitly locate this $( \ZZ /2\ZZ )^6$ inside $\mathrm{SL}_8 (\CC )$: 
\begin{itemize}
\item
For each two element subset $i_0, \: j_0 \in \{ 1, \dots , 4 \}$ the substitution mapping $e_{i_0} \mapsto - e_{i_0}$, $e_{j_0} \mapsto -e_{j_0}$ and $e_{4+ i_0} \mapsto - e_{4+i_0}$, $e_{4+j_0 } \mapsto - e_{4+j_0}$, and keeping all other $e_j$ fixed, is in the stabilizer in general position. These substitutions generate a $(\ZZ /2\ZZ )^2$ (remark again: counted modulo the ineffectivity kernel). Moreover, we may also take
\[
e_i \mapsto - e_i, \: i=1, \dots , 4, \quad e_j \mapsto e_j, \: j = 5, \dots , 8
\]
which is another $\ZZ /2\ZZ$. 
\item
The permutations in $\mathfrak{S}_8$ of the form
\[
(12)(34)(56)(78), \; (13)(24)(57)(68)
\]
generate a Klein group $(\ZZ /2\ZZ )^2$ and induce maps which are in the stabilizer in general position. Likewise, the permutation
\[
(15)(26)(37)(48)
\]
generates another stabilizing $\ZZ /2 \ZZ$.
\end{itemize}
Thus we have found the appropriate $( \ZZ /2\ZZ )^6$ in this case. We will now use this knowledge to describe the stabilizer in general position for $\mathrm{Sp}_8 (\CC )$ on $\Lambda^4_0 \CC^8$ explicitly. First, we need to find an appropriate Cartan subalgebra $\mathfrak{h}' \subset \mathfrak{h}$. Note that a four-vector $\xi\in \Lambda^4 \CC^8$ will be in the space of traceless four-vectors $\Lambda^4_0 \CC^8$ if $\Lambda (\xi ) =0$ where $\Lambda$ is defined quite generally on pure $N$-vectors  $v_1 \wedge \dots \wedge v_N$ by
\begin{gather*}
\Lambda (v_1 \wedge\dots \wedge v_N ) = \sum_{i< j} (-1)^{i+j -1} \omega (v_i,\: v_j)v_1\wedge \dots \wedge \hat{v}_i \wedge \dots \wedge \hat{v}_j \wedge \dots \wedge v_N \, .
\end{gather*}
Thus we see that the following elements are in the kernel of $\Lambda$:
\[
s_5 - s_2,\:  s_3 -s_2,  \: s_1, \: s_4, \: s_6, \: s_7 \, ,
\]
hence they generate a Cartan subalgebra $\mathfrak{h}'$ (here we use that under the natural embeddings of the summands in $E_6$ into those of $E_7$ the bracket is preserved). By the table of  \cite{Po-Vi}, the stabilizer in general position for the action of $\mathrm{Sp}_8 (\CC )$ on $\Lambda^4_0 \CC^8$ is -modulo the ineffectivity kernel of the action- equal to $( \ZZ /2\ZZ )^6$. It is now easy to write down what it is explicitly, and also to see what we have to adjust in comparison to $\mathrm{SL}_8 (\CC )$. We have:
\begin{itemize}
\item
For each two element subset $i_0, \: j_0 \in \{ 1, \dots , 4 \}$ the substitution mapping $e_{i_0} \mapsto - e_{i_0}$, $e_{j_0} \mapsto -e_{j_0}$ and $e_{4+ i_0} \mapsto - e_{4+i_0}$, $e_{j_0 +4} \mapsto - e_{4+j_0}$, and keeping all other $e_j$ fixed, giving a $(\ZZ /2\ZZ )^2$, as above. Moreover, we may take in the symplectic case 
\[
e_i \mapsto - \sqrt{-1} e_i, \: i=1, \dots , 4, \quad e_j \mapsto \sqrt{-1} e_j, \: j = 5, \dots , 8
\]
which is another $\ZZ /2\ZZ$. 

\item
The permutations in $\mathfrak{S}_8$ of the form
\[
(12)(34)(56)(78), \; (13)(24)(57)(68)
\]
yield again a $( \ZZ /2 \ZZ )^2$ in the stabilizer. However, to get a symplectic transformation, we must now combine the two maps $(15)(26)(37)(48)$ and $e_i \mapsto - e_i, \: i=1, \dots , 4, \quad e_j \mapsto e_j, \: j = 5, \dots , 8$ into one: we take their composite which gives the remaining $\ZZ /2\ZZ$.
\end{itemize}

\

After these somewhat lengthy, but necessary preliminary considerations we can now address our question whether the extension
 \[
0 \to \Lambda^3_0 \CC^8 + \Sigma^{2,1,1,1}_0 \CC^8 \to E_2 \to \Lambda^4_0 \CC^8 \to 0 \, .
\]
is generically free for $\mathrm{Sp}_8 (\CC )\ltimes \CC^8$. In fact, we will show that this is already true for
\[
0 \to \Lambda^3_0 \CC^8  \to E'_2 \to \Lambda^4_0 \CC^8 \to 0 \, .
\]
The translations $\CC^8$ act generically freely again on $E_2'$ (e.g. use the proof of Lemma \ref{lSpFree}), so we proceed as before, find the decomposition of $\Lambda^8_0 \CC^8$ as an $H$ module where $H$ is the stabilizer in general position, so the above $( \ZZ /2\ZZ )^6$ extended by the center $\ZZ /2\ZZ $. Then $\Lambda^8_0 \CC^8$  splits as $\CC^8 +$(remainder), and we show that $H$ acts generically freely in the remainder. It will be sufficient to find the weight space decomposition of $\Lambda^3_0 \CC^8$ and to understand how the weights are grouped into Weyl group orbits as $H$ can be seen as a subgroup of the Weyl group of $\mathrm{Sp}_8 (\CC )$ extended by a maximal torus $T$. We have
\[
\Lambda^3_0 \CC^8 = \bigoplus_{\chi } V_{\chi } \oplus \bigoplus_{\chi'} V_{\chi'}
\]
where $\chi$ runs over all weights $(a,b,c, d)$ which have precisely one zero entry, and the other three entries are $+1$ or $-1$. They have multiplicity $1$, so $\dim \bigoplus_{\chi } V_{\chi }=32$. On the other hand, $\chi'$ runs over weights with precisely one nonzero entry $+1$ or $-1$, which have multiplicity $2$, so $\dim \bigoplus_{\chi'} V_{\chi'} = 16$. As a check note that $\dim \Lambda^3_0 \CC^8 = {8 \choose 3} - {8 \choose 1}= 48$. It will be convenient now to revert to our notational convention made at the beginning of this section, and use the basis $e_1, f_1, \dots , e_n, f_n$, $\omega (e_i, f_j ) = \delta_{ij}$, in $\CC^8$. It is obvious that the vectors
\begin{gather*}
x_i \wedge x_j \wedge x_k , \quad x \in \{ e, \: f \} , \;  1 \le i < j< k \le 4
\end{gather*}
are annihilated by $\Lambda$, so are in $\Lambda^3_0 \CC^8$, and form a basis of $\bigoplus_{\chi } V_{\chi }$. Likewise, the vectors
\begin{gather*}
x_i \wedge e_j \wedge f_j - x_i \wedge e_k \wedge f_k, \quad x \in \{ e, \: f \} , \; \\  1 \le i  \le 4, \: j, k\in \{ 1, \dots , 4\} -\{ i\}, \; j\neq k 
\end{gather*}
are annihilated by $\Lambda$ and generate $\bigoplus_{\chi'} V_{\chi'}$ (and are a basis if e.g. we require that $j$ is the smallest element in the set $\{ 1, \dots , 4\} -\{ i\}$ and $k$ any one of the remaining two elements; this also explains multiplicity $2$). The space $\CC^8$ has also weights $\chi'$ as above and a decomposition for $N(T)$ (the normalizer of a maximal torus in $\mathrm{Sp}_8 (\CC )$)
\[
\CC^8 = \bigoplus_{\chi' } \tilde{V}_{\chi'}\, ,
\]
but the $\tilde{V}_{\chi'Ê}$ are one-dimensional. In fact, $\bigoplus V_{\chi'}$ is not just two copies of $\CC^8$ as $N(T)$-module, but this is true as $H$-module: the submodule spanned by
\begin{gather*}
x_1 \wedge (e_2 \wedge f_2) - x_1 \wedge (e_3\wedge f_3), \quad x_2 \wedge (e_1 \wedge f_1) -  x_2 \wedge (e_4\wedge f_4), \\
x_3 \wedge (e_4 \wedge f_4) - x_3 \wedge (e_1\wedge f_1), \quad x_4 \wedge (e_3 \wedge f_3) -  x_4 \wedge (e_2\wedge f_2)
\end{gather*}
is $H$-invariant where $x\in \{ e, \: f\}$, but not $N(T)$-invariant. Thus we see that it will be sufficient to show that the $32$-dimensional $N(T)$-summand $\bigoplus V_{\chi }$ of $\Lambda^3_0 \CC^8$ consisting of the multiplicity $1$ weight spaces is generically free, i.e. faithful, for $H$. We can see this in two steps: the part of  $H$ generated by the elements in $T$  (corresponding to the maps given in the first item above) acts generically freely. Together with the center this is a $( \ZZ / 2\ZZ )^4$. A way to see this immediately is to note that: (1) The action of the $(\ZZ / 2\ZZ )^3$ generated by all even sign changes among the $e_i$ and corresponding sign changes in the $f_i$ is generically free: it is generated by
\begin{gather*}
\iota_1\, : \, e_1\mapsto -e_1, \; e_2 \mapsto -e_2 \, , \;  f_1\mapsto -f_1, \; f_2 \mapsto -f_2\\
\iota_2 \, :\, e_1\mapsto -e_1, \; e_3 \mapsto -e_3 \, , \; f_1\mapsto -f_1, \; f_3 \mapsto -f_3 \\
\iota_3\, :\, e_1\mapsto -e_1, \; e_4 \mapsto -e_4 , \;  f_1\mapsto -f_1, \; f_4 \mapsto -f_4\, ,
\end{gather*}
where the $e_i$ and $f_j$ which do not show in the line of one $\iota_k$ above are fixed under the respective $\iota_k$. But now there is always a $e_{i_1}\wedge e_{i_2} \wedge e_{i_3}$ which is fixed under $\iota_j$ and $\iota_k$, but mapped to its negative under $\iota_l$, $\{ j, k, l\} = \{ 1,2,3\}$; hence the assertion. And (2): the element $\mathrm{diag} (\sqrt{-1}, \dots , \sqrt{-1}  , -\sqrt{-1}, \dots , -\sqrt{-1})$ acts generically freely on these $( \ZZ /2 \ZZ )^3$ orbits. 
\\
 Secondly, the group $(\ZZ /2\ZZ )^3$ generated by the permutations in $\mathfrak{S}_8$ given above, acts on the weights $(a, b, c, d)$ as the Klein four-group permuting $a,b,c,d$ and an overall sign change $(a,b,c,d) \mapsto (-a, -b, -c, -d)$. This has no ineffectivity kernel for the weights in the Weyl orbit corresponding to the $\chi$. Thus the action of $(\ZZ /2\ZZ)^3$ on the orbits of $(\ZZ / 2\ZZ )^4$ in $\bigoplus V_{\chi }$ is also faithful, which was to be shown.
\end{proof}

To complete the proof of Theorem \ref{tSp} it suffices to show the next two propositions.

\begin{proposition}\xlabel{pExteriorPowerSp}
Theorem \ref{tSp} is true if $V=\Sigma^{1, \dots , 1, 0, \dots , 0}_0 \CC^{2n}$ is an exterior power other than 
\[
\Lambda^5_0 \CC^{10} \, .
\]
\end{proposition}

\begin{proof}
Suppose first that $V =\Lambda^k_0 \CC^{2n}$, $k\ge 3$, is an exterior power with $k\neq n$. Then it has a filtration as
$P$-module which has quotients
\begin{gather*}
Q_{-1} = \Lambda^{k-1}_0 \CC^{2n-2}, \; Q_0 = \Lambda^{k}_0 \CC^{2n-2} + \Lambda^{k-2}_0 \CC^{2n-2} , \; Q_1 = \Lambda^{k-1}_0
\CC^{2n-2} \, .
\end{gather*}
We distinguish two cases: if $k=3$, then $V/P$ is generically a vector bundle over the quotient of 
\[
0\to \Lambda^3_0 \CC^{2n-2} \to R \to \Lambda^2_0 \CC^{2n-2} \to 0
\]
by $\mathrm{Sp}_{2n-2}(\CC ) \ltimes \CC^{2n-2}$. The proof of the rationality of this quotient is completely analogous to the
one given in Proposition \ref{pLambda3} for $2n-2 =8$ (just change $4$ to $n-1$ in the argument and it goes through).\\
If $k> 3$, then it suffices to prove two facts:
\begin{itemize}
\item[(1)]
$\dim (\Lambda^{k-1}_0 \CC^{2n-2} + \Lambda^{k-2}_0 \CC^{2n-2}) \ge \dim P\, .$ 
\item[(2)]
The extension
\[
0\to \Lambda^{k}_0 \CC^{2n-2} \to R \to \Lambda^{k-1}_0 \CC^{2n-2} \to 0
\]
is generically free for $P/\mathbb{G}_a  = \mathrm{Sp}_{2n-2}(\CC )\ltimes \CC^{2n-2}$.
\end{itemize}
The proof of (1) is immediate: for $k\ge 5$ both summands are R-representations whence the assertion. For 
$k=4$,
$\Lambda^{k-1}_0 \CC^{2n-2}$ is an R-representation, hence has dimension $\ge \mathrm{Sp}_{2n-2}(\CC )$, and $\Lambda^{2}_0
\CC^{2n-2}$ has dimension $\ge 2n-2+1 =\dim R_u (P)$.\\ 
Note that (2) is an immediate consequence of (the
proof of) Lemma
\ref{lSpFree}.\\
Now if $V = \Lambda^n_0 \CC^{2n}$, $n\ge 5$, then as $P$-module $V$ has quotients
\begin{gather*}
Q_{-1} = \Lambda^{n-1}_0 \CC^{2n-2}, \; Q_0 = \Lambda^{n-2}_0 \CC^{2n-2} , \; Q_1 = \Lambda^{n-1}_0 \CC^{2n-2} \, .
\end{gather*}
In this case 
\[
\dim \Lambda^{m}_0 \CC^{2m} = {2m \choose m} - {2m \choose m-2 } \ge \dim P = 2m^2 + m +2m +1 
\]
unless $m=4$ when $\dim \Lambda^4_0 \CC^8 = 42$, $\dim P = 45$. Also in the other cases $\Lambda^{n-1}_0 \CC^{2n-2}$ is an
R-representation.
\end{proof}

\begin{proposition}\xlabel{pExceptionSp}
The quotient $\Lambda^5_0 \CC^{10} /P$ is rational.
\end{proposition}

\begin{proof}
In this case we get a filtration
\begin{gather*}
Q_{-1} = \Lambda^{4}_0 \CC^{8}, \; Q_0 = \Lambda^{3}_0 \CC^{8} , \; Q_1 = \Lambda^{4}_0 \CC^{8} \, .
\end{gather*}
Dividing by $\mathbb{G}_a$ first, we get a vector bundle of rank $41$ over the $P/\mathbb{G}_a= \mathrm{Sp}_8 (\CC
) \ltimes \CC^8$-representation $E$ with quotients
\[
\Lambda^{3}_0 \CC^{8} , \; \Lambda^{4}_0 \CC^{8}
\]
which is generically free. We claim that the quotient of the latter by $\mathrm{Sp}_8 (\CC
) \ltimes \CC^8$ is stably rational of level $8+9=17$: consider $E+ E'$
where
$E'$ is the
$9$-dimensional extension
\[
0 \to \CC^8 \to E' \to \CC \to 0 \, .
\]
Taking a section we reduce to a rationality question for the quotient of 
\[
\Lambda^{3}_0 \CC^{8} + \Lambda^{4}_0 \CC^{8} + \CC
\]
by the group $\mathrm{Sp}_8 (\CC )$. But we know already that $\Lambda^{3}_0 \CC^{8}$ is generically free for that group and that 
$\Lambda^{3}_0 \CC^{8}/\mathrm{Sp}_8 (\CC )$ is stably rational of level $8$.
\end{proof}

\section{The case of orthogonal groups}\xlabel{sOrthogonalGroups}

Let $\mathrm{SO}_N (\CC )$ resp. $\mathrm{O}_N (\CC )$ be the special resp. full orthogonal group associated to a nondegenerate
symmetric bilinear form $B$. If $e_1, \dots , e_N$ is the standard basis in $\CC^N$, and $x= (x_1, \dots , x_N )$, $y=(y_1,
\dots , y_N) \in \CC^N$, it will be convenient to take
\[
B(x, \: y) = \sum_{i=1}^{N}  x_i  y_{N+1 -i}
\]
for some computations. Then in the case $N=2n$, we put $e_{-i}:= e_{2n+1 -i}$, $i= 1,\dots , n$, so that $e_{\pm i}$, $i=1,
\dots , n$ is a $B$-isotropic basis, $B(e_i, e_j) = \delta_{i, -j}$. In the case $N = 2n+1$ we put 
\[
e_{-i}:= e_{2n+2 -i}, \; i= 1,\dots , n, \quad e_0 := e_{n +1} \, ,
\]
so that $B(e_i, e_j ) =\delta_{i, -j}$ for the basis $e_0, \: e_{\pm i}$.\\
In the case $N = 2n +1$, the element $-I$ (the negative of the identity), is in $\mathrm{O}_N (\CC )$, but not in
$\mathrm{SO}_{N} (\CC )$, and
\[
\mathrm{O}_{2n+1} (\CC ) = \mathrm{SO}_{2n+1} (\CC ) \cup (-I) (\mathrm{SO}_{2n+1} (\CC )) \, .
\]
For $N = 2n$, one defines an element $g_0 \in \mathrm{O}_{2n} (\CC )$ by
\[
g_0 (e_n ) = e_{-n}, \: g_0 (e_{-n}) = e_n, \quad g_0 (e_{\pm i}) = e_{\pm i}, \; i \neq n \, .
\]
The $g_0$ is not in $\mathrm{SO}_{2n} (\CC )$ and
\[
\mathrm{O}_{2n} (\CC ) = \mathrm{SO}_{2n} (\CC ) \cup g_0 (\mathrm{SO}_{2n} (\CC ))\, .
\]

\

We begin by recalling the aspects of the representation theory for these groups that we will need. The irreducible
representations of $\mathrm{SO}_N (\CC )$ can be labelled by sequences $\lambda = ( \lambda_1, \dots , \lambda_n )$ of integers
satisfying $\lambda_1 \ge \dots \ge \lambda_n\ge 0$ if $N = 2n+1$ is odd, and satisfying $\lambda_1 \ge \dots \ge
\lambda_{n-1} \ge | \lambda_n |$ if $N = 2n$ is even. These are the dominant integral weights. We denote them by
\[
\Sigma^{\lambdaÊ}_0 \CC^N
\]  
in both cases. If $N=2n$ is even, the representations associated to $(\lambda_1,
\dots ,
\lambda_{n-1}, -\lambda_n )$ is obtained as follows: one precomposes the
representation associated to
$(\lambda_1,
\dots , \lambda_n )$ with the outer automorphism
\[
g_0 () g_0^{-1} \, :\, \mathrm{SO}_{2n} (\CC ) \to \mathrm{SO}_{2n} (\CC )
\]
obtained by conjugating by $g_0$. 

\

We will also need the picture of the irreducible representations of the full orthogonal groups $\mathrm{O}_N (\CC )$. For
$N=2n+1$, the situation is easy since $\mathrm{O}_N (\CC ) = \mathrm{SO}_N (\CC ) \times \ZZ / 2\ZZ = \mathrm{SO}_N (\CC )
\times \{ \pm I
\}$ (direct product). We can extend each irreducible representation $\Sigma^{\lambda }_0 (\CC^N )$ in two ways to $\mathrm{O}_N
(\CC )$:
\[
\Sigma^{\lambda }_0 (\CC^N )^+ , \quad \Sigma^{\lambda }_0 (\CC^N )^-
\]
where $-I$ acts either as multiplication by $+1$ or $-1$. We can thus label them as $(\lambda , \epsilon )$, $\epsilon = \pm
1$.\\
The case of $\mathrm{O}_N (\CC )$ with $N = 2n$ is more complicated due to the fact $\mathrm{O}_{2n} (\CC ) = \mathrm{SO}_{2n }
(\CC ) \rtimes (\ZZ /2\ZZ ) = \mathrm{SO}_{2n } (\CC ) \rtimes (g_0 )$ (semidirect product). For an irreducible representation
$(\varrho_{\lambda } , \: V=\Sigma^{\lambda }_0 \CC^N)$ of $\mathrm{SO}_{2n} (\CC )$ we can first form the induced
representation
\[
\mathrm{Ind} (V) = \mathrm{Ind}^{\mathrm{O}_N}_{\mathrm{SO}_N} (V)
\]
consisting of those regular $V$-valued functions $f$ on $G$ satisfying $f(hg) = \varrho_{\lambda } (h) f(g)$ for all $h \in
\mathrm{SO}_N (\CC )$ and all $g\in\mathrm{O}_N (\CC )$. This $\mathrm{Ind} (V)$ becomes a $\mathrm{O}_N (\CC )$-module via the
right translation action and decomposes as $\mathrm{SO}_N (\CC )$-representation as 
\[\Sigma^{\lambda_1, \dots , \lambda_n } \CC^N \oplus \Sigma^{\lambda_1, \dots , -\lambda_n } \CC^N\] 
(the one summand are functions
supported on $\mathrm{SO}_n (\CC )$, the other functions supported on $g_0 \mathrm{SO}_N (\CC )$). Now there are two cases:
\begin{itemize}
\item[(1)]
$\lambda_n \neq 0$ (this amounts to $g_0 \cdot \lambda \neq \lambda$ for the action of $g_0$ on weights): then $\mathrm{Ind}
(V)$ is irreducible as a $\mathrm{O}_N (\CC )$-representation. We will say that this is the representation of $\mathrm{O}_N
(\CC )$ associated to $(\lambda , \emptyset )$ and write
\[
\Sigma^{\lambda }_0 (\CC^N )^{\emptyset }
\]
in this case for the irreducible $\mathrm{O}_N (\CC )$-representation.
\item[(2)]
$\lambda_n =0$ (or equivalently $g_0 \cdot \lambda = \lambda$): in this case, the $\mathrm{O}_N (\CC )$-representation
$\mathrm{Ind} (V)$ decomposes into two irreducible summands for $\mathrm{O}_N (\CC )$:
\[
\mathrm{Ind}^{\mathrm{O}_N}_{\mathrm{SO}_N} (V) = \Sigma^{\lambda }_0 (\CC^N )^{+ } \oplus \Sigma^{\lambda }_0 (\CC^N
)^{- }\, .
\]
In fact, as $\mathrm{SO}_N (\CC )$-representations
\begin{gather*}
\Sigma^{\lambda }_0 (\CC^N )^{+ } = \Sigma^{\lambda_1, \dots , 0 } \CC^N \, ,\\
\Sigma^{\lambda }_0 (\CC^N )^{- } = \Sigma^{\lambda_1, \dots , 0 } \CC^N
\end{gather*}
and $g_0$ acts on a line of highest weight vectors in these spaces as multiplication by $+1$ resp. $-1$. We will say that these are the representations associated
to $(\lambda , \pm )$ of $\mathrm{O}_{2n} (\CC )$.
\end{itemize}
The cases considered in (1) and (2) give \emph{all the irreducible representations of} $\mathrm{O}_{2n} (\CC )$. We can
accordingly index them by $(\lambda_1, \dots , \lambda_{n}, \emptyset )$, $\lambda_1 \ge \dots \ge \lambda_n > 0$, resp.
$(\lambda_1, \dots , \lambda_{n}, \pm )$, $\lambda_1 \ge \dots \ge \lambda_n = 0$.

\

An explicit tensor algebra construction of these representations goes as follows: as in the case of $\mathrm{Sp}_{2n} (\CC )$
treated in section \ref{sInductionProcessSp} we define \emph{harmonic tensors} $\mathcal{H} (\otimes^k \CC^N, \: B)\subset
\otimes^k
\CC^N$ (just replace the symplectic form 
$\omega$ by $B$ everywhere). Then (\cite{G-W}, Chapter 10) one has for $N=2n +1$, $\lambda$ as above, $k=|\lambda |$
\[
\Sigma^{\lambda }_0 \CC^{2n +1} = s(A) \left( \mathcal{H} (\otimes^k \CC^N, \: B) \right) 
\]
where $s(A)$ is the Young symmetrizer associated to a tableau of shape $\lambda$ as above. If $N =2n$ is even, then if
$\lambda_n =0$ 
\[
\Sigma^{\lambda }_0 \CC^{2n} = s(A) \left( \mathcal{H} (\otimes^k \CC^N, \: B) \right)\, .
\]
If $\lambda_n > 0$, then
\[
s(A) \left( \mathcal{H} (\otimes^k \CC^N, \: B) \right) = \mathrm{Ind}^{\mathrm{O}_N}_{\mathrm{SO}_N} (\Sigma^{\lambda }_0
\CC^N)
\]
as irreducible $\mathrm{O}_N (\CC )$-modules.

\

Below we have to know the irreducible $\mathrm{O}_{2n} (\CC )$-representations with a nontrivial stabilizer in general position, i.e. those for which this stabilizer differs from the 
ineffectivity kernel of the action. The irreducible $E$-representations for the group $\mathrm{SO}_{2n}(\CC )$ can be read off from the table in \cite{Po-Vi}, so we need to know that we do not pick up a nontrivial stabilizer in general position when passing to the $\ZZ/2\ZZ$-extension $\ZZ/2\ZZ \ltimes \mathrm{SO}_{2n} (\CC )$; more precisely, we need to know that a representation $\Sigma^{\lambda, \emptyset }_0 \CC^{2n}$ or $\Sigma^{\lambda, \pm}_0 \CC^{2n}$ of $\mathrm{O}_{2n} (\CC )$ which splits into R-representations for $\mathrm{SO}_{2n} (\CC )$, has the property that the stabilizer in general position in the larger group $\mathrm{O}_{2n} (\CC )$ also coincides with the ineffectivity kernel of the action. This is easy to see for the representations of type $(\lambda_1, \dots , \lambda_n, \emptyset  )$ above (e.g. look at vectors of different lengths in $\Sigma^{\lambda_1, \dots , \lambda_n} \CC^N$ and $\Sigma^{\lambda_1, \dots , -\lambda_n }\CC^N$, extending the quadratic form $q$ to these tensor spaces). For representations of type $(\lambda_1, \dots , \lambda_{n-1}, 0, \pm )$ it is surprisingly difficult to see this directly, but we will show it is true in Appendix \ref{AppendixOrthogonalGroup}; we felt that including the proof in the main text would have interrupted the main line of the argument of this section too much and would have been inconvenient for the reader.

\begin{remark}\xlabel{rSO}
The stabilizer of a generic line in the standard representation $\CC^N$ inside $\mathrm{SO}_{N} (\CC )$ or $\mathrm{O}_N (\CC
)$ is $\mathrm{O}_{N-1} (\CC )$ or $\ZZ /2 \times \mathrm{O}_{N-1} (\CC )$. Thus if $W$ is a generically free $\mathrm{SO}_N
(\CC )$-representation
\[
(W + \CC^N ) /\mathrm{SO}_{N} (\CC ) \simeq (W + \CC )/\mathrm{O}_{N-1} (\CC ) \, .
\]
Similarly, if $W$ is a generically free $\mathrm{O}_N (\CC )$-representation
\begin{gather*}
(W + \CC^N ) /\mathrm{O}_{N} (\CC ) \simeq (W + \CC )/\mathrm{O}_{N-1} (\CC )\times (\ZZ/2 )  \, \\
\simeq W/\mathrm{O}_{N-1} (\CC ) \times (\CC /\ZZ/2 ) \simeq W/\mathrm{O}_{N-1} (\CC ) \times \CC \, .
\end{gather*}
Thus we have to understand the corresponding branching laws.
\end{remark}

\

The branching $\mathrm{SO}_{2n +1} (\CC ) \to \mathrm{SO}_{2n } (\CC )$ is multiplicity-free and described by (\cite{G-W},
Chapter 8) 
\begin{gather}\label{fSODec1}
\Sigma^{\lambda }_0 \CC^{2n+1} \mid_{\mathrm{SO}_{2n } (\CC )} = \bigoplus_{\lambda_1 \ge \mu_1 \ge \dots \ge \lambda_{n-1} \ge
\mu_{n-1} \ge \lambda_n \ge |\mu_{n}|}
\Sigma^{\mu}_0
\CC^{2n}
\, .
\end{gather}
The branching $\mathrm{SO}_{2n} (\CC ) \to \mathrm{SO}_{2n -1} (\CC )$ is also multiplicity-free and described by
\begin{gather}\label{fSODec2}
\Sigma^{\lambda }_0 \CC^{2n} \mid_{\mathrm{SO}_{2n-1 } (\CC )} = \bigoplus_{\lambda_1 \ge \mu_1 \ge \dots \ge \lambda_{n-1} \ge
\mu_{n-1} \ge |\lambda_n |} \Sigma^{\mu}_0 \CC^{2n-1} \, .
\end{gather}

Because of Remark \ref{rSO} we should also understand some facts about these decompositions as $\mathrm{O}_{N-1} ( \CC )$-representations. The
following ones follow directly from what was explained above about the structure of the representations of the full orthogonal
group:
\begin{itemize}
\item
In formula \ref{fSODec1} the decomposition into irreducibles for the full orthogonal group $\mathrm{O}_{2n}(\CC )$ is
\[
 \bigoplus_{\mu , \mu_n > 0}
\left( \Sigma^{\mu_1, \dots , \mu_{n-1} , \mu_n }_0
\CC^{2n} + \Sigma^{\mu_1, \dots , \mu_{n-1}, -\mu_n}_0
\CC^{2n} \right)  + \bigoplus_{\mu , \mu_n = 0}
\Sigma^{\mu}_0
\, .
\]
The element $g_0$ in $\mathrm{O}_{2n}(\CC )$ acts on the highest weight vectors in the summands $\Sigma^{\mu}$ with $\mu_n =0$ via certain signs which we will have
no need of knowing precisely.
\item
Formula \ref{fSODec2} gives the decomposition into $\mathrm{O}_{2n-1} (\CC )$-irreducibles, and the element $(-I)$ in
$\mathrm{O}_{2n-1} (\CC )$ acts via a sign in the summands which depends on the parity of the number of boxes in $\mu$. \end{itemize}

\

Note that formulas \ref{fSODec1} and \ref{fSODec2} also give us how irreducible $\mathrm{O}_N (\CC )$-representations
decompose when restricted to $\mathrm{O}_{N-1} (\CC )$ (the embedding $\mathrm{O}_{N-1} (\CC ) \subset \mathrm{O}_N (\CC )$
coming from Remark \ref{rSO} as always here). We just have to use the classification of irreducible $\mathrm{O}_N (\CC
)$-representations in terms of irreducible $\mathrm{SO}_N (\CC )$-representations recalled above.

\

The notions of R- and E-representations extend to the $\mathrm{SO}_N (\CC )$-case without change. 

\begin{definition}\xlabel{dRReprOrthogonal}
We call an irreducible representation $V$ of the group $\mathrm{O}_N (\CC )$ an R-representation if its restriction to
$\mathrm{SO}_N (\CC )$ decomposes as a sum of R-representations for $\mathrm{SO}_N (\CC )$. If $V$ is not an R-representation
we say it is an E-representation.
\end{definition}

Thus for $N=2n$, the R-representations of $\mathrm{O}_N (\CC )$ are 
\[
\Sigma^{\lambda }_0 (\CC^N )^{\pm }, \: (\lambda_n =0), \quad \Sigma^{\lambda }_0 (\CC^N )^{\emptyset }= \Sigma^{\lambda_1,
\dots ,
\lambda_n } \CC^N \oplus
\Sigma^{\lambda_1, \dots , -\lambda_n } \CC^N
\]
where $\lambda$ gives an R-representation for $\mathrm{SO}_N (\CC )$; for $N =2n +1$, they are similarly
\[
\Sigma^{\lambda }_0 (\CC^N )^{\pm }\, .
\]

\begin{remark}\xlabel{rSOCase34}
Note the accidental isomorphisms
\[
\mathrm{SO}_6 (\CC ) = \mathrm{SL}_4 (\CC )/(\ZZ/2\ZZ), \: \mathrm{SO}_5 (\CC ) = \mathrm{Sp}_4 (\CC ) /(\ZZ /2\ZZ )\, ,
\]
and 
\[
\mathrm{SO}_4 ( \CC ) = (\mathrm{SL}_2 (\CC ) \times \mathrm{SL}_2 (\CC )) /(\pm (\mathrm{id}, \mathrm{id})), \; \mathrm{SO}_3
(\CC ) = \mathrm{PSL}_2 (\CC ) \, .
\]
Comparing corresponding fundamental weights shows that for $\mathrm{SO}_6 (\CC )$
\[
\Sigma^{\lambda_1, \lambda_2, \lambda_3}_0 \CC^6 = \Sigma^{\lambda_1+\lambda_2, \: \lambda_1 - \lambda_3, \: \lambda_2 -\lambda_3, \: 0} \CC^4 \, ,
\]
while for $\mathrm{SO}_5 (\CC )$
\[
\Sigma^{\lambda_1, \lambda_2 }_0 \CC^5 = \Sigma^{\lambda_2 +\lambda_1, \: \lambda_1 - \lambda_2}_0 \CC^4\, ;
\]
for $\mathrm{SO}_4 (\CC )$
\[
\Sigma^{\lambda_1, \lambda_2 }_0 \CC^4 = \mathrm{Sym}^{\lambda_1 + \lambda_2 } \CC^2 \otimes \mathrm{Sym}^{\lambda_1 -\lambda_2
} \CC^2
\]
and for $\mathrm{SO}_3 (\CC )$
\[
\Sigma^{\lambda_1}_0 \CC^3 = \mathrm{Sym}^{2\lambda_1 } \CC^2 \, .
\]
\end{remark}

\begin{remark}\xlabel{rStabilizerSO}
We need the following information on the irreducible E-representations of $\mathrm{SO}_N (\CC )$: for $N\ge 7$ they are
\[
\CC , \: \CC^N, \: \Sigma^{2,0, \dots ,0}_0 \CC^N = \mathrm{Sym}^2_0 \CC^N, \: \Sigma^{1,1,0, \dots ,0}_0 \CC^N
=
\Lambda^2 \CC^N \, .
\]
For $N=5, \: 6$ we have the accidental isomorphisms from Remark \ref{rSOCase34}. Checking the list in \cite{Po-Vi} we find that
in addition to the standard E-representations for $N\ge 7$ we get only the following additions 
\[
\Sigma^{(1,1,1)}_0 \CC^6 = \mathrm{Sym}^2 \CC^4, \: \Sigma^{(1,1,-1)}_0 \CC^6= \mathrm{Sym}^2 (\CC^4 )^{\vee }
\]
For $N=4$, following \cite{APopov78} and using Remark \ref{rSOCase34}, they are 
\begin{gather*}
\Sigma^{1, 0 }_0 \CC^4 = \CC^2 \otimes \CC^2 \, , \\
\Sigma^{c, c }_0 \CC^4 = \mathrm{Sym}^{2c} \CC^2 \otimes \CC, \:  \Sigma^{c, -c }_0 \CC^4 = \CC \otimes
\mathrm{Sym}^{2c} \CC^2, \: c=0, 1, 2\, \\
\Sigma^{2, 0 }_0 \CC^4 = \mathrm{Sym}^2 \CC^2 \otimes \mathrm{Sym}^2\CC^2 \; \mathrm{and}\\
\Sigma^{2, 1}_0 \CC^4 = \mathrm{Sym}^3 \CC^2\otimes \CC^2 , \;  \Sigma^{2, -1}_0 \CC^4 = \CC^2 \otimes \mathrm{Sym}^3 \CC^2 . 
\end{gather*}
For $N=3$ they are clearly
\[
\CC, \: \Sigma^{1}_0 \CC^3, \: \Sigma^{2}_0 \CC^3\, .
\]
\end{remark}

We first prove

\begin{theorem}\xlabel{tO}
Let $N\ge 4$. Suppose that $V$ is an irreducible R-representation for $\mathrm{O}_N (\CC )$. If $V$ is already generically free for
$\mathrm{O}_N (\CC )$, then $V/\mathrm{O}_N (\CC )$ is stably rational of level $N$. Otherwise, $(V\oplus \CC^N)/\mathrm{O}_N
(\CC )$ is stably rational of level $N$.
\end{theorem}

\begin{remark}\xlabel{rProblemSO}
Note that one should not expect to be able to prove the complete analogue of Theorem \ref{tO} for the group $\mathrm{SO}_N (\CC
)$. This is due to the existence of irreducible R-representations
\[
\Sigma^{(c,c,\dots , c , \pm c)}_0 \CC^{2n}
\]
for $\mathrm{SO}_{2n} (\CC )$ which restrict to irreducible $\mathrm{O}_{2n-1}(\CC )$-representations. We will deal with the
$\mathrm{SO}_N (\CC )$-case later.
\end{remark}

As usual we will use induction to prove Theorem \ref{tO}. The induction base will be furnished by $\mathrm{O}_3 (\CC ) = \ZZ /2\ZZ \times \mathrm{PSL}_2 (\CC )$ (both for the special and full orthogonal groups). We will do this first. To begin with, we note the following Lemma.

\begin{lemma}\xlabel{lO}
Suppose that the irreducible R-representation $V$ for $\mathrm{O}_N (\CC )$ decomposes for $\mathrm{O}_{N-1} (\CC )$ in such a way that it contains two distinct summands
$V'$ and $W$ which are irreducible for $\mathrm{O}_{N-1} (\CC )$ and 
where $V'$ is an R-representation for $\mathrm{O}_{N-1} (\CC )$ and $\dim W \ge 2 (N-1)$.\\
Then if Theorem \ref{tO} holds for
$V'$, it holds for $V$.
\end{lemma}

\begin{proof}
This follows directly from the Severi-Brauer method (cf. the proof of Proposition \ref{pInductionStepSL}), which we have
already used several times earlier; more precisely: by replacing $V$ by $V\oplus \CC^N$ if necessary we can suppose that $V$ is generically free.  By Remark \ref{rSO}, we know that then also $V \oplus \CC$ is generically free for $\mathrm{O}_{N-1} (\CC )$. In $V'$ only the center of $\mathrm{O}_{N-1} (\CC )$ may act trivially. If the action of $\mathrm{O}_{N-1}(\CC )$ on $V'\oplus W$ is then generically free, the Severi Brauer methods applies, as we know already stable rationality of level $N-1$ for $V'\oplus \CC^{N-1}$ by hypothesis (Theorem \ref{tO} holds for $V'$).\\
If the center of $\mathrm{O}_{N-1}(\CC )$ acts still trivially in $V' \oplus W$, we know that there will be an irreducible summand of $V\oplus \CC$, different from $W$, $S$ say, such that $V'\oplus S$ is generically free for $\mathrm{O}_{N-1} (\CC )$.  By the no-name lemma and the hypothesis on $V'$, we then know stable rationality of level $2(N-1)$ for $(V'\oplus S)/\mathrm{O}_{N-1}(\CC )$. So $(V'\oplus S \oplus W)/\mathrm{O}_{N-1}(\CC )$ is rational, hence also $( V\oplus \CC )/\mathrm{O}_{N-1}(\CC )$.
\end{proof}

\begin{proposition}\xlabel{pInductionBaseO}
Theorem \ref{tO} is true for $N=4$. 
\end{proposition}

\begin{proof}
An irreducible representation $V$ of $\mathrm{O}_{4} (\CC )$ has as underlying $\mathrm{SO}_4 (\CC )$-representation 
\[
\Sigma^{\lambda_1, \lambda_2 }_0 \CC^4 + \Sigma^{\lambda_1, -\lambda_2 }_0 \CC^4 , \; \lambda_2 > 0
\]
or $\Sigma^{\lambda_1, 0}_0 \CC^4$. We know the stable rationality of level $1$ of generically free $\mathrm{O}_3 (\CC ) = \mathrm{PSL}_2 (\CC ) \times \ZZ /2\ZZ$-representations: if $\tilde{V}$ is one such, add $\CC$ where $\ZZ /2\ZZ$ acts via a nontrivial character. Then $(\tilde{V}\oplus \CC )/ (\mathrm{PSL}_2 (\CC ) \times \ZZ /2\ZZ ) \simeq \tilde{V}/\mathrm{PSL}_2 (\CC ) \times \CC$, and we know that $\tilde{V}/\mathrm{PSL}_2 (\CC )$ is always rational. Thus Theorem \ref{tO} is true for $N=4$ as we can reduce its statement to a stable rationality result of level $1$ for $\mathrm{O}_3 (\CC )$. 
\end{proof}

We now describe the induction step. We have to show that we can reduce to the case $N=4$. 

\begin{proposition}\xlabel{pInductionStepO}
Let $N\ge 5$. Then the hypotheses of Lemma \ref{lO} are satisfied for every R-representation of $\mathrm{O}_N (\CC )$ which is
not one of the following:
\[
(\Sigma^{1, 1, 1}_0 \CC^{7})^{\pm } \, , \; ( \Sigma^{2,2}_0 \CC^5)^{\pmÊ}, \; ( \Sigma^{2,1}_0 \CC^5)^{\pm } \, .
\]
\end{proposition}

\begin{proof}
We start with a preliminary remark: for $m \ge 4$ there are no irreducible representations of $\mathrm{O}_{m} (\CC )$ of dimension $< 2m$ except $\CC$ and $\CC^m$. In fact, for $m\ge 5$ we have
\[
\dim \mathrm{O}_m (\CC ) = \frac{m (m-1) }{2} \ge 2m \, .
\]
So for $m\ge 5$, every irreducible representation of $\mathrm{O}_{m} (\CC )$ of dimension $< 2m$ must be an $E$-representation, hence be composed of $E$-representations for $\mathrm{SO}_{m} (\CC )$. Going through the list in Remark \ref{rStabilizerSO} shows that this leaves only $\CC$ or $\CC^m$.\\
For $m=4$, we know that the irreducible $\mathrm{O}_4 (\CC )$-representations are of the form
\[
\mathrm{Sym}^a \CC^2 \otimes \mathrm{Sym}^a \CC^2,  \quad \mathrm{or} \quad \mathrm{Sym}^a \CC^2 \otimes \mathrm{Sym}^b \CC^2 + \mathrm{Sym}^b \CC^2 \otimes \mathrm{Sym}^a \CC^2, \; a\neq b
\]
of dimensions $(a+1)^2$ resp. $2(a+1)(b+1)$. This is strictly smaller than $8$ only if $a=0, 1$ in the first case,  and for the pairs $(0, 0)$, $(0,1)$, $(0, 2)$ or vice versa in the second case. These are all $E$-representations.  In particular note also that an irreducible representation of $\mathrm{O}_m (\CC )$ of the form
\[
(\Sigma^{\lambda_1, \dots , \lambda_m }_0 \CC^m )^{\emptyset } \; \mathrm{or } \; (\Sigma^{\lambda_1, \dots , \lambda_m }_0 \CC^m )^{\pm }
\]
will have dimension $\ge 2m$ for $m\ge 4$ if $ \lambda_1 + \dots + \lambda_{m-1} + |\lambda_m| \ge 2$ except if it is $( \Sigma^{1,1}_0 \CC^4 )^{\emptyset }$. 

\

Suppose now that $V$ is an $R$-representation for $\mathrm{O}_N (\CC )$, and $N\ge5$. We will distinguish two cases:  
\begin{enumerate}
\item[(1)]
 $V= (\Sigma^{c, \dots , c, c}_0 \CC^N)^{\emptyset }$, $N$ is even, 
\item[(2)] 
not the case (1). 
\end{enumerate}

We recall that the representation $V$ is of the form $(\Sigma_0^{\lambda_1, \dots , \lambda_n} \CC^{2n+1})^{\pm }$ for odd $N=2n+1$, or it is $(\Sigma_0^{\lambda_1, \dots , \lambda_{n-1}, 0} \CC^{2n})^{\pm }$ resp. $(\Sigma_0^{\lambda_1, \dots , \lambda_n} \CC^{2n})^{\emptyset}$ for even $N=2n$.

\

\textbf{Case 1}: Here
\[
V = \Sigma^{c,\dots , c, c}_0 \CC^{2n} + \Sigma^{c,\dots , c, -c}_0 \CC^{2n}, 
\]
and we can satisfy the hypotheses of Lemma \ref{lO} by putting $V' = ( \Sigma^{c, \dots , c}_0 \CC^{2n-1})^{\pm }$ and $W = ( \Sigma^{c, \dots , c}_0 \CC^{2n-1})^{\pm }$ ($V'$ coming from the one irreducible $\mathrm{SO}_{2n} (\CC )$-summand of $V$, $W$ coming from the other one). 

\

\textbf{Case 2}: For
$N=2n+1$, $\lambda_n\neq 0$, the representation
$V$ will contain the
$\mathrm{O}_{N-1} (\CC )$-subrepresentation 
\[
V' := (\Sigma^{(\lambda_1, \dots , \lambda_n )}_0 \CC^{2n} + \Sigma^{(\lambda_1, \dots , -\lambda_n )}_0 \CC^{2n} )
\]
which has the properties required except for $V =(\Lambda^{1,1,1}_0 \CC^7 )^{\pm }$, $(\Sigma^{2,2}_0 \CC^5)^{\pm}$ or $(\Sigma^{2,1}_0 \CC^5)^{\pm}$; these are precisely the exceptions listed in the statement of the proposition.\\
If $\lambda_n =0$, $V$ will contain the
$\mathrm{O}_{N-1} (\CC )$-subrepresentation 
\[
V' := (\Sigma^{(\lambda_1, \dots , \lambda_{n-1}, 0 )}_0 ( \CC^{2n}))^{\pm } \, .
\]
which is also an $R$-representation because $V$ is one.\\
If $N=2n$, then $V$ will contain the
$\mathrm{O}_{2n-1}(\CC )$ subrepresentation
\[
V' := ( \Sigma^{(\lambda_1, \dots , \lambda_{n-1} )}_0 \CC^{2n-1} )^{\pm }
\]  
which also is an R-representation because otherwise, $2n-1$ being odd and $\ge 5$, it would be one of $\CC$, $\CC^{2n-1}$, $\mathrm{Sym}^2_0 \CC^{2n-1}$ or $\Lambda^2 \CC^{2n-1}$ . But then also $V$ would be an $E$-representation of one of these types.\\
Thus we have found in each case an $R$-representation $V'$ as required for Lemma \ref{lO}. Thus it remains to find a representation $W$ as required for that Lemma as well. Note that the total number of boxes in the Young diagrams parametrizing the $\mathrm{SO}_{N-1} (\CC )$ irreducible summands of the representations $V'$ above will be $\ge 3$ in all cases. For if there were only $\le 2$ boxes we would have $E$-representations. Now because we are not in case 1, so $V\neq (\Sigma^{c, \dots , c, c}_0 \CC^N)^{\emptyset }$, $N$ is even, we get that the decomposition of $V$ into $\mathrm{O}_{N-1}(\CC )$-irreducibles will contain -besides $V'$- also an irreducible $\mathrm{O}_{N-1}( \CC )$-representation $W$ parametrized by a Young diagram with total number of boxes precisely one less than in the Young diagram of $V'$. Now by the considerations at the beginning of the proof, we will also have $\dim W \ge 2(N-1)$ unless if $(W = \Sigma^{1,1}_0 \CC^4)^{\emptyset}$. In that case, $V'$ must have one box more, and looking at the above construction, we see that we could only end up with this $W$ if $V = (\Sigma^{2,1}_0 (\CC^5))^{\pm}$. This is one of the exceptions listed in the statement of the Proposition.
\end{proof}

\begin{proposition}\xlabel{pFurtherExceptionsO}
Let $V$ be one of the $\mathrm{O}_5 (\CC )$-representations $(\Sigma^{2,1}_0 \CC^5)^{\pmÊ}$ or $( \Sigma^{2,2}_0 \CC^5)^{\pm}$. Then Theorem \ref{tO} holds for $V$.
\end{proposition}

\begin{proof}
Remark  that -if restricted to $\mathrm{O}_4 (\CC )$- both representations contain a summand 
\[
\Sigma^{2,1}_0 \CC^4 + \Sigma^{2, -1}_0\CC^4 \, .
\]
We claim that this is already generically free for $\mathrm{O}_4 (\CC )$. Indeed, according to \cite{APopov78}, one of the summands taken alone as $\mathrm{SO}_4 (\CC )$-representation has a nontrivial stabilizer in general position $( \ZZ /2\ZZ )^2$. This is easy to see as a pencil of binary cubics in $\Sigma^{2,1}_0 \CC^4 = \mathrm{Sym}^3 \CC^2 \otimes \CC^2$ can be seen as a covering map $\PP^1 \to \PP^1$ of degree $3$, ramified in four points by the Riemann-Hurwitz formula. These four points are stabilized inside $\mathrm{PGL}_2 (\CC )$ by a Klein four group. At the same time we see from this description that two pencils of binary cubics have a trivial stabilizer in general.\\
Now we know stable rationality of level $10$ for generically free quotients of the group $\mathrm{O}_4 (\CC )$ already by the argument in \cite{Bogo86}, see also the introduction to this paper where this argument is recalled. However, the codimension of $ \Sigma^{2,1}_0 \CC^4 + \Sigma^{2, -1}_0\CC^4$ in $\Sigma^{2,1}_0 \CC^5$ resp. $\dim \Sigma^{2,2}_0 \CC^5$ is bigger than $10$, 
so we conclude by the no-name lemma.
\end{proof}

\begin{proposition}\xlabel{pExteriorPowerO}
Let $V$ be the $\mathrm{O}_7 (\CC )$-representation
\[
V = (\Sigma^{1, 1, 1}_0 \CC^{7})^{\pm }\, .
\]
Then Theorem \ref{tO} holds for $V$.
\end{proposition}

\begin{proof}
Let us consider $V= \Lambda^3 \CC^7 = (\Sigma^{1, 1, 1}_0 \CC^{7})^{-}$. We have to prove stable rationality of level $7$ for
$V/\mathrm{O}_7 (\CC )$. This reduces to a stable rationality result of level $1$ for the quotient of the representation
\[
\Lambda^3 \CC^6 + \mathfrak{so}_6 (\CC )
\]
for the natural action by the group $\mathrm{O}_6 (\CC ) $. The action of $\mathrm{O}_6 (\CC )$ on $\mathfrak{so}_6 (\CC )$ has a
linear $(\mathrm{O}_6 (\CC ), \; \ZZ/2\ZZ \ltimes N(T))$-section given by a Cartan subalgebra $\mathfrak{t}\simeq \CC^3 \subset
\mathfrak{so}_6 (\CC )$. Here $N(T)$ is of course the normalizer of a maximal torus in $\mathrm{SO}_6 (\CC )$.

\

The residual representation
\[
\Lambda^3 \CC^6 
\]
decomposes with respect to $\ZZ/2\ZZ \ltimes N(T)$ into two summands $V_s$ and $V_b$ (corresponding to two Weyl group orbits of weights). Here
\begin{gather*}
V_b = \bigoplus_{\chi= ( \chi_1, \chi_2, \chi_3 )} V_{\chi }
\end{gather*}
where the sum is over all six weights $\chi$ of $\Lambda^3 \CC^6 $ which have precisely one nonzero entry $1$ or $-1$. Thus a basis of weight vectors in $W$ is given by $e_i \wedge e_{-i} \wedge e_{\pm j }$, $i, j \in \{ 1,2,3\}$, $i \neq j$. The other summand is
\begin{gather*}
V_s = \bigoplus_{\chi'= ( \chi'_1, \chi'_2, \chi'_3 )} V_{\chi' }
\end{gather*}
where $\chi'$ denotes one of the eight weights $(\pm 1 , \pm 1, \pm 1)$ with weight vectors $e_{\pm 1}\wedge e_{\pm 2 } \wedge e_{\pm 3}$. It follows that the $V_{\chi }$ have dimension $2$ whereas $V_{\chi '}$ has dimension $1$. Thus $\dim V_b =12$, $\dim V_s = 8$ (subscripts indicating \emph{b}ig resp. \emph{s}mall). The action of $\ZZ /2\ZZ \ltimes N(T)$ on $V_b$ is generically free. To see this, remark that the torus $T$ acts clearly generically freely in $V_b$; and the Weyl group $\mathfrak{S}_3 \ltimes (\ZZ /2 \ZZ )^3$ of $\mathrm{O}_6 (\CC )$ acts on the $T$-orbits in $V_b$ generically freely. 

\

Thus it is sufficient to produce a generically
free representation $R$ for $\ZZ/2\ZZ \ltimes N(T)$, of dimension $\le 12$ and with a rational quotient; then the stable rationality
of level $1$ of $\Lambda^3\CC^6 + \mathfrak{so}_6 (\CC )$ will be proven. In fact
\[
\dim V_s + \dim \mathfrak{t} +1 = 12\, .
\]

\

What we will take is $\CC^6 + \CC^6$. Here $\CC^6$ is the restriction of the standard $\mathrm{O}_6 (\CC )$-representation $\CC^6$ to $\ZZ / 2 \ZZ \ltimes N(T)$. Now $\ZZ /2 \ZZ \ltimes N (T)$ is nothing but the Weyl group $\mathfrak{S}_3 \ltimes (\ZZ /2 \ZZ )^3$ of $\mathrm{O}_6 (\CC )$ extended by the torus $T\simeq (\CC^{\ast })^3$. Here, with respect to our standard basis
\[
e_1, \: e_2, \: e_3, \quad e_{-1}, \: e_{-2} , \: e_{-3}
\]
an element $\sigma \in \mathfrak{S}_3$ acts via permuting $1, 2, 3$ (and subjects $-1,-2,-3$ to the same permutation), and the standard generators of $( \ZZ /2\ZZ )^3$ act by interchanging $e_i$ and $e_{-i}$. An element $t= \mathrm{diag}(t_1, t_2, t_3) \in T$ acts by scalings $e_i \mapsto t_i e_i$, $e_{-i} \mapsto t_i^{-1} e_{-i}$, $i=1, 2 , 3$.  We denote by $x_{\pm 1}, x_{\pm 2}, x_{\pm 3}$ coordinates in the space $\CC^6$ with respect to the basis $e_{\pm i}$. It is clear that $x_1x_{-1}$, $x_2x_{-2}$, $x_{3}x_{-3}$ give coordinates on the $T$-orbits on which the Weyl group acts. The ineffectivity kernel is $H=( \ZZ /2 \ZZ )^3$ which is hence equal to the stabilizer in general position in $\CC^6$. Certainly, $(\CC^6)^H$ has dimension $3$ and is spanned by $e_i + e_{-i}$, $i=1, 2, 3$. The normalizer of $H$ in the extension of the Weyl group by the torus $T$ is equal to
\[
N(H)  = \mathfrak{S}_3 \ltimes ( (\ZZ /2 \ZZ )^3 \times \langle \mathrm{diag} (\pm 1, \pm 1 , \pm 1) \rangle )\, .
\]
Here $\langle \mathrm{diag} (\pm 1, \pm 1 , \pm 1) \rangle =  ( \ZZ /2 \ZZ )^3$ is the subgroup of $T$ consisting of matrices with entries $t_1$, $t_1$, $t_3$ either $+1$ or $-1$. In fact, if the equation
\[
(\sigma, \: t ) \cdot (\iota , \: 1) \cdot (\sigma , \: t )^{-1} = (\sigma \iota \sigma^{-1} , \: t (\sigma \iota \sigma^{-1}) (t^{-1})) \in H 
\]
holds for some fixed $\sigma$ in the Weyl group, a fixed $t \in T$, and all $\iota \in H$, then we always get the equations $t_1^2 = t_2^2 = t_3^2 =1$ as equivalent to the condition that $(\sigma \iota \sigma^{-1}) (t^{-1}) =1\in T$. Thus it remains to prove rationality of 
\[
\left( \CC^3 + \CC^6  \right ) / N(H) \, .
\]
The action of $N(H)$ on $\CC^6$ is generically free. The six lines spanned by 
\[
(e_1 \pm e_{-1}) , \; (e_2 \pm e_{-2}), \;  (e_3 \pm e_{-3})
\]
are the $( \ZZ /2\ZZ )^6$-eigenspaces. The quotient $\CC^6 / (( \ZZ /2\ZZ )^6)$ has an action of the group $\mathfrak{S}_3$ which is birationally equivalent to a linear action. In fact denoting by $y_{\pm 1} , \: y_{\pm 2} ,\: y_{\pm 3}$ coordinates with respect to the basis given by the preceding vectors spanning the eigenspaces, we see that 
\[
(y_1, \dots , y_{-3} ) \quad \mathrm{and} \quad (y_1', \dots , y_{-3}')
\]
are in the same $( \ZZ /2\ZZ )^6$-orbit if and only if $y_i^2 = (y_i' )^2$ for all $i$. So the map $\CC^6 \to \CC^6$ given by squaring the coordinates $y_i$ induces a birational map of $\CC^6 / (\ZZ /2\ZZ )^6$ and $\CC^6$ which is $\mathfrak{S}_3$-equivariant for the action of $\mathfrak{S}_3$ on $z_i = y_i^2$ given by permuting the $z_1, \: z_2, \: z_3$ and $z_{-1}, \: z_{-2}, \: z_{-3}$. Thus $\CC^6$ splits as twice the standard representation $\CC^3$ of $\mathfrak{S}_3$ whose quotient is rational.
\end{proof}

Note that with the proof of Proposition \ref{pExteriorPowerO}, we have completed the proof of Theorem \ref{tO} as follows from Propositions \ref{pInductionStepO}, \ref{pFurtherExceptionsO} and \ref{pInductionBaseO}. 

\

For the group $\mathrm{SO}_N (\CC )$ we can only prove slightly less, cf. Remark \ref{rProblemSO}.

\begin{theorem}\xlabel{tSO}
Let $N\ge 4$. Suppose that $V$ is an irreducible R-representation for $\mathrm{SO}_N (\CC )$. 
\begin{itemize}
\item[(1)]
If $N =2n$ and $V = \Sigma^{c, \dots , c, \pm c }_0 \CC^{2n}$, then the following holds: 
if $V$ is already generically free for
$\mathrm{SO}_N (\CC )$, then $V/\mathrm{SO}_N (\CC )$ is stably rational of level $2N$. Otherwise, $(V\oplus
\CC^N)/\mathrm{SO}_N (\CC )$ is stably rational of level $2N$.
\item[(2)]
In all other cases, we have: if $V$ is generically free for
$\mathrm{SO}_N (\CC )$, then $V/\mathrm{SO}_N (\CC )$ is stably rational of level $N$. Otherwise, $(V\oplus
\CC^N)/\mathrm{SO}_N (\CC )$ is stably rational of level $N$.
\end{itemize}
\end{theorem}

\begin{remark}\xlabel{rlSO}
Lemma \ref{lO} continues to hold if we take an $\mathrm{SO}_N (\CC )$-representation $V$ instead of an $\mathrm{O}_N (\CC
)$-representation if we replace Theorem \ref{tO} by Theorem \ref{tSO} in the conclusion, i.e. say: if Theorem \ref{tO} holds for $V'$, then Theorem \ref{tSO} holds for $V$. 
\end{remark}

The proof will follow from Theorem \ref{tO}. We first prove a Proposition.

\begin{proposition}\xlabel{pExceptionalCasesSO}
Part (1) of Theorem \ref{tSO} holds for all $N=2n \ge 6$.
\end{proposition} 

\begin{proof}
If we restrict a representation of the form $V = \Sigma^{c, \dots , c, \pm c }_0 \CC^{2n}$ to $\mathrm{O}_{2n-2}(\CC )$ we
obtain summands of the form
\[
( \Sigma^{c, \dots , c, c'} \CC^{2n-2})^{\emptyset }, \quad 1\le c' \le c \, .
\]
The representation
\[
( \Sigma^{c, \dots , c, c} \CC^{2n-2})^{\emptyset }
\] 
is an R-representation of $\mathrm{O}_{2n-2} (\CC )$ unless $2n=6$ and $c=2$ or $2n=8$ and $c=1$. Moreover, the representation
\[
( \Sigma^{c, \dots , c, c-1} \CC^{2n-2})^{\emptyset } \: \mathrm{if} \: c> 1 \: \mathrm{or} \: ( \Sigma^{1, \dots , 1, 0}
\CC^{2n-2})^{\pm } \: \mathrm{if} \: c=1 \, ,
\]
has dimension at least $2(N-2)$ unless $2n =6$ and $c=2$ or $2n=8$ and $c=1$. So we can conclude by the Severi-Brauer
method, but have to treat the two exeptional cases separately. The case $c=2$, $2n=6$ presents no problem since it follows from Proposition \ref{pFurtherExceptionsO}. So it remains to consider the case $c=1$, $2n=8$. This
follows from Proposition \ref{pExteriorPowerO}.
\end{proof}

We will now complete the proof of Theorem \ref{tSO}.

\begin{proof}
(of Theorem \ref{tSO}) Actually we have already done almost all the work. Let us assume $N\ge 5$ first. By Proposition \ref{pExceptionalCasesSO} we can assume  that $\Sigma_0^{\lambda } \CC^N$ is a representation of $\mathrm{SO}_{N} (\CC )$ which is not of the form $\Sigma_0^{c, \dots , c, \pm c } \CC^{2n}$. The proof of Proposition \ref{pInductionStepO} then shows that  for $N\ge 5$ the reduction of $V$ to the subgroup $\mathrm{O}_{N-1} (\CC )$ will contain an irreducible $R$-representation $V'$ for $\mathrm{O}_{N-1} (\CC )$ and an irreducible $\mathrm{O}_{N-1}$-representation $W$ of dimension $\ge 2(N-1)$ \emph{except if}  $V = \Sigma_0^{c, \dots , c, \pm c } \CC^{2n}$ \emph{or} $V = \Lambda^3 \CC^7$, $=\Sigma^{2,1}_0 \CC^5$ or $=\Sigma^{2,2}_0 \CC^5$. Hence we can conclude by the Severi-Brauer method and Theorem \ref{tO} for the full orthogonal group, which we have already established.\\
The proofs for the cases $V = \Lambda^3 \CC^7$, $\Sigma^{2,1}_0 \CC^5$ or $\Sigma^{2,2}_0 \CC^5$ are the same as in Propositions  \ref{pExteriorPowerO} and \ref{pFurtherExceptionsO}.\\
So we have obtained the full statement of Theorem \ref{tSO} now for $N\ge 5$. It remains to consider the case $N=4$.  The proof is entirely similar to the proof of Proposition \ref{pInductionBaseO}. Note that $\Sigma_0^{c, \pm c} \CC^4$ are $E$-representations. 
\end{proof}

\section{The case of the group $G_2$}\xlabel{sG2}

We recall some facts about $G_2$ and its representation theory. A \emph{composition algebra} over a field $k$ is
a (not necessarily associative)
$k$-algebra $C$ with an identity $1$ and a nondegenerate quadratic form $N$ on $C$ such that $N(xy) =N(x)N(y)$ (all $x$ and $y$ in
$C$). $N$ is called the norm form, and the associated bilinear form will be denoted by $\langle \cdot , \cdot \rangle$. In any
composition algebra $C$ one may introduce an involutive anti-automorphism $C\to C$, $x\mapsto \bar{x}$, called the
\emph{conjugation}, defined by $\bar{x} := -s_{1} (x)$ where $s_1$ is the reflection in the subspace $1^{\perp}$ defined by $\langle
\cdot , \cdot \rangle$. Then $x\bar{x} =\bar{x}x =N(x) \cdot 1$.\\
Now let $\mathbb{H} := \mathrm{Mat}_{2\times 2} (\CC )$ be the \emph{split quaternion algebra} (over $\CC$) of two by two complex
matrices which is a composition algebra with norm form the determinant. If
\[
x = \left( \begin{array}{cc} a & b \\ c & d \end{array}\right)
\]
then consequently
\[
\bar{x} = \left( \begin{array}{cc} d & -b \\ -c & a \end{array}\right)\, .
\]
The \emph{split octonion algebra} $\mathbb{O}$ is the composition algebra which is constructed from $\mathbb{H}$ by the process of
\emph{doubling}: as vector space $\mathbb{O} = \mathbb{H} \oplus \mathbb{H}$ and the product resp. norm form are given by
\[
(x, \: y) (u, \: v) := (xu +\bar{v}y, \: vx +y\bar{u}), \quad x, y,u,v \in \mathbb{H}
\]
resp.
\[
N((x, \: y)) := \det (x) -\det (y) , \quad x, \: y \in \mathbb{H}\, .
\]
We view $\mathbb{H}$ as embedded into $\mathbb{O}$ as the first factor, and $1_{\mathbb{O}} = (1_{\mathbb{H}}, \: 0)$. The group of
automorphisms $\mathrm{Aut} (\mathbb{O})$ of the composition algebra $\mathbb{O}$ (which means unital, norm-preserving algebra
automorphisms) is the simple group of type $G_2$, $G_2\subset O(N, \: \mathbb{O})$. Its dimension is $14$. It has two fundamental
representations: the adjoint $V_{14}=\mathfrak{g}_2$ and the traceless octonions $V_7 =1_{\mathbb{O}}^{\perp } \subset
\mathbb{O}$. If we choose a maximal torus in $G_2$, denote by $\alpha_1$ resp. $\alpha_2$ the long resp. short simple root in a
basis for the root system and let
$\omega_1$ and $\omega_2$ be the corresponding fundamental weights, then
$V_{\omega_1} = V_{14}$, $V_{\omega_2}= V_7$. The irreducible $G_2$-modules are thus indexed by highest weights $\lambda = m_1
\omega_1 + m_2 \omega_2$ where $m_1, \: m_2 \ge 0$ are integers.

\begin{lemma}\xlabel{lStabilizerG2}
The stabilizer of a generic point in $V_7$ inside $G_2$ is $H= \mathrm{SL}_3 (\CC )$ and its normalizer $N(H)$ is
\[
N(H) = \mathbb{Z}/2\mathbb{Z} \ltimes \mathrm{SL}_3 (\CC ) \, .
\]
The subspace of $H$-invariants in $V_7$ is one-dimensional.
\end{lemma} 

\begin{proof}
The stabilizers in general position in $V_7$ and $\mathbb{O}$ of course coincide. We use a model for $\mathbb{O}$ in terms of
vector matrices (see
\cite{Sp-Veld}, section 1.8): $\mathbb{O}$ can be realized as the algebra of vector matrices
\[
X= \left( \begin{array}{cc} a & v \\ w & b \end{array} \right) , \quad a, b\in\CC, \; v\in \CC^3, w\in \CC^3
\]
with multiplication
\[
\left( \begin{array}{cc} a_1 & v_1 \\ w_1 & b_1 \end{array} \right) \left( \begin{array}{cc} a_2 & v_2 \\ w_2 & b_2
\end{array} \right)
=
\left( \begin{array}{cc} a_1a_2+\langle v_1, w_2 \rangle  & b_2v_1+a_1v_2-w_1\times w_2 \\ 
a_2w_1 +b_1w_2 +v_1\times v_2 & b_1b_2+\langle w_1 , v_2 \rangle
\end{array}
\right)
\]
with conjugation and norm given by
\begin{gather*}
\bar{X} = \left( \begin{array}{cc} b & -v \\ -w & a \end{array} \right) , \quad N(X) = ab - \langle v, \: w \rangle \,
.
\end{gather*}
Here $\langle , \rangle$ denotes the standard scalar product on $\CC^3$, and the vector product
 is defined by $\langle v_1\times v_2, \: v_3\rangle = \det (v_1, v_2, v_3)$.  
An element of $G_2$ which preserves the diagonal matrices pointwise must leave the two copies of $\CC^3$ stable: suppose for
example that
$g$ maps 
\[
\left( \begin{array}{cc} 0 & 0 \\ v^{\ast} & 0 \end{array} \right)
\] 
to
\[
\left( \begin{array}{cc} 0 & w \\ w^{\ast} & 0 \end{array} \right), \: w\neq 0,
\]
then the compatibility with the product would force
\begin{gather*}
g \left( \left( \begin{array}{cc} 0 & 0 \\ v^{\ast} & 0 \end{array} \right) \cdot \left( \begin{array}{cc} a & 0 \\ 0 & b
\end{array} \right)  \right) = g\left(  \left( \begin{array}{cc} 0 & 0 \\ av^{\ast} & 0 \end{array} \right)\right) =\left(
\begin{array}{cc} 0 & aw \\ aw^{\ast} & 0 \end{array} \right)\\
=\left( \begin{array}{cc} 0 & w \\ w^{\ast} & 0 \end{array} \right)\cdot \left( \begin{array}{cc} a & 0 \\ 0 & b \end{array}
\right) = \left( \begin{array}{cc} 0 & bw \\ aw^{\ast} & 0 \end{array} \right)
\end{gather*}
which is impossible.
Thus if an element in $G_2$ preserves the subspace of diagonal matrices pointwise it must be of the form
\[
\left( \begin{array}{cc} a & v \\ v^{\ast } & b \end{array} \right) \mapsto \left( \begin{array}{cc} a & l(v) \\ m (v^{\ast })
& b
\end{array} \right)
\]
$l\, :\, \CC^3\to \CC^3$ and $m \, :\, \CC^3 \to \CC^3$ some invertible linear transformations which must satisfy $\langle l (v),
\: m(v^{\ast} )\rangle = \langle v, v^{\ast } \rangle $ for all $v\in \CC^3$, $v^{\ast }\in \CC^3$. It follows that $m
=(l^t)^{-1}$. On the other hand we also want $l$ to be compatible with the product in $\mathbb{O}$ which means in particular
\[
l(v_1)\times l(v_2) = (l^{t})^{-1} (v_1\times v_2)
\]
whence $\det (l(v_1), \: l(v_2), \: l(v_3)) = \det (v_1, \: v_2, \: v_3)$ and $l\in\mathrm{SL}_3 (\CC )$. This proves
$H=\mathrm{SL}_3 (\CC )$.\\
Note that the map
\[
\left( \begin{array}{cc} a & v \\ w & b \end{array} \right) \mapsto \left( \begin{array}{cc} b & w \\ v & a \end{array} \right)
\]
is in $G_2$, and conjugation by it induces on $\mathrm{SL}_3 (\CC )$ the outer automorphism $g\mapsto (g^t)^{-1}$. Hence the
statement on $N(H)$ follows. That the subspace of $H$-invariants in $V_7$ is one-dimensional is already clear from what was said
above.
\end{proof}

Thus we have for a generically free $G_2$ representation $V_{\lambda }$
\[
V_{\lambda } /G_2 + \CC^7 = (V_{\lambda } + V_7 )/G_2 = (V_{\lambda } + \CC )/( \ZZ/2\ZZ \ltimes \mathrm{SL}_3 (\CC ) )\, .
\]
We propose to prove

\begin{theorem}\xlabel{tG2}
If $V_{\lambda }$ is a generically free representation for the group $G_2$, then $V_{\lambda }/ ( \ZZ/2\ZZ \ltimes \mathrm{SL}_3
(\CC ) )$ is rational, hence $V_7/G_2$ is stably rational of level $7$.
\end{theorem}

The branching law if we restrict a $G_2$-representation $V_{\lambda }$ to $\mathrm{SL}_3 (\CC )$ is known (\cite{Per}) and can
be described as follows: a Gelfand-Zetlin pattern associated to $\lambda = m_1 \omega_1 + m_2 \omega_2$ is a diagram $\mu
(a,b,c)$ 
\begin{gather*}
\begin{array}{ccccc}
m_1 +m_2 &   &  m_2 &   &  0\\
         & a &      & b &   \\
         &   &  c   &   &   
\end{array}
\end{gather*}
where $a$, $b$ and $c$ are nonnegative integers with $m_1 +m_2 \ge a  \ge m_2 \ge b \ge 0$ and $a\ge c\ge b$, so that the second
row of the diagram interlaces the first row, and the third row interlaces the second row. Then the decomposition as
$\mathrm{SL}_3 (\CC )$-module is
\[
V_{\lambda } = \bigoplus_{\mu (a,b,c)} \Sigma^{(m_1 + c, \: a-m_2 + b , \: 0 )} \CC^3 \, ,
\]
the sum running over all Gelfand-Zetlin patterns $\mu (a,b,c )$ associated to $\lambda$.

\

We need 

\begin{proposition}\xlabel{pSL3Z2}
Let $V(a, b)$ be a generically free representation of $\mathrm{SL}_3 (\CC )$. Then 
\[
(V(a, b) + V(b , a) ) / N(H)
\]
is stably rational of level $6$.
\end{proposition}

\begin{proof}
Since
\begin{gather*}
(V(a, b) + V(b , a) + \CC^3 + (\CC^3)^{\vee } ) / N(H) \\= (V(a, b) + V(b, a))/ (\mathrm{SL}_2 (\CC )\rtimes (\ZZ /2\ZZ ))
\end{gather*}
we study the decomposition of $V(a, b)$ as $\mathrm{SL}_2 (\CC )$-module. Note that generically free $\mathrm{SL}_2 (\CC
)\rtimes (\ZZ /2\ZZ )$-modules are certainly stably rational of level $4$ (add $\CC^2 + \CC^2 = \CC^2 + (\CC^2)^{\vee }$ and
reduce to $\ZZ/2\ZZ$). So it suffices to note that in the decomposition of $V(a, b)$ there will be the
$\mathrm{SL}_2 (\CC )$-modules $V(a+b)$ and $V(a+b-1)$: one of them is generically free for $\mathrm{SL}_2 (\CC )$ for $a+b\ge
5$, and under this assumption the other one will have dimension $\ge 4$ whence the result.\\
Thus the remaining cases to consider separately are
\[
V(1, 2) + V(2, 1) , \quad V(4, 0) + V(0, 4), \quad V(3, 1) + V(1, 3) \, .
\] 
In each of these cases we get a copy $\CC^2 + \CC^2$ in the decomposition and reduce to $\ZZ/2\ZZ$ as above.
\end{proof}

\begin{proof} (of Theorem \ref{tG2})
From the branching law we see that $V_{\lambda }$ contains in the $\ZZ/2\ZZ \ltimes \mathrm{SL}_3 (\CC
)$-decomposition ($(a,b,c) = (m_1 +m_2 , m_2, m_1 +m_2) $)
\begin{gather*}
\Sigma^{( 2m_1 + m_2, \: m_1 + m_2, \: 0 )} \CC^3 + \Sigma^{( 2m_1 + m_2, \: m_1 + m_2, \: 0 )} (\CC^3)^{\vee } \\= V(m_1,
m_1+m_2) + V(m_1+m_2 , m_1)
\end{gather*}
in the case where $m_2\neq 0$ and $3$ does not divide $m_2$. If $m_2\neq 0$ and $3$ divides $m_2$ we may take instead ($(a,b,c) =
(m_1 +m_2 , m_2-1, m_1 +m_2) $)
\begin{gather*}
\Sigma^{( 2m_1 + m_2, \: m_1 + m_2-1, \: 0 )} \CC^3 + \Sigma^{( 2m_1 + m_2, \: m_1 + m_2-1, \: 0 )} (\CC^3)^{\vee } \\= V(m_1+1,
m_1+m_2-1) + V(m_1+m_2-1 , m_1+1)\, .
\end{gather*}
Note that for $a-b$ not divisible by $3$ there are only the following E-representations for $\mathrm{SL}_3 (\CC )$:
\[
V(0, 1), \; V(1, 0), \; V(2, 0), \; V(0, 2)\, .
\]
We may encounter one of these cases only if 
\begin{quote}
$m_1=0$, $m_2\le 2$  in the case $m_2 \neq 0$.  
\end{quote}
This leaves only the case $V_{\lambda } = V_{2\omega_2}$. In this case we have the decomposition
\[
V_{2\omega_2} = \mathrm{Sym}^2 \CC^3 + \mathrm{Sym}^2 (\CC^3)^{^\vee } + \mathrm{Ad}_0 \CC^3 + \CC^3 + (\CC^3)^{\vee } + \CC \, .
\]
We can reduce to a rationality question for $\mathrm{SL}_2 (\CC ) \rtimes \ZZ /2\ZZ$ because this is the stabilizer of a generic
point in $\CC^3 + (\CC^3)^{\vee }$. But with respect to that group $\mathrm{Sym}^2 \CC^3 + \mathrm{Sym}^2 (\CC^3)^{^\vee }$ has a
summand $\CC^2 + (\CC^2)^{\vee }$ so that we can reduce further to $\ZZ/2\ZZ$ and are done.

\

Otherwise one of the above two representations will be generically free for $N(H)= \ZZ/2\ZZ \ltimes \mathrm{SL}_3 (\CC )$ and will
be composed of generically free summands for $\mathrm{SL}_3 (\CC )$. Moreover, it is clear that both $V(m_1,
m_1+m_2) + V(m_1+m_2 , m_1)$ and $V(m_1+1,
m_1+m_2-1) + V(m_1+m_2-1 , m_1+1)$ have dimension $\ge 6$ in this case. So the question remains when the two will be distinct
representations in the decomposition. This will always be the case \emph{unless} $m_2=1$. Thus in this special case we take as
above
\[
V(m_1, m_1 +1) + V(m_1 +1, m_1)
\]
as a subrepresentation which is composed of generically free $\mathrm{SL}_3 (\CC )$-representations (as $m_1\ge 1$ now because
$V_{\lambda }$ is generically free). But then we get also as a subrepresentation (for $(a,b,c)= (m_1, 1, m_1)$)
\[
\Sigma^{(2m_1, m_1, 0)} \CC^3 = V(m_1, m_1)
\]
which has certainly dimension $\ge 6$ (as the adjoint has already dimension $8$).

\

We turn to the case where $m_2 =0$. Then certainly $m_1\ge 2$ and $V_{\lambda }$ contains ($(a,b,c) = (m_1, 0, m_1-1)
$)
\begin{gather*}
\Sigma^{( 2m_1 -1, \: m_1 , \: 0 )} \CC^3 + \Sigma^{( 2m_1 -1, \: m_1 , \: 0 )} (\CC^3)^{\vee } \\= V(m_1-1,
m_1) + V(m_1 , m_1-1)\, .
\end{gather*}
It is composed of generically free summands for $\mathrm{SL}_3 (\CC )$. We likewise find the representation
\begin{gather*}
\Sigma^{( 2m_1 , \: m_1 , \: 0 )} \CC^3 = V(m_1, m_1) 
\end{gather*}
in the decomposition ($(a,b,c) = (m_1, 0, m_1)
$), and since $m_1\ge 2$, its dimension is $\ge 6$.
\end{proof}

\appendix
\section{Generic Stabilizers for $\mathrm{O}_{2n} (\CC )$}\xlabel{AppendixOrthogonalGroup}
The purpose of this Appendix is to prove the following technical result which is needed in Section \ref{sOrthogonalGroups}; we retain the notation regarding the orthogonal groups introduced at the beginning of that section. 

\begin{theorem}\xlabel{tOrthogonalsgp}
Let $n\ge 2$ be an integer. Let $V=\Sigma^{\lambda }_0 \CC^{2n}$ be an irreducible representation of $\mathrm{SO}_{2n} (\CC )$ with $\lambda_n =0$. Suppose the ineffectivity kernel of the action on $V$ coincides with the stabilizer in general position in $\mathrm{SO}_{2n} (\CC )$ (i.e., in the terminology of the paper, $V$ is an R-representation of $\mathrm{SO}_{2n}(\CC )$). Then for each of the two extensions $V^+$ and $V^-$ of the representation $V$ to an $\mathrm{O}_{2n} (\CC )$-representation this property is preserved: the stabilizer in general position for $V^+$ or $V^-$ in $\mathrm{O}_{2n} (\CC )$ coincides with the ineffectivity kernel of the action of $\mathrm{O}_{2n} (\CC )$ on $V^+$ or $V^-$.
\end{theorem}

The proof of this theorem will be broken up into several intermediate results. One has to show that a generic point in $V^{\pm }$ is not invariant under some $h\in\mathrm{O}_{2n} (\CC )$. Since $V$ is an R-representation of $\mathrm{SO}_{2n} (\CC )$ this implies that $h^2 = 1$ or $h^2 = -1$  (and $\det h = -1$). We claim that the case $h^2 = -1$ cannot occur. In this case, $h$ would be semisimple (diagonalizable), as it is of finite order, with eigenvalues $+i$ and $-i$. The respective eigenspaces for these eigenvalues are isotropic subspaces of $V^{\pm }$: for $v, w$ in the $i$-eigenspace, e.g., one has 
\[
\langle v, \: w \rangle  = \langle h(v), \: h(w)\rangle   = \langle i\cdot v , \: i\cdot w \rangle = - \langle v, \: w\rangle \, .
\]
As dimensions of isotropic subspaces cannot exceed $n$, but the dimensions of the $\pm i$ eigenspaces add up to $2n$, both the eigenvalue $i$ and the eigenvalue $-i$ of $h$ occur with multiplicity $n$ whence $\det h =1$, a contradiction.\\
Thus we can assume $h^2 = 1$ and $\det h = -1$. Let $T$ be a maximal torus of $\mathrm{O}_{2n} (\CC )$ with Lie algebra the Cartan algebra $\mathfrak{t}$, and let $W = W (\mathrm{O}_{2n}) = N(T) /T$ be the Weyl group of $\mathrm{O}_{2n}(\CC )$. We let $ \epsilon_i$, $i=1, \dots , n$ be the standard coordinate functions on $\mathfrak{t}$: thus we have $\epsilon_i (\mathrm{diag} (t_1, \dots , t_{2n})) = t_i$. We identify $\Lambda = \ZZ \cdot \epsilon_1 \oplus \dots \oplus \ZZ \cdot \epsilon_n$ with the weight lattice of $\mathrm{SO}_{2n} (\CC )$. Recall that we have then

\ 

\begin{center}
\begin{tabular}{l  l}
Roots: & $\pm \epsilon_i \pm \epsilon_j, \; 1 \le i\neq j \le n$\\
Simple roots: & $\epsilon_i -\epsilon_{i+1}, \; 1\le i \le n-1, \quad \epsilon_{n-1} + \epsilon_n$\\
Fundamental weights (of $\mathfrak{so}_{2n}$): & $\omega_i = \epsilon_1 + \dots + \epsilon_i, \; 1\le i \le n-2$, \\
         &  $\omega_{n-1}= \frac{1}{2} (\epsilon_1 + \dots + \epsilon_{n-1} - \epsilon_n )$, \\
          &  $ \omega_{n}= \frac{1}{2} (\epsilon_1 + \dots + \epsilon_{n-1} + \epsilon_n )$
\end{tabular}
\end{center}

\

and the Weyl group $W= W(\mathrm{O}_{2n} (\CC )) = ( \ZZ / 2\ZZ )^n \rtimes \mathfrak{S}_n$ where $(\ZZ /2\ZZ )^n$ acts on weights $a_1\epsilon_1 + \dots + a_n \epsilon_n =: (a_1, \dots , a_n)$ via sign changes of entries and $\mathfrak{S}_n$ by permutations; note that  the Weyl group of $\mathrm{SO}_{2n} (\CC )$ is strictly smaller, namely acts only by even numbers of sign changes on weights, but as $g_0$ (cf. Section \ref{sOrthogonalGroups} above) induces an element of $W (\mathrm{O}_{2n})$ and $g_0 \cdot (a_1, \dots , a_n ) = (a_1, \dots , -a_n )$, the larger Weyl group contains all sign changes.

\

The element $h$ is characterized -up to conjugacy in $\mathrm{O}_{2n} (\CC )$- by its numbers of $+1$ and $-1$ eigenvalues. Since the symbols $V^{\pm}$ are already in use,  we will write $R$ for any one of $V^{\pm }$ in the sequel, and denote the eigenspaces of $h$ corresponding to $+1$ resp. $-1$ by $R_{+}$ resp. $R_{-}$ (or sometimes by $R_+ (h)$, $R_{-} (h)$ if we want to indicate $h$).

\begin{lemma}\xlabel{lOrthogonalEstimate}
Assume that both eigenspaces $R_{+}$ and $R_{-}$ of $h$ have codimension $> \dim \mathrm{O}_{2n} (\CC ) - \dim Z_h$ where $Z_h$ is the centralizer of $h$ in $\mathrm{O}_{2n} (\CC )$. Then the stabilizer of a generic point in $R$ does not contain any element conjugate to $h$ or $-h$. 
\end{lemma}

\begin{proof}
Note that the space $R_{+} (-h)$ is -depending on the parity of $n$- equal to either $R_{+}(h)$ or $R_{-} (h)$. It is thus sufficient to show that the conclusion of the Lemma holds for $h$ if $\mathrm{codim} R_{+} (h) > \dim \mathrm{O}_{2n} (\CC ) - \dim Z_h$. Suppose that the stabilizer in general position in $R$ would contain an element conjugate to $h$. Then $\mathrm{O}_{2n} (\CC ) \cdot R_+ (h)$ would be dense in $R$, but this contradicts the previous assumption on the codimension of $R_+ (h)$.
\end{proof}

\begin{remark}\xlabel{rSizeCentralizer}
\begin{itemize}
\item[(1)]
By the preceding lemma we may henceforth assume -replacing $h$ by $-h$ if necessary- that the number of $(-1)$-eigenvalues of $h$ is odd and $\le n$. We may then also take $h$ to be equal to (a representative of) an element in the normal subgroup $( \ZZ / 2\ZZ )^n \subset W$.
\item[(2)]
One has $\dim \mathrm{O}_{N} (\CC ) = N(N-1)/2$, hence, if $h$ has $k$ eigenvalues $-1$, $1 \le k \le n$, then  $Z_h = \mathrm{O}_k (\CC ) \times \mathrm{O}_{2n -k} (\CC )$ and 
\[
\dim (\mathrm{O}_{2n} (\CC) / Z_h) = (2n-k ) k \, . 
\]
The right hand side is 
\begin{align*}
\dim (\mathrm{O}_{2n} (\CC) / Z_h)&  =2n-1, \quad k=1, \\
   & = n^2, \quad\quad\quad k= n, \\
   & \le n^2\quad\quad\quad 1\le k  \le n \, .
\end{align*}
\end{itemize}
\end{remark}

The representation $R = V^{\pm}$ decomposes into weight spaces for the action of $T$ which in turn can be grouped into Weyl group orbits of the form
\[
R_{[\chi]} = \bigoplus_{w \in W / \mathrm{Stab}_W (\chi )} R_{w\cdot \chi }
\]
where $\chi$ is a weight of $R$ and $\mathrm{Stab}_W (\chi )$ the stabilizer of the weight $\chi$ in $W$ so that the orbit $[\chi ] = W / \mathrm{Stab}_W (\chi )$. The element $h$ acts on $R_{[\chi ]}$. 

\begin{definition}
We will call a weight $\mu$ in the orbit $[\chi ] = W\cdot \chi$ an $h$\emph{-regular weight} if $ \mu \neq h\cdot \mu $. The subset of weights $\mu \in [\chi ]$ with $h\cdot \mu = \mu$ will be denoted by $[\chi]^h$.
\end{definition}

\begin{lemma}\xlabel{lWeylOrbit}
Let $R_{[\chi ],\;  +}$ resp. $R_{[\chi ],\;  -}$ be the $+1$ resp. the $-1$-eigenspace of $h$ in $R_{[\chi ]}$. Then
\begin{gather*}
\mathrm{codim} \left(  R_{[\chi ],\;  \pm }  \right)  \ge \frac{1}{2} \left(  \mid [\chi ] \mid - \mid [ \chi ]^h \mid \right) \cdot \dim V_{\chi } \, .
\end{gather*}
\end{lemma}

\begin{proof}
Indeed, if $\mu\in [\chi ]$ is $h$-regular, then the trace of $h$ on $V_{\mu } + V_{h\cdot \mu }$ is $0$. 
\end{proof}

Thus everything comes down to producing for the representation $R$ a weight $\chi$ with $R_{\chi } \neq (0)$ such that $[\chi ]$ contains sufficiently many $h$-regular elements. We start with some easy observations. 

\begin{lemma}\xlabel{lWeightsNonzero}
If $n\ge 5$ and  $R$ contains a weight space $R_{\chi }$ such that for 
\[
\chi = (\chi_1, \dots , \chi_n )
\]
we have: (1) $\chi_i \neq 0$ for all $i$, and (2) there are $i, \: j $ with $| \chi_{i_0} | \neq |\chi_{j_0} |$, then we get the estimate 
\[
\mathrm{codim} \left(  R_{[\chi ],\;  \pm }  \right)  > n^2 \, ,
\]
so the conclusion of Lemma \ref{lOrthogonalEstimate} holds.
\end{lemma}

\begin{proof}
In this case, all elements in the orbit $[\chi ]$ are $h$-regular, and $(1/2) | [\chi ] | \ge (1/2)\cdot 2^{n+1} > n^2$ for $n\ge 5$. In fact, every such weight $\chi$ gives $2^n$ pairwise different weights by acting on it via sign changes of entries, and in each of these weights we can still interchange the entries whose absolute values are equal to $\chi_{i_0}$ resp. $\chi_{j_0}$. These then give $2^{n}\cdot 2 = 2^{n+1}$ pairwise distinct weights in the Weyl group orbit $[\chi ]$.
\end{proof}

We have to consider two special types of weights: weights with $|\chi_1 |   = \dots = |\chi_n | $, which we will temporarily call \emph{absolutely constant weights}, or weights which contain a zero entry $\chi_i =0$, which we will call \emph{degenerate weights}. 

\

\begin{remark}\xlabel{rPhiSaturatedness}
Except $W$-invariance, we will use one further property of the set of weights of an irreducible representation $R$ of a semisimple Lie group: this set is $\Phi$\emph{-saturated}, $\Phi$ being the set of roots $\alpha \in \mathfrak{t}^{\ast }$, cf. \cite{G-W}, p. 155 ff. This means that whenever $\chi$ is a weight of $R$, and $\alpha \in\Phi$ is a root, then
\begin{gather*}
\chi  - k \alpha \; \mathrm{is}\;  \mathrm{a}\;  \mathrm{weight}\; \mathrm{ of } \; R \; \mathrm{for}\; \mathrm{all}\;   k \; \mathrm{between}\; 0 \; \mathrm{and}\;  \langle \chi , \alpha^{\vee } \rangle
\end{gather*}
where $\alpha^{\vee}\in \mathfrak{t}$ is the coroot to $\alpha$; in terms of a $W$-invariant inner product  $(\cdot , \cdot )$ on $\mathfrak{t}^{\ast }$, one can identify $\alpha^{\vee }$ with $(2\alpha )/(\alpha, \alpha )$. 
\end{remark}

Applying the previous remark we get that every representation $R$ which has only absolutely constant or degenerate weights, has a weight $\chi$ with $\chi_n =0$: in fact, an absolutely constant weight $\chi$ which is not zero is in the same $W$-orbit as a weight $\chi'$ with $\chi_1' = \chi_2' = \dots = \chi_n' > 0$, and by Remark \ref{rPhiSaturatedness}, $\chi' - \chi'_n \cdot (\epsilon_{n-1} + \epsilon_n)$ is also a weight of $R$ and has last entry equal to zero. 

\begin{remark}\xlabel{rEstimateOrbitSizes}
We assume here now without loss of generality that $R$ has a weight $\chi$ which we write as 
\begin{gather*}
\chi = ( x^{0}_1,  \dots , x^{0}_{n_0}, \dots , \: x^{j}_1, \dots ,  x^{j}_{n_j}, \dots , x^{m}_1, \dots , \: x^{m}_{n_m} )
\end{gather*}
where $x^{j}_k = x^{j}_l$ for all $j$ and all $1 \le k, \: l \le n_j$, we have $x^j_{k} \neq x^h_l$ for $j \neq h$, and $x^0_1=\dots = x^0_{n_0} =0$. Moreover we can assume that all the $x^i_j$ are nonnegative. Thus in short, we just group entries with the same value in $\chi$ together and put the zeroes in front. We can also assume that $h\in (\ZZ/2 \ZZ )^n \subset W$ acts via sign changes in the first $k$ entries of a weight.\\
Let us first assume $k=1$. Then by elementary combinatorial considerations we have for the Weyl group orbit size 
\begin{gather*}
| [\chi ] | = 2^{n -n_0} \cdot \frac{n!}{n_0 ! \cdot   n_1 ! \cdot \ldots \cdot  n_m !}
\end{gather*}
whereas the weights in $[ \chi ]$ which are not regular are in number
\begin{gather*}
| [\chi ]^h | = 2^{n - n_0} \cdot \frac{(n-1)! } {(n_0 -1)! \cdot n_1 ! \cdot \ldots \cdot n_m!} 
\end{gather*} 
Hence the percentage of weights in the orbit $[\chi ]$ which is not $h$-regular for $h$ with one eigenvalue $-1$ is
\[
\frac{| [\chi ]^h | }{| [\chi ] | } = \frac{n_0} {n} \, .
\]
Finally we remark that the ratio $| [\chi ]^h |  :  | [\chi ] |$ will be smaller than this if $h$ has $k> 1$ eigenvalues $-1$ because every weight which is invariant under sign changes in the first $k$ entries is  also invariant under a sign change in only the first entry.
\end{remark}

\begin{lemma}\xlabel{lWeylOrbitEstimate}
In the set-up of Remark \ref{rEstimateOrbitSizes}, the Weyl group orbit cardinality
\begin{gather*}
| [\chi ] | = 2^{n -n_0} \cdot \frac{n!}{n_0 ! \cdot   n_1 ! \cdot \ldots \cdot  n_m !}\, ,
\end{gather*}
viewed as a function of the partition $n_0 + n_1 + \dots + n_m = n$, attains its maximum, for a fixed number of zero entries $n_0$, if all the $n_i$, $i=1, \dots , m$ are equal to $1$, and is then equal to
$2^{n-n_0} n! / n_0 !$.\\
It attains its minimum, again for fixed $n_0$, if $m=1$ and then its value is $2^{n-n_0} { n \choose n_0 }$.\\
The function $f(n_0) = 2^{n-n_0} {n \choose n_0 }$ increases monotonously from $f(0) = 2^n$  to its maximum at $n_0 = \left[ (n-2)/3 \right] +1$, and afterwards decreases monotonously to its value $f(n) = 1$.\\
Moreover, if the integer $z\le n$ satisfies
\[
z > \frac{n-2}{3}+1  \; \mathrm{and} \; { n \choose z } < 2^z \, ,
\]
then, among all those weights with less than or equal to $z$ zero entries,  the overall minimum of the Weyl group orbit cardinality is attained for a weight with precisely $z$ zeroes and all remaining $n-z$ entries of equal absolute values, and the value of this minimum is $2^{n - z }{n \choose z}$. 
\end{lemma}

\begin{proof}
The first two assertions are clear if we adopt the following viewpoint: think of the absolute values of the entries in the weights in $[\chi ]$ as colours, and think of their signs as charges; fix some number of zero entries $n_0$. Suppose $\chi'$ is the weight obtained from $\chi$ by changing all nonzero colours into the colour of the $x^{1}_j$, say. This gives a surjection of the Weyl group orbit of $\chi$ unto the Weyl group orbit of $\chi'$ whence the assertion about the minimum of the Weyl group orbit cardinality for fixed $n_0$.\\
Suppose we pass from the weight $\chi$ to another weight $\chi''$ by choosing for the $n_1$ equally coloured elements $x^{1}_j$ other $n_1$ pairwise distinct colours, and similarly for the $x^2_k$, up to the $x^{m}_l$, so that all the entries of $\chi''$ then have pairwise distinct colors. Then we get a surjection from the Weyl group orbit of $\chi''$ to the Weyl group orbit of $\chi$, resubstituting the old colours. Hence the assertion about the maximum of the Weyl group orbit cardinality for fixed $n_0$.  

\

For the last assertion, note that
\begin{gather*}
2^{n-n_0 -1} {n \choose n_0 +1 } = 2^{n-n_0} \cdot \frac{1}{2} \cdot \frac{n!}{(n_0 +1)! (n-n_0-1)!}\\
 = 2^{n-n_0} { n \choose n_0} \cdot \frac{n-n_0}{2(n_0 +1)}\, .
\end{gather*}
Now 
\begin{gather*}
 \frac{n-n_0}{2(n_0 +1)} \ge 1  \iff  n_0 \le \frac{n -2}{3} , \\
 \frac{n-n_0}{2(n_0 +1)} < 1  \iff  n_0 > \frac{n -2}{3}\, .
\end{gather*}
This means that the function $f(n_0) = 2^{n-n_0} { n \choose n_0Ê}$ increases monotonously from its value $f(0) = 2^n$ to its maximum (located at $\left[ (n-2)/3 \right] +1$), and then decreases to $f(n) = 1$. In particular, if the integer $z$ satisfies 
\[
z > \frac{n-2}{3}+1  \; \mathrm{and} \; 2^{n-z} { n \choose z } < 2^n 
\]
then we get that among weights with less than or equal to $z$ zero entries the overall minimum of the Weyl group orbit cardinality is attained for a weight with precisely $z$ zeroes and all remaining $n-z$ entries of equal absolute values, and the value of this minimum is $2^{n - z }{n \choose z}$. 
\end{proof}

\begin{lemma}\xlabel{lBetterEstimate}
Let $n \ge 10$ and suppose that $R$ contains a weight space belonging to a weight $\chi$ which can be written as in Remark \ref{rEstimateOrbitSizes}, and assume that $\chi$ contains $\ge 3$ nonzero entries (which, however, need not have distinct absolute values). Then the conclusion of Lemma \ref{lOrthogonalEstimate} holds. 
\end{lemma}

\begin{proof}
We apply the last part of Lemma \ref{lWeylOrbitEstimate} with $z = n-3$. We have that the overall minimum of the Weyl group orbit cardinality, taken over weights with at least three nonzero entries, is equal to
\[
8 \cdot \frac{n (n-1) (n-2)}{6}
\]
provided that the inequalities
\[
n-3 > \frac{n-2}{3}+1  \; \mathrm{and} \; { n \choose 3 } < 2^{n-3} 
\]
hold. The first one is valid for $n> 5$, but the second one is more stringent and holds for $n \ge 10$. From Remark \ref{rEstimateOrbitSizes} we get that at least a portion of $3/n$ of the weights in $[\chi ]$ is $h$-regular. Hence
\begin{gather*}
\frac{1}{2} \left( | [\chi ] |  - | [\chi ]^h |  \right) \ge \frac{1}{2} \cdot 8 \cdot \frac{3}{n} \cdot \frac{n (n-1) (n-2)}{6} = 2 (n-1)(n-2) >  n^2
\end{gather*}
for $n\ge 10$ (actually for $n\ge 6$, but we were lead to assume $n\ge 10$ before) so that we conclude by Lemmas \ref{lWeylOrbit} and \ref{lWeightsNonzero}.
\end{proof}

Now we can state

\begin{proposition}\xlabel{pCharacterizationFewNonzero}
Theorem \ref{tOrthogonalsgp} holds for $n\ge 10$. More precisely, the only representations $R$ of $\mathrm{O}_{2n} (\CC )$ all of whose weights have $\le 2$ nonzero entries have a nontrivial s.g.p. already as $\mathrm{SO}_{2n} (\CC )$ representations. This last assertion holds also for $n\ge 3$.
\end{proposition}

\begin{proof}
Suppose $R$ is an $\mathrm{O}_{2n} (\CC )$ representation all of whose weights have $\le 2$ nonzero entries so that in particular this holds for the highest weight
\[
\lambda = (\lambda_1, \dots , \lambda_n ) \, .
\]
We use Remark \ref{rPhiSaturatedness}, the $\Phi$-saturatedness of the set of weights in $R$, again, together with the knowledge of the roots $\pm e_i \pm e_j$, $1 \le i\neq j \le n$. In fact, $\lambda$ can only be of the form
\[
(a, \: b, \: 0 , \dots , 0 ), \quad a,\: b \in \mathbb{Z}, \quad a\ge b\, .
\]
Suppose that $b \neq 0$, and one of $a$ or $b$ is $\ge 2$. Then we get that -depending on which of $a$ or $b$ is $\ge 2$- that $\lambda - \epsilon_1 + \epsilon_3$ or $\lambda -\epsilon_2 + \epsilon_3$ is also a weight of $R$ with $\ge 3$ nonzero entries. Thus if $b \neq 0$ this leaves only $\lambda$ equal to $(1,\: 1, \: 0, \dots , 0)$ which gives the $E$-representation $\Lambda^2_0 \CC^{2n}$. Now suppose that $b = 0$, but $a\ge 3$. Then $\lambda - (\epsilon_1 - \epsilon_2) - (\epsilon_1 - \epsilon_3)$ will be a weight of $R$ with at least three nonzero entries. Hence we might only have
\[
R = \mathrm{Sym}^2_0 \CC^{2n}\, , \quad   R = \CC^{2n} , \quad \mathrm{or } \quad R = \CC
\]
all of which are $E$-representations. 
\end{proof}

Let us investigate in more detail what happens for $2\le n\le 9$. 
The first thing that is special for $n\le 9$ is that, among all weights $\chi$ with at least three nonzero entries, the minimum of the Weyl orbit cardinality is not attained for a weight with precisely $3$ equal nonzero entries, but for an absolutely constant weight, and the minimum there is $2^n$. However, we can get further with slightly sharper estimates.

\begin{lemma}\xlabel{lDownToSix}
Let $6\le n \le 9$. Suppose the $\mathrm{O}_{2n} (\CC )$ representation $R$ contains a weight $\chi$ with $\ge 3$ nonzero entries. Then $[\chi ] - [\chi ]^h$ has more than $2n^2$ elements, hence by Proposition \ref{pCharacterizationFewNonzero}, we get that Theorem \ref{tOrthogonalsgp} holds for $n\ge 6$.
\end{lemma}

\begin{proof}
Suppose the weight $\chi$ contains precisely $a$ nonzero entries. By Lemma \ref{lWeylOrbitEstimate} we know that the minimum Weyl orbit cardinality among weights with precisely $n-a$ zeroes is $2^a { n \choose a }$, and by Remark \ref{rEstimateOrbitSizes} we know that at least a portion of $a/n$ of the orbit elements is $h$-regular. Hence it suffices to check that the inequality
\[
2 n^2 <  2^a { n \choose a } \frac{a}{n}Ê\quad \mathrm{or} \quad n^2 < 2^{a-1} { n-1 \choose a-1 }
\]
holds for all $6 \le n \le 9$ and all $3\le a \le n-1$. Note that we can assume $a \le n-1$  (for $a = n = 6$ the inequality would fail as $36 > 2^5$):  in fact, under the assumptions of the Lemma, $R$ will also contain a weight with at least one zero and at least three nonzero entries. Namely, the highest weight $\lambda$ has $\lambda_n =0$, so if $\lambda$ has already three nonzero entries we are done. If $\lambda$ contains only $\le 2$ nonzero entries (then necessarily $\lambda_1$ and $\lambda_2$), then the method explained in the proof of Proposition \ref{pCharacterizationFewNonzero} will produce from $\lambda$ a weight with three nonzero entries \emph{and} at least one zero entry for $n\ge 6$, unless $R$ is already an $E$-representation.
\end{proof}

Thus the cases $n=2, \: 3, \: 4, \: 5$ have still to be considered. We will exclude $n=2$ for the moment and treat it last. It will be useful to tabulate the minimum possible value for
\[
\frac{1}{2} \left( | [\chi ] |  - | [\chi ]^h |  \right)
\]
depending on the number $a$ of nonzero entries of $\chi$, which we know to be $F(n, a):= (1/2) 2^{a} (a/n) { n \choose a }$, which is attained when all the nonzero entries are equal.

\

\begin{center}
\begin{tabular}{| c || c | c | c | c | c || c |}
\hline 
  $F(n, a)$             &   $a=1$ &  $a=2$  &  $a = 3$  &  $a=4$  &  $a=5$ & $n^2$ \\  \hline\hline
 $n=5$  &    $1$     &    $8$          &    $24$             &      32         &    $16$   & $25$      \\ \hline
 $n=4$  &     $1$          &   $6$            &    $12$             &       $8$        &         & $16$      \\ \hline
 $n=3$  &     $1$         &    $4$          &       $4$          &               &        & $9$  \\ \hline
\end{tabular}
\end{center}

\

Instead of only one Weyl group orbit, we now consider several orbits, more precisely we want to try to look for weights $\chi_1$, $\chi_2, \dots $ in the representations $R$ of $\mathrm{O}_{2n} (\CC )$ for $n=3,\: 4,\: 5$, such that the corresponding Weyl group orbits are pairwise disjoint, and contain in total more than $2n^2$ elements. 

\

\textbf{The case $n=5$}. Suppose that the highest weight $\lambda$ of $R$ contains at least three nonzero entries $\lambda_1 \ge \lambda_2 \ge \lambda_3 > 0$. Note that in any case $\lambda_5 =0$ (this is the type of representation we consider). If one $\lambda_i>1$, $i=1, \dots , 3$, then by Remark \ref{rPhiSaturatedness}, $\chi= \lambda - (\epsilon_i  - \epsilon_5)$ is also a weight of $R$, and the Weyl group orbits of $\lambda$ and $\chi$ are disjoint since $\chi$ contains less zero entries than $\lambda$. Then we are done (cf. the table, $24 + 32 > 25$). 

\

Suppose now that there are only two nonzero entries $\lambda_1 \ge \lambda_2 > 0$ or only one $\lambda_1 > 0$. If in the first case either $\lambda_1$ or $\lambda_2$ is $\ge 2$ then we get also a weight with exactly three nonzero entries in $R$. Hence we are done as $8 + 24 > 25$. And $\Lambda^2_0 \CC^{10}$ is an $E$-representation. If in the second case $\lambda_1 \ge 3$, then we get that $\chi =  \lambda - (\epsilon_1 - \epsilon_2)$ and $\chi' = \chi - (\epsilon_1 - \epsilon_3)$ also belong to $R$. They are all in distinct Weyl group orbits as they have different number of zero entries, and these contain enough regular elements. For $\lambda_1 < 3$ we get $E$-representations again. 

\

So it remains to consider only the exterior powers  $\lambda = (1, 1 , 1, 0, 0)$ resp. $(1 , 1, 1, 1, 0)$. In the second case we have that also $(1, 1, 0, 0, 0)$ is a weight and we are done. In the first case we get that also $(1, 0, 0, 0, 0)$ is a weight. Unfortunately, according to the table $1 + 24$ is just too small to be strictly bigger than $25$. However note that the weight space of $(1, 0, 0, 0, 0)$ has dimension $4$ so that by the estimate in Lemma \ref{lWeylOrbit} we are done.

\

\textbf{The case $n=4$}.  Assume again first that the highest weight $\lambda$ of $R$ contains at least three nonzero entries $\lambda_1 \ge \lambda_2 \ge \lambda_3 > 0$. If one $\lambda_i>1$, $i=1, \dots , 3$, we can argue as in the case $n=5$ above since $12 + 8 > 16$. 

\

If there are at most two nonzero entries, we argue precisely as in the case $n=5$ above, so it remains to consider the exterior power $\lambda = (1,1, 1, 0)$ here. It contains also the weight $(1, 0,0,0 )$ with multiplicity $3$, but $3 + 12=15  < 16$, so our estimates are too crude here. We have to distinguish the cases that $h \in (\ZZ /2 \ZZ)^4$ has $k=1$, $2$ or $3$ eigenvalues $-1$. By Remark \ref{rSizeCentralizer} it suffices that the codimensions of the $+1$ and $-1$ eigenspaces of $h$ are bigger than $(2n-k)k$ which is  $15$ for $k=3$, $12$ for $k=2$, $7$ for $k=1$. So we just have to reconsider the case $k=3$. What is too crude in the case $k=3$ is the estimate for the portion of regular elements in an orbit coming from Remark \ref{rEstimateOrbitSizes}: for the weight $(1, 0, 0, 0)$ it is not $1/4$ (as it would be for $k=1$). In fact, there are $6$ regular elements in the orbit contributing at least an amount of $3$ to the codimensions of the $+1$ and $-1$ eigenspaces of $h$. However, $12 + 3 \cdot 3 > 15$ and we are done.

\

\textbf{The case $n=3$}. These representations $R$ have highest weights $\lambda = (\lambda_1, \: \lambda_2,\:  0)$, $\lambda_1 > \lambda_2 \ge 0$. Let us assume first that $\lambda_1 \ge 4$. Then $\chi = \lambda - (\epsilon_1 - \epsilon_3)$ and $\chi' = \lambda - 2 (\epsilon_1 - \epsilon_3)$ are both weights of $R$. Together with $\lambda$ they generate three pairwise distinct Weyl group orbits ($\chi$ and $\chi'$ have less zeroes than $\lambda$, and $\chi$ and $\chi'$ differ in the number of entries $1$ and $2$). If $\lambda_2 > 0$ we conclude as $4 + 2\cdot 4 > 9$. If $\lambda_2 = 0$, but $\lambda_1 \ge 4$, we also have the weight $\chi'' = \chi - (\epsilon_1 -\epsilon_2)$ and $3\cdot 4 + 1 > 9$.

\

Thus the only cases remaining are $\lambda_1 \le 3$. For the highest weights $(3,3,0), \: (3,2,0), \: (3,1,0)$ and $(2,2,0)$ one can easily find three weights in $R$ in distinct Weyl group orbits, and having $\ge 2$ nonzero entries, using the $\Phi$-saturatedness. Thus we are left with
\[
(3,0,0),  \; (2,1,0) 
\]
because the rest are $E$-representations of $\mathrm{SO}_6 (\CC )$. As in the case $n=4$ we will consider again subcases here, according to whether $h$ has $1$ or $3$ eigenvalues $-1$ (it has to be an odd number). We need codimensions of the $+1$ resp. $-1$ eigenspaces of $h$ strictly bigger than $(6-k)k$ which is $5$ for $k=1$ and $9$ for $k=3$. Thus the case $k=1$ is no problem (we can take the Weyl orbits corresponding to $(3,0,0), \: (2,1,0), \: (1,1,1)$ in the first case, and the Weyl orbits of $(2,1,0), \: (1,1,1)$ in the second case). For $k=3$, again, the portion of $h$-regular elements in an orbit is higher than the portion we used for our standard estimates: in fact, every weight which is not identically $0$ is $h$-regular in this case. We need $> 2\cdot 9 =18$ regular weights, and these can be found in the orbit of $(2,1,0)$ in both cases.

\

Finally we consider

\

\textbf{The case $n=2$}. These representations $R$ have highest weights $\lambda = (\lambda_1, \:  0)$ and we may assume $\lambda_1 \ge 3$ to exclude obvious $E$-representations. Here, since $h\in\mathrm{O}_4 (\CC ) \backslash \mathrm{SO}_4 (\CC )$, we need only consider the case that $h$ has one eigenvalue $-1$, so $k=1$. We have to produce Weyl group orbits containing in total strictly more than six $h$-regular elements. This is easy: each of the above representations contains a weight with both entries nonzero and of different absolute values, and the orbit contains $8$ elements.

\end{document}